\documentclass[11pt,reqno]{amsart}
\usepackage{amsfonts,amstext}
\usepackage{amsmath,amstext,amssymb,amscd,mathrsfs,amsthm,latexsym}
\usepackage{enumerate}
\usepackage[dvipsnames,usenames]{xcolor}
\allowdisplaybreaks[4]

\newtheorem{theorem}{Theorem}[section]

\newtheorem{lemma}[theorem]{Lemma}
\newtheorem{corollary}[theorem]{Corollary}
\theoremstyle{definition}
\newtheorem{definition}[theorem]{Definition}
\theoremstyle{assumption}
\newtheorem{assumption}[theorem]{Assumption}

\theoremstyle{remark}
\newtheorem{remark}[theorem]{Remark}
\numberwithin{equation}{section}

\newenvironment{proofth1}{\medskip\noindent{\bf Proof of Theorem \ref{thmg}.}\enspace}{\hfill \qed  \medskip}
\newenvironment{proofth2}{\medskip\noindent{\bf Proof of Theorem \ref{thm2}:}\enspace}{\hfill \qed  \medskip}
\newenvironment{proofcor1}{\medskip\noindent{\bf Proof of Corollary \ref{cor1}:}\enspace}{\hfill \qed  \medskip}
\newenvironment{proofth210}{\medskip\noindent{\bf Proof of Theorem \ref{thm2.10}.}\enspace}{\hfill \qed  \medskip}

\usepackage[allbordercolors={1 1 1}]{hyperref}
\usepackage[disable]{todonotes}

\begin{document}

\author[M.  Liao]{Menglan Liao}
\address{School of Mathematics, Hohai University, Nanjing, 210098, China.}
 \email{liaoml@hhu.edu.cn}

\author[B. Feng]{Baowei Feng}
\address{School of Mathematics, Southwestern University of Finance and Economics, Chengdu 611130,  China. }
 \email{bwfeng@swufe.edu.cn}

\thanks{This work was supported by the National Natural Science Foundation of China (No. 12401290), the Natural Science Foundation of Jiangsu Province (No. BK20230946), the Fundamental Research Funds for the Central Universities (No. B240201090) and the Natural Science
Foundation of Sichuan Province (No. 2026NSFSC0003). }

\title[A piezoelectric beam model]{A piezoelectric beam model with nonlinear  dampings and supercritical sources}

%\date{February 14, 2025}
\begin{abstract} This paper aims to investigate a three-dimensional fully magnetic effected  piezoelectric beam model with strong sources and nonlinear interior dampings. By employing nonlinear semigroups and the theory of monotone operators, the existence of local weak solutions is established. By the potential well method, we obtain the  global existence of potential well solutions. Decay rates of the total energy are obtained in terms of the behavior of the damping terms.  The main advantage in this work is that the stabilization estimate \emph{does not generate lower-order terms}, and in addition we \emph{remove some strong conditions in previous results} to obtain a weaker energy decay. Finally, when the initial total energy is negative, positive but small, respectively, the blow-up results for weak solutions if the source terms are stronger than damping terms are obtained according to the differential inequality technique. Moreover, if interior dampings are linear, a blow-up result with arbitrarily high initial energy is established by the concavity method and an upper bound for the blow-up time is also derived. All results are independent of any relation among  the model coefficients.
\end{abstract}

\maketitle

\section{Introduction}\label{S1}

Piezoelectric materials are multi-functional smart materials and
have been extensively applied in industrial fields, for instance,
ultrasonic welders, micro-sensors, inchworm robots, wearable
human-machine interfaces, based on the remarkable physical property,
that is,  transform successfully electrical energy to mechanics, and
vice versa. The related investigation of mathematical models based
on piezoelectric materials is of great importance.

In modeling piezoelectric systems,  mechanical, electrical, and
magnetic are three major effects and their interrelations need to be
considered.  Due to Maxwell's equations, a single piezoelectric beam
is modeled based on the electro-static assumption or quasi-static
assumption, which neglects the dynamic interactions of
electromagnetism.  In many studies, magnetic effects have been
ignored and only mechanical and electrical effects have been
considered. Mechanical effects are typically modeled using
Kirchhoff, Euler-Bernoulli, or Mindelin-Timoshenko small
displacement assumptions, see \cite{BCW,Han,Sim,Yang}.
 There are three main methods to include electrical and magnetic effects: electrostatic, quasi-static and fully dynamic, see \cite{Tier}. In general, unlike static electricity, the fully dynamic electromagnetic effect has a significant impact on controllability  and
stabilizability.

 A variety of  mathematical models  of piezoelectric structures mainly take mechanical and electrical effects into consideration but  directly neglect magnetic effects. However, the coupling effect of piezoelectric/piezoelectric composite materials can even be one hundred times greater than that of single-phase magneto electric materials in some cases.  In 2013,
Morris and \"{O}zer \cite{M2013}  first proposed the theory of
piezoelectric systems with full magnetic effects by virtue of the
classical variational approach, and derived a PDE dynamic model
described by coupling the elastic equation and charge equation
through piezoelectric constants
\begin{align}
\label{0add1.1}
\begin{cases}
\rho v_{tt}(x,t)-\alpha v_{xx}(x,t)+\gamma\beta p_{xx}(x,t)=F_v(x,t), &(x,t)\in(0,L)\times\mathbb{R}^+,\\
\mu p_{tt}(x,t)-\beta  p_{xx}(x,t)+\gamma\beta v_{xx}(x,t)=F_p(x,t),
&(x,t)\in(0,L)\times\mathbb{R}^+,
\end{cases}
\end{align}
where $v(x, t)$ and $p(x, t)$ denote the transverse displacement of
the beam and the total load of the electric displacement along the
transverse direction  respectively. The coefficients have the
following explanation:
 \begin{equation*}
\begin{split}
&\rho:~\text{the mass density per unit volume};\ \alpha:~\text{elastic stiffness};\ \gamma:~\text{piezoelectric coefficient},\\
&\mu:~\text{magnetic permeability};\ \ \beta:~\text{impermeability
coefficient of the beam},
\end{split}
\end{equation*}
 and all of them are positive constants.  Moreover,
\[\alpha_1=\alpha-\gamma^2\beta>0.\]
They considered the following  boundary and initial  conditions
\begin{align}\label{1.3}
\begin{cases}
v(0,t)=p(0,t)=0, &t\in\mathbb{R}^+,\\
\alpha v_x(L,t)-\gamma\beta p_x(L,t)=F_{b_1}(t), &t\in\mathbb{R}^+,\\
\beta p_x(L,t)-\gamma\beta v_x(L,t)=F_{b_2}(t),  &t\in \mathbb{R}^+,\\
v(x,0)=v_0(x),~v_t(x,0)=v_1(x), &x\in (0,L),\\
~p(x,0)=p_0(x),~p_t(x,0)=p_1(x), &x\in (0,L),
\end{cases}
\end{align}
where
 $F_v(x,t)$ and $F_p(x,t)$ are distributed
mechanical and magnetic controllers, and $F_{b_1}(t)$, $F_{b_2}(t)$
are boundary strain and voltage controllers. For
$F_v(x,t)=F_p(x,t)=0,~F_{b_1}(t)=0,~F_{b_2}(t)=-\frac{V(t)}{h}$,
with $V(t)$ means the voltage happened at the electrodes, for a beam
of length $L$ and thickness $h$, Morris and \"{O}zer \cite{M2013}
proved the well-posedness, and showed that the closed-loop system is
strongly stable in the energy space for a dense set of system
parameters by using only an electrical feedback controller (the
current flowing through the electrodes).  Furthermore, the same
authors \cite{M2014} realized an exponentially stable closed-loop
system for a set of system parameters of zero Lebesgue measure.
\"{O}zer \cite{OZ1} considered  voltage-actuated piezoelectric beams
with magnetic effects and obtained some  stabilization and exact
observability results. For
$F_v(x,t)=F_p(x,t)=0,~F_{b_1}(t)=-\xi_1\frac{v_t(L,t)}{h},~F_{b_2}(t)=-\xi_2\frac{p_t(L,t)}{h}$
with $\xi_1$ and $\xi_2$ are positive constants, Ramos, Freitas, et
al. \cite{R2019} obtained that the exponential stability of the
system is independent of any relation between the coefficients by
exploiting terms of feedback at the boundary,  and consequently
proved their equivalence with the exact observability at the
boundary. For $F_v(x,t)=-\delta
v_t(x,t),~F_p(x,t)=0,~F_{b_1}(t)=F_{b_2}(t)=0$, Ramos, Gon\c{c}alves
and Corr\^ea Neto \cite{R2018} proved that the dissipation produced
by damping $\delta v_t$ is sufficiently strong to stabilize
exponentially the system \eqref{0add1.1} by using energy method
whatever the physical parameters of the model. For $F_v(x,t)=-\delta
v_t(x,t)-f(v),~F_p(x,t)=0,~F_{b_1}(t)=g_1(v_t(L,t)),~F_{b_2}(t)=g_2(p_t(L,t))$,
i.e. the fully-dynamic piezoelectric beam equations with nonlinear
state feedback and nonlinear external sources,  Freitas, \"{O}zer
and Ramos \cite{FO2022} proved that solutions of fully-dynamic
piezoelectric beam equations converge to ones of the electrostatic
(quasi-static) equations as the magnetic permeability coefficient
$\mu\to 0$. Feng and \"{O}zer \cite{FO} studied the
boundary-controlled fully-dynamic piezoelectric beams with various
distributed and boundary delays and proved exponential energy decay
if the delays are small.
 For $F_v(x,t)=-\int_{-\infty}^t g(t-s)v_{xx}(x,s)ds,~F_p(x,t)=0,~F_{b_1}(t)=F_{b_2}(t)=0$,  Dos Santos,  Fortes and
Cardoso \cite{D2022}  proved that the past history term acting on
the longitudinal motion equation is enough to give the exponential
decay of the semigroup associated with the system, independent of
any relation among  the model coefficients. Feng and \"{O}zer
\cite{FO1} established a general energy decay for piezoelectric
beams with long-range memory effects in the boundary.  Zhang, Xu and
Han \cite{Z2022} considered the longtime asymptotic behavior of a
class of fully magnetic effected nonlinear multi-dimensional
piezoelectric beam with viscoelastic infinite memory is considered.
By means of the semigroup theories and Banach fixed-point theorem,
they first proved the well-posedness of the  nonlinear coupled PDEs'
system. And then, the corresponding coupled linear system was
indirectly stabilized exponentially by only one viscoelastic
infinite memory term based on frequency-domain analysis. Moreover,
the exponential decay of the solution to the nonlinear coupled PDEs'
system was established by the energy estimation method. The same
authors also investigated the impact of the presence of only one
friction-type infinite memory term on the stability of the
piezoelectric coupled plate in \cite{Z2023}, that is
$F_v(x,t)=-\xi\int_{-\infty}^t
g(t-s)v_t(x,s)ds,~F_p(x,t)=0,~F_{b_1}(t)=F_{b_2}(t)=0$.  The authors
\cite{Z2023} proved that this system can be indirectly stabilized
polynomially by only one friction-type infinite memory term located
on one of these strongly coupled PDEs based on frequency domain
analysis. In addition,  the optimality of this decay rate was also
further verified by detailed spectral analysis for the system
operator. The long time behavior of the damped piezoelectric system
\eqref{0add1.1} has been investigated extensively by many scholars.
It is beyond this paper to give a comprehensive review, the
interested readers can refer to
\cite{A2023,A2022,A20241,A2024,F2021,F20212,F2022,H2024,K2023,L2025,L20231,OZ2,R2022}
and the references therein.

In this paper, we study the influence of the relationship between
the  sources and damping terms on the behaviors of solutions of the
following   fully magnetic effected  piezoelectric beam model in
three-dimension:
\begin{align}
\label{1.1}
\begin{cases}
\rho v_{tt}(x,t)-\alpha\Delta v(x,t)+\gamma\beta\Delta p(x,t)+g_1(v_t)=f_1(v), &(x,t)\in\Omega\times\mathbb{R}^+,\\
\mu p_{tt}(x,t)-\beta \Delta p(x,t)+\gamma\beta\Delta v(x,t)+g_2(p_t)=f_2(p), &(x,t)\in\Omega\times\mathbb{R}^+,\\
v(x,t)=p(x,t)=0, &(x,t)\in \Gamma_0\times\mathbb{R}^+,\\
\alpha \frac{\partial v}{\partial \vec{n}}(x,t)-\gamma\beta\frac{\partial p}{\partial \vec{n}}(x,t)=\beta \frac{\partial p}{\partial \vec{n}}(x,t)-\gamma\beta\frac{\partial v}{\partial \vec{n}}(x,t)=0,  &(x,t)\in \Gamma_1\times\mathbb{R}^+,\\
v(x,0)=v_0(x),~v_t(x,0)=v_1(x),  &x\in \Omega,\\
p(x,0)=p_0(x),~p_t(x,0)=p_1(x), &x\in \Omega.
\end{cases}
\end{align}
Here $\mathbb{R}^+:=(0,+\infty)$,  $\Omega\subset \mathbb{R}^3$,
which represents the entire volume of the beam,  is a
bounded domain with smooth boundary $\partial\Omega=\Gamma_0\cup
\Gamma_1$  with  $ \Gamma_0\cap\Gamma_1=\emptyset$,
and $\vec{n}$ is the unit outward vector of $\Gamma_1.$ The
nonlinearities $f_1(v)$ and $f_2(p)$ are supercritical interior
sources, as well as the damping functions $g_1$ and $g_2$ vanishing
at the origin are arbitrary continuous and monotone increasing.
Along the transverse direction at any point $x\in \Omega$, one
denotes the transverse displacement of the beam and the total load
of the electric displacement by the unknown functions $v(x, t)$ and
$p(x, t)$, respectively. $(v_0,v_1,p_0,p_1)\in
H_{\Gamma_0}^1(\Omega)\times L^2(\Omega) \times
H_{\Gamma_0}^1(\Omega)\times L^2(\Omega)$  are the specific  initial
data.

 The three-dimensional system \eqref{1.1} could be regarded as a
generalization of the one-dimensional piezoelectric beam model. The 3D magnetically affected piezoelectric beam model is critical
for advancing modern technologies that rely on multifunctional
materials and complex interactions between mechanical, electrical 
and magnetic fields, for instance, advanced energy harvesting, spintronic systems or magnetoelectric sensors, smart materials, etc.
 While the terminology ``beam'' is retained for historical continuity and to emphasize the connection with the extensive literature on piezoelectric beams, the domain $\Omega \subset \mathbb{R}^3$ should be understood as a general bounded piezoelectric medium, not necessarily a slender beam in the geometric sense.
 Since the beam has a cross-section in 3D, so the boundary $\partial\Omega$ include the two end faces and the surrounding surfaces.
 The boundary decomposition $\partial\Omega = \Gamma_0 \cup \Gamma_1$ is a mathematical abstraction that separates two types of boundary control:
\begin{itemize}
    \item $\Gamma_0$: where Dirichlet conditions ($v = p = 0$) are imposed, modeling perfect clamping and electrical grounding.
    \item $\Gamma_1$: where coupled Neumann-type conditions reflect distributed mechanical and electrical controls that incorporate magnetic coupling effects.
\end{itemize}
Physically, this formulation accommodates various configurations:
\begin{enumerate}[$(i)$]
    \item Slender beam-like domains, where $\Gamma_0$ and $\Gamma_1$ may represent portions of the lateral surface with different control strategies.
    \item Piezoelectric plates or shells with surface-mounted electrodes.
    \item Bulk piezoelectric materials with mixed boundary conditions.
\end{enumerate}
We stress that the mathematical structure coupled wave equations with nonlinear
damping and supercritical sources is our primary focus. The analysis
developed here is independent of the specific geometry. This
abstraction allows our results to apply to a broader class of
piezoelectric structures beyond simple one-dimensional beams.
Similar abstract multidimensional generalizations can be found in
recent works on piezoelectric continua \cite{Z2022,Z2023}. {\color{blue}In
\cite{Z2022}, the authors studied} a fully magnetic effected nonlinear
piezoelectric beam with viscoelastic infinite memory of the form
\begin{align*}
\begin{cases}
\rho v_{tt}(x,t)-\alpha\Delta v(x,t)+\gamma\beta\Delta p(x,t)+\displaystyle\int^\infty_0g(s)\Delta v(x,t-s)ds=f_1(v,p), &(x,t)\in\Omega\times\mathbb{R}^+,\\
\mu p_{tt}(x,t)-\beta \Delta p(x,t)+\gamma\beta\Delta
v(x,t)=f_2(v,p), &(x,t)\in\Omega\times\mathbb{R}^+,
\end{cases}
\end{align*}
with the same boundary and initial conditions as \eqref{1.1}. They
proved that the corresponding coupled linear system can be
indirectly stabilized exponentially and then obtained the
exponential decay of  the nonlinear coupled  system  by the energy
estimation method under certain condition. However if the infinite
memory term  is replaced by  friction-type infinite memory term, its
damping effect is too weak to indirectly stabilize the system
exponentially. They proved that the system can be indirectly
stabilized polynomially, see  \cite{Z2023}. But they did not
consider the supercritical sources in the above papers. One can also
find some stability results on 2D fully magnetic effected
piezoelectric beam in \cite{AV1,AV2}.

However, to the best of our knowledge, there is no work discussing
the impact of the interior nonlinear dampings and the strong
nonlinear sources on solutions of piezoelectric beam model. This
paper aims to discuss  the influence of the relationship between the
strong sources and damping terms on the behaviors of solutions of
the strongly coupled hyperbolic equations \eqref{1.1}.

It is well-known that wave equations with subcritical or critical
semi-linear nonlinearities have been extensively discussed in terms
of the standard fixed-point theorems and Galerkin approximations
method. PDEs with higher-order nonlinearities, for instance,
supercritical and super-supercritical sources are more difficult to
discuss the well-posedness, asymptotic behavior of solutions than
subcritical or critical sources. For a single wave equation with
strong nonlinearities, there are many papers such as
\cite{AC2009,BLR,B2009,B20082,B20081,B2010,B20121,CDL,Dom2023,Ga2006,GR2014,GR2017,GRS2018,R2012,R20121,Viti1,Viti2}
on the well-posedness, blow-up results, energy decay of solutions
and so on.  For systems of nonlinear wave equations with strong
nonlinearities, Guo and  Rammaha \cite{G2014} considered a coupled
wave equation with   supercritical  sources and nonlinear dampings.
A nonlinear source and a damping term are also included in the
boundary condition. They obtained the existence of local and global
weak solutions, and uniqueness of weak solutions by using nonlinear
semigroups and the theory of monotone operators. In \cite{G20131},
they  verified that every weak solution blows up in finite time if
the initial energy is negative and the source term surpasses the
damping term. In addition, in \cite{G2013}, they  also  obtained
global existence of a unique weak solution and established
exponential and algebraic uniform decay rates of energy based on the
potential well theory proposed by Sattinger \cite{S1968} and
developed by Liu \cite{L2003}. And they proved a blow-up result for
weak solutions with nonnegative initial energy.

Motivated by the above results, in this paper we study the influence
of the relationship between the  sources and damping terms on the
behaviors of solutions of the strongly coupled problem \eqref{1.1}.
The main novelty and contribution of this paper are listed as
follows.
\begin{enumerate}[$(1)$]
  \item Inspired by \cite{B1993,G2014}, we will use the theory of nonlinear semigroups and monotone operators to discuss the local existence of solutions to problem \eqref{1.1}. The main strategy is to transfer problem \eqref{1.1} to a Cauchy problem of semigroup form and to verify that the semigroup generator is $m$-accretive in an appropriate phase space.
%  The maximal monotonicity and coercivity of a certain nonlinear operator is the biggest challenge.
A suitable choice of the function space and a combination of various
techniques in monotone operator theory is necessary.
  \item By the potential well theory, the global existence of solutions to problem \eqref{1.1} is established.  When both nonlinear dampings exponent $m_1$ and $m_2$ are equal to $1$, the  exponential decay of the total energy is obtained. When at least one of $m_1$ and $m_2$  is not equal to $1$, the polynomial, or logarithmic decay rates of the total energy are established. The main novelties on stability are two aspects: (I) The proof of the stabilization estimate does not generate lower-order terms and hence the standard lengthy compactness-uniqueness argument to absorb the lower-order terms is completely discarded. This leads to the proof of the main result is shorter and more concise. (II) As in
 \cite{G2013,GRS2018}, to get the polynomial decay, the  regularity result that $v\in
L^\infty(\mathbb{R}^+;L^{\frac{3}{2}(m_1-1)}(\Omega))$ if $m_1>5$,
$p\in L^\infty(\mathbb{R}^+;L^{\frac{3}{2}(m_2-1)}(\Omega))$ if
$m_2>5$ must be assumed. Thus, the initial data $v_0,~p_0$ have to
be smoother. In this paper  we remove the strong conditions and
obtain a weaker energy decay, see \eqref{241542} below, which is  a
general energy decay.
  \item When the source terms are stronger than damping terms, the blow-up phenomena of weak solutions are discussed in terms of three cases of the initial data. (I) The initial total energy is negative, which represents that the initial potential energy due to the nonlinear forces is large enough. (II) The initial total energy is positive but small enough,  and the initial quadratic energy is large. (III) When interior dampings are linear, the initial total energy is arbitrarily high.   In general, under high initial energy, the potential well approach is not applicable. This poses a certain challenge to study the blow-up of problem \eqref{1.1}.  To overcome this difficulty, we construct a new auxiliary functional that is related to the Nehari functional in a sense of increasing property. With the help of the Levine concavity theory, we provide a new finite time  criterion for problem \eqref{1.1} if  the initial energy is bounded by a constant multiple of $\rho\int_\Omega v_0v_1dx+\mu\int_\Omega p_0p_1dx$ as the upper bound. We also prove that when the initial energy is bounded by $C_0(\|v_0\|^2_2+\|p_0\|^2_2)$, where $C_0>0$ is a constant, the solution to problem \eqref{1.1} blows up in finite time. In addition, an upper bound for the blow up time is derived.
  \end{enumerate}

The rest of this paper is organized as follows. Section \ref{sec2}
introduces some notations and the main theorems. Section \ref{sec3}
deals with the local existence of solutions to problem \eqref{1.1}
by the nonlinear semigroups and monotone operators. Section
\ref{sec4} is devoted to discussing the global existence of
solutions to problem \eqref{1.1}  by the potential well theory. In
Section \ref{sec5}, the exponential,  polynomial, or logarithmic
decay rates of the total energy are proved in terms of the behavior
of the damping terms. The blow-up results are proved for three
different initial total energy in Section \ref{sec6}.

\section{Main results}\label{sec2}

As is customary, $C$ represents a positive constant which may be
different from line to line. Let us  denote the Lebesgue space
$L^2(\Omega)$ of integrable square functions on $\Omega$ equipped
with the norm denoted by $\|\cdot\|_2$. Define the Sobolev space
$H^k(\Omega),~k=1,2$, of the functions in $L^2(\Omega)$ so that its
derivative of order up to $r\le k$ belongs to $L^2(\Omega)$ in the
distributional sense.  Moreover, define the space
$H_{\Gamma_0}^1(\Omega)$ by
\[H_{\Gamma_0}^1(\Omega):=\{u\in H^1(\Omega);u=0~\text{on }\Gamma_0\}.\]
It is clear that the Poincar\'e's inequality holds for \(u\in H^{1}_{\Gamma_{0}}(\Omega)\), i.e.
\begin{align}
\label{2.1} \|u\|_2\le c\|\nabla u\|_2,\quad \text{for all } u\in
H_{\Gamma_0}^1(\Omega),
\end{align}
for some $c>0$. Further $\|u\|_{H_{\Gamma_0}^1(\Omega)}:=\|\nabla
u\|_2$ is an equivalent norm in $H_{\Gamma_0}^1(\Omega)$. Denote the
dual space of $H_{\Gamma_0}^1(\Omega)$ by $(H_{\Gamma_0}(\Omega))'$.

The following assumptions on sources and damping terms will be
imposed throughout this paper.
\begin{assumption}
\label{ass} For $i=1,2$, we assume that
\begin{enumerate}[$(1)$]
  \item $g_i:\mathbb{R}\to \mathbb{R}$ are continuous and monotone increasing functions with $g_i(0)=0$.  In addition, there exist positive constants $a$ and $b$ such that, for $|s|\ge1$,
\[a|s|^{m_i+1}\le g_i(s)s\le b|s|^{m_i+1},\quad m_i\ge1.\]
  \item $f_1$ and $f_2$ are functions in $C^1(\mathbb{R})$ such that
\[|f_i'(s)|\leq C(|s|^{n_i-1}+1),\quad 1\le n_i< 6.\]
  \item $n_i\frac{m_i+1}{m_i}<6$.
\end{enumerate}
\end{assumption}

Throughout this paper,  denote
\[\hat{c}=\min\{n_1+1,n_2+1\}>2.\]

\begin{definition}[\textbf{Weak solution}]\label{def1}
A pair of functions $(v,p)$ is said to be a weak solution of problem
\eqref{1.1} on $[0,T]$ if $(v,p)\in
C([0,T];H_{\Gamma_0}^1(\Omega))\times
H_{\Gamma_0}^1(\Omega)),~v_t\in C([0,T];L^2(\Omega))\cap
L^{m_1+1}(\Omega\times(0,T)),~p_t\in C([0,T];L^2(\Omega))\cap
L^{m_2+1}(\Omega\times(0,T)),~(v(0),p(0))=(v_0,p_0)\in
H_{\Gamma_0}^1(\Omega))\times
H_{\Gamma_0}^1(\Omega),~(v_t(0),p_t(0))=(v_1,p_1)\in
L^2(\Omega)\times L^2(\Omega)$ and the following identities hold
\begin{align}
\label{171146}
&\rho\int_\Omega v_t(t)\varphi(t)dx-\rho\int_\Omega v_t(0)\varphi(0)dx-\rho\int_0^t\int_\Omega v_{t}(s)\varphi_t(s)dxds+\alpha\int_0^t\int_\Omega \nabla v(s)\nabla\varphi(s)dxds \notag\\
&\quad-\gamma\beta\int_0^t\int_\Omega \nabla
p(s)\nabla\varphi(s)dxds+\int_0^t\int_\Omega
g_1(v_t(s))\varphi(s)dxds=\int_0^t\int_\Omega
f_1(v(s))\varphi(s)dxds,
\end{align}
\begin{align}
\label{171147}
&\mu\int_\Omega p_{t}(t)\psi(t)dx-\mu\int_\Omega p_{t}(0)\psi(0)dx-\mu\int_0^t\int_\Omega p_{t}(s)\psi_t(s)dxds+\beta\int_0^t\int_\Omega \nabla p(s)\nabla\psi(s)dxds\notag\\
&\quad-\gamma\beta\int_0^t\int_\Omega\nabla
v(s)\nabla\psi(s)dxds+\int_0^t\int_\Omega
g_2(p_t(s))\psi(s)dxds=\int_0^t\int_\Omega f_2(p(s))\psi(s)dxds,
\end{align}
for $\varphi\in C([0,T];H_{\Gamma_0}^1(\Omega))\cap
L^{m_1+1}(\Omega\times(0,T))$, $\psi\in
C([0,T];H_{\Gamma_0}^1(\Omega))\cap L^{m_2+1}(\Omega\times(0,T))$
with $\varphi_t,\psi_t\in C([0,T];L^2(\Omega))$.
\end{definition}

The first theorem of this paper establishes the existence of the
local weak solution to problem \eqref{1.1} as follows.
\begin{theorem}[\textbf{Local existence and energy identity}]
\label{thm2.2} Let Assumption $\ref{ass}$ hold, then there exists a
local weak solution $(v,p)$ to problem \eqref{1.1} defined on
$[0,T]$ for some $T>0$ depending on the initial quadratic energy
$E(0)$, where the quadratic energy $E(t)$ is given by
\begin{equation}
 \label{12161914}
E(t):=\frac{1}{2}\left(\rho\|v_t\|_2^2+\mu\|p_t\|_2^2+\alpha_1\|\nabla
v\|_2^2+\beta\|\gamma\nabla v-\nabla p\|_2^2\right).
\end{equation}
In addition, for all $t\in [0,T]$, $(v,p)$ satisfies the following
energy identity
\begin{equation}
\label{12161910} E(t)+\int_0^t\int_\Omega
(g_1(v_t)v_t+g_1(p_t)p_t)dxds=E(0)+\int_0^t\int_\Omega
(f_1(v)v_t+f_2(p)p_t)dxds.
\end{equation}
\end{theorem}

\begin{remark}
For the sake of clarity and conciseness, in what follows, we will
study the global existence, decay estimates of the total energy and
the blow-up results of solutions for problem \eqref{1.1} in a
particular case that the nonlinear sources $f_1(v)=|v|^{n_1-1}v$ and
$f_2(p)=|p|^{n_2-1}p$, as well as the nonlinear dampings
$g_1(v_t)=|v_t|^{m_1-1}v_t$ and $g_2(p_t)=|p_t|^{m_2-1}p_t$.
\end{remark}

In order to derive the global existence of solutions,
let us  introduce the concepts of Nehari manifold and potential
well. For space $\mathcal{V}:=H_{\Gamma_0}^{1}(\Omega)\times
H_{\Gamma_0}^{1}(\Omega)$, define the nonlinear functional
$\mathcal{J}:\mathcal{V}\to \mathbb{R}$ by
\begin{align}\label{17-1}
\mathcal{J}(v,p):=\frac{1}{2}\left(\alpha_1\|\nabla
v\|_2^2+\beta\|\gamma\nabla v-\nabla p\|_2^2\right)
-\frac{1}{n_1+1}\|v\|_{n_1+1}^{n_1+1}-\frac{1}{n_2+1}\|p\|_{n_2+1}^{n_2+1}.
\end{align}
Then,  the \emph{potential energy} of problem \eqref{1.1} is shown
by $\mathcal J(v,p)$. The Fr\'{e}chet derivative of $\mathcal{J}$ at
$(v,p)\in \mathcal{V}$ is shown by
\begin{align}\label{17-2}
\langle\mathcal{J}'(v,p),(\bar{v},\bar{p})\rangle
&=\int_\Omega \alpha_1\nabla v\nabla \bar{v}dx+\int_\Omega\beta(\gamma\nabla v-\nabla p)(\gamma\nabla \bar{v}-\nabla \bar{p})dx\notag\\
&\quad-{\color{blue}\int_\Omega |v|^{n_1-1}v\bar{v} dx-\int_\Omega |p|^{n_2-1}p
\bar{p}dx},
\end{align}
for $(\bar{v},\bar{p})\in \mathcal{V}$.  The \emph{kinetic energy}
of problem \eqref{1.1} is defined by
$\frac{1}{2}(\|v_t(t)\|^2_2+\|p_t(t)\|^2_2)$. For $v\in
L^{n_1+1}(\Omega)$ and $p\in L^{n_2+1}(\Omega)$, it is natural to
define the total energy $\mathcal{E}(t)$ by the summation of the
kinetic energy and the potential energy, that is,
\begin{align}\label{231213}
&\mathcal{E}(t):=   \frac{1}{2}(\rho\|v_t\|^2_2+\mu\|p_t\|^2_2)+\mathcal{J}(u(t),p(t))\notag\\
&\quad= \frac{1}{2}\left(\rho\|v_t\|^2_2+\mu\|p_t\|^2_2+\alpha_1\|\nabla v\|_2^2+\beta\|\gamma\nabla v-\nabla p\|_2^2\right)-\frac{1}{n_1+1}\|v\|_{n_1+1}^{n_1+1}-\frac{1}{n_2+1}\|p\|_{n_2+1}^{n_2+1}\notag\\
&\quad=
E(t)-\frac{1}{n_1+1}\|v\|_{n_1+1}^{n_1+1}-\frac{1}{n_2+1}\|p\|_{n_2+1}^{n_2+1}.
\end{align}
Therefore, the energy identity \eqref{12161910} is equivalent to
\begin{align} \label{231218}
\mathcal{E}(t)+\int_0^t(\|v_t\|_{m_1+1}^{m_1+1}+\|p_t\|_{m_2+1}^{m_2+1})ds=\mathcal{E}(0).
\end{align}
Taking the derivative with respect to $t$ arrives at
\begin{align}   \label{231219}
\mathcal{E}'(t)   + \|v_t\|_{m_1+1}^{m_1+1}+\|p_t\|_{m_2+1}^{m_2+1}
= 0,
\end{align}
which gives
\begin{align}   \label{231220}
\mathcal E'(t) \leq 0.
\end{align}
Therefore, $\mathcal{E}(t)$ is non-increasing with respect to time
$t$.

The \emph{Nehari manifold} associated with the functional $\mathcal
J$ is defined as
\begin{align}
\label{17-3} \mathcal{N}:=\left\{(v,p)\in
\mathcal{V}\backslash\{(0,0)\}:
\langle\mathcal{J}'(v,p),(v,p)\rangle=0\right\}.
\end{align}
Let us define the \emph{potential well} associated with the
potential energy $\mathcal{J}(v,p)$ by
\begin{align}    \label{17-5}
\mathcal{W}&:=\{(v,p)\in \mathcal{V}: \mathcal{J}(v,p)<d\},
\end{align}
where the \emph{depth of the potential well} $\mathcal{W}$ is
defined as
\begin{align}    \label{17-6}
d:=\inf_{(v,p)\in\mathcal{N}}\mathcal{J}(v,p).
\end{align}
It is clear from \eqref{17-5} and \eqref{17-6} that the potential
well $\mathcal W$ and the Nehari manifold $\mathcal N$ are disjoint
sets.   We will verify that $d$ is strictly positive (see Lemma
\ref{lem4-1}).  In fact, the depth of the potential well $d$ is
consistent with the mountain pass level, more precisely,
\begin{align}
\label{171711}
d:=\inf_{(v,p)\in\mathcal{N}}\mathcal{J}(v,p)=\inf_{(v,p)\in
\mathcal{V}\backslash\{(0,0)\}}\sup_{\lambda\geq0}\mathcal{J}(\lambda
(v,p)).
\end{align}
\eqref{171711} will be justified in Lemma \ref{1221}.

The potential well $\mathcal W$ can be decomposed into the following
two subparts:
\begin{align*}
\mathcal{W}_1&=\left\{(v,p)\in\mathcal{W}:\alpha_1\|\nabla v\|_2^2+\beta\|\gamma\nabla v-\nabla p\|_2^2>\|v\|_{n_1+1}^{n_1+1}+\|p\|_{n_2+1}^{n_2+1}\right\} \cup\{(0,0)\},\\
\mathcal{W}_2&=\left\{(v,p)\in\mathcal{W}:\alpha_1\|\nabla
v\|_2^2+\beta\|\gamma\nabla v-\nabla
p\|_2^2<\|v\|_{n_1+1}^{n_1+1}+\|p\|_{n_2+1}^{n_2+1}\right\}.
\end{align*}
Clearly, $\mathcal{W}_1\cap\mathcal{W}_2=\emptyset$ and
$\mathcal{W}_1\cup\mathcal{W}_2=\mathcal{W}.$

The following two theorems show that if the initial data
$(v_0,p_0)\in \mathcal W_1$ and the initial total energy $\mathcal
E(0)<d$, then the weak solution of problem \eqref{1.1} is global in
time. Also, we prove the uniform decay rate of energy under
additional assumptions.

\begin{theorem}[\textbf{Global existence}]\label{thmg}
Let $1<n_1,n_2\leq 5$ and Assumption $\ref{ass}(3)$ hold. Suppose
\[(v_0,p_0)\in\mathcal{W}_1,~\mathcal{E}(0)<d,\]
 then problem
\eqref{1.1} admits a global weak solution $(v,p)$. In addition,  for
any $t\geq 0$, the potential energy $\mathcal{J}(v,p)$, the total
energy $ \mathcal{E}(t)$ and the quadratic energy $E(t)$ fulfill
\begin{enumerate}[$(i)$]
  \item $\mathcal{J}(v,p)\leq
\mathcal{E}(t)\leq\mathcal{E}(0)<d,$
  \item $(v,p)\in\mathcal{W}_1,$
  \item $E(t) < \frac{\hat{c}d}{\hat{c}-2},$
  \item $\frac{\hat{c}-2}{\hat{c}}E(t)\leq\mathcal{E}(t)\leq E(t).$
\end{enumerate}
\end{theorem}

\begin{theorem}[\textbf{{\color{blue}Energy} decay rates}]\label{thm2}
Let $1<n_1,n_2\le 5$ and Assumption $\ref{ass}(3)$ hold. Suppose
$v_0\in L^{m_1+1}(\Omega),~p_0\in L^{m_2+1}(\Omega)$ and
 \[(v_0,p_0)\in\tilde{\mathcal{W}}^{\delta}_1,~\mathcal{E}(0)\leq\Lambda(s^*-\delta),\] for a suitably small $\delta>0$ satisfying \eqref{delta-1}, where $\tilde{\mathcal{W}}^{\delta}_1$ and $\Lambda(s^*-\delta)$ are defined in Section $\ref{sec5}$.
 Then the global solution of problem \eqref{1.1}
has the following decay rates.
\begin{enumerate}[$(i)$]
  \item Both $m_1$ and $m_2$ are equal to $1$,
the total energy $\mathcal{E}(t)$ and the quadratic energy $E(t)$
decay to zero exponentially, that is, for any $t\geq0$,
\begin{align*}
\frac{\hat{c}-2}{\hat{c}}E(t)\leq\mathcal{E}(t)\le
\mathcal{E}(0)e^{1-\omega t},
\end{align*}
where $\omega$ is a positive constant.
  \item At least one of $m_1$ and $m_2$  is not equal to $1$, in addition, let $v\in
L^\infty(\mathbb{R}^+;L^{\frac{3}{2}(m_1-1)}(\Omega))$ if $m_1>5$,
$p\in L^\infty(\mathbb{R}^+;L^{\frac{3}{2}(m_2-1)}(\Omega))$ if
$m_2>5$, then the total energy $\mathcal{E}(t)$ and the quadratic
energy $E(t)$ decay polynomially, that is,
\begin{align*}
\frac{\hat{c}-2}{\hat{c}}E(t)\leq\mathcal{E}(t)\le
\mathcal{E}(0)\Big(\frac{1+\eta}{1+\omega \eta
t}\Big)^{\frac{1}{\eta}},
\end{align*}
where $\omega$ is a positive constant, and
$\eta=\max\{\frac{m_1-1}{2},\frac{m_2-1}{2}\}$.
  \item At least one of $m_1$ and $m_2$  is not equal to $1$, assume that $\psi : \mathbb{R}^+\to\mathbb{R}^+$ is a strictly increasing function of class $C^1$, and
\begin{align*}\psi(t)=\int_0^t\chi(s)ds\text{ with }\lim_{t\to \infty}\psi(t)=+\infty,
\end{align*} $\chi$ mapping  $\mathbb{R}^+$ to  $\mathbb{R}^+$ is a non-increasing differential function with $\chi(0)>0$, and $\chi(t)(1+t)$ is uniformly bounded for $t\ge 0$, then
\begin{equation}
\label{241542} \frac{\hat{c}-2}{\hat{c}}E(t)\le \mathcal{E}(t)\leq
\mathcal{E}(0)\left(\frac{1+\eta}{1+\omega\eta\psi(t)}\right)^{\frac{1}{\eta}},
\end{equation}
where $\omega$ is a positive constant, and
$\eta=\max\{\frac{m_1-1}{2},\frac{m_2-1}{2}\}$.
\end{enumerate}
\end{theorem}

\begin{remark}
In Theorem \ref{thm2}$(ii)$, the new regularity result that $v\in
L^\infty(\mathbb{R}^+;L^{\frac{3}{2}(m_1-1)}(\Omega))$ if $m_1>5$,
$p\in L^\infty(\mathbb{R}^+;L^{\frac{3}{2}(m_2-1)}(\Omega))$ if
$m_2>5$ must be established. Thus, the initial data $v_0,~p_0$ have
to be smoother. Although here we remove the strong conditions, a
weaker energy decay is obtained in Theorem \ref{thm2}$(iii)$. As a
example, let us choose $\chi(t)=\frac{1}{C+t}$,  where  $C\ge 1$ is
constant.  A simple computation implies that $\chi(t)$ satisfies the
conditions in Theorem \ref{thm2}$(iii)$. In addition,
$\psi(t)=\int_0^t\frac{1}{C+s}ds=\ln\frac{C+t}{C}.$ Therefore, one
has the following logarithmic energy decay
\[\mathcal{E}(t)\leq  \mathcal{E}(0)\left(\frac{1+\eta}{1+\omega\eta\ln\frac{C+t}{C}}\right)^{\frac{1}{\eta}}.\]
\end{remark}

If the source terms are stronger than damping terms, we obtain
blow-up results when the initial energy is negative, positive,
respectively.

\begin{theorem}[\textbf{Blow-up in finite time for negative initial energy}]
\label{thmb1} Let Assumption $\ref{ass}(3)$ hold. Assume $n_1>m_1$,
$n_2>m_2$, and $\mathcal{E}(0)<0$. Then the weak solution $(v,p)$ of
problem \eqref{1.1} blows up in finite time. In particular,
$$
\limsup_{t \to T_{max}^-} (\alpha_1\|\nabla
v\|_2^2+\beta\|\gamma\nabla v-\nabla p\|_2^2)= +\infty,
$$
for some $0<T_{max}<\infty$.

\end{theorem}

\begin{remark} \label{rmk-p}
For  $n_1>m_1$, $n_2>m_2$  and $n_1\frac{m_1+1}{m_1}<6$,
$n_2\frac{m_2+1}{m_2}<6$ from Assumption $\ref{ass}(3)$, in fact, we
give the restriction that $1<n_1,n_2<5$ and $1\leq m_1,m_2 <5$ for
the blow-up results.
\end{remark}

\begin{theorem}[\textbf{Blow-up in finite time for bounded positive initial energy}]\label{thm6-2}
Let Assumption $\ref{ass}(3)$ hold. Assume $n_1>m_1$, $n_2>m_2$, and
\begin{align*}
E(0)>y_0,\quad 0\leq\mathcal{E}(0)<
\mathcal{M}:=\frac{\hat{c}-2}{2(2+\hat{c})}y_0.
\end{align*}
Then the weak solution $(v,p)$ of problem \eqref{1.1} blows up in
finite time. In other words, $ \limsup_{t \to T_{max}^-}
(\alpha_1\|\nabla v\|_2^2+\beta\|\gamma\nabla v-\nabla p\|_2^2)=
+\infty, $ for some $0<T_{max}<\infty$. Here $y_0$ is defined in
\eqref{b2}.
\end{theorem}

\begin{corollary}\label{cor1}
Let Assumption $\ref{ass}(3)$ hold. Assume $n_1>m_1$, $n_2>m_2$, and
$$
0\leq\mathcal{E}(0)<\mathcal{M}:=\frac{\hat{c}-2}{2(2+\hat{c})}y_0.
$$
 If $(v_0,p_0)\in\mathcal{W}_2$, then the weak solution $(v,p)$ of problem \eqref{1.1} blows up in
finite time.
\end{corollary}

For $g_1(v_t)=v_t$ and $g_2(p_t)=p_t$, we give the following blow-up
result for arbitrarily positive initial energy.
\begin{theorem}[\textbf{Blow-up in finite time for arbitrarily positive initial energy}]\label{thm2.10}
Let Assumption $\ref{ass}(3)$ hold, and $n_1>1,~n_2>1,~m_1=m_2=1$.
Suppose
\[0<\mathcal{E}(0)<\frac{\bar{C}}{\hat{c}}\left(\rho\int_\Omega
v_0v_1dx+\mu\int_\Omega p_0p_1dx\right),\] then   the weak solution
$(v,p)$ of problem \eqref{1.1} blows up in finite time
% In other words,
%$
%\limsup_{t \to T_{max}^-} (\alpha_1\|\nabla v\|_2^2+\beta\|\gamma\nabla v-\nabla p\|_2^2)= +\infty,
%$
%for some $0<T_{max}<\infty$,
where $\bar{C}$ is the positive constant shown in \eqref{8168}.

In addition, if
\[\mathcal{E}(0)\le \frac{\hat{c}-2}{2\hat{c}}\Big(\max\Big\{\frac{2\gamma^2+1}{\alpha_1},\frac2\beta\Big\}\Big)^{-1}c^2\min\Big\{\frac1\rho,\frac1\mu\Big\}(\|v_0\|_2^2+\|p_0\|_2^2),\]
then the blow-up time $T_{max}$ can be estimated from above as
follows
\[T_{max}\le \frac{2(\rho\|v_0\|_2^2+\mu\|p_0\|_2^2+\kappa\tau^2)}{(\hat{c}-2)\Big[\rho\displaystyle\int_\Omega v_0v_1dx+\mu\int_\Omega p_0p_1dx+\kappa\tau\Big]-2(\|v_0\|_2^2+\|p_0\|_2^2)},\]
where $\kappa$ and $\tau$ are shown in \eqref{81616} and
\eqref{8160}, respectively.
\end{theorem}

\section{Local solutions}\label{sec3}
This section is used to prove the local existence in Theorem
\ref{thm2.2}.

Define the Hilbert space
\[\mathcal{H}=H_{\Gamma_0}^1(\Omega)\times L^2(\Omega) \times H_{\Gamma_0}^1(\Omega)\times L^2(\Omega)\]
with
\begin{align*}
\langle(v,u,p,q),(\bar{v},\bar{u},\bar{p},\bar{q},)\rangle_{\mathcal{H}}&=\int_\Omega \alpha_1\nabla v\nabla \bar{v}dx+\int_\Omega \rho u\bar{u}dx\notag\\
&\quad+\int_\Omega\beta(\gamma\nabla v-\nabla p)(\gamma\nabla
\bar{v}-\nabla \bar{p})dx+\int_\Omega\mu q\bar{q}dx,
\end{align*}
and the norm
\begin{equation*}
\label{12161224}
\begin{split}
\|(v,u,p,q)\|^2_{\mathcal{H}}=\alpha_1\|\nabla
v\|_2^2+\beta\|\gamma\nabla v-\nabla
p\|_2^2+\rho\|u\|_2^2+\mu\|q\|_2^2.
\end{split}
\end{equation*}

Define the nonlinear operator  $\mathscr{A}$ in $\mathcal{H}$ by
\begin{equation*}
\mathscr{A}\left(\begin{array}{c}v \\u \\p \\q
\end{array}\right)^T=\left(\begin{array}{c}-u
\\\frac{1}{\rho}\left[-\alpha\Delta v+\gamma\beta\Delta
p+g_1(u)-f_1(v)\right] \\-q \\\frac{1}{\mu}\left[-\beta \Delta
p+\gamma\beta\Delta v+g_2(q)-f_2(p)\right] \end{array}\right)^T
\end{equation*}
with domain
\begin{equation*}
\begin{split}
\mathcal{D}(\mathscr{A})=\Big\{(v,u,p,q)\in
H_{\Gamma_0}^1(\Omega)\times H_{\Gamma_0}^1(\Omega)\times
H_{\Gamma_0}^1(\Omega)\times
H_{\Gamma_0}^1(\Omega);\\
g_1(u),~g_2(q)\in
(H_{\Gamma_0}(\Omega))'\cap L^1(\Omega)
\\ 
-\alpha\Delta v+\gamma\beta\Delta p+g_1(u)-f_1(v)\in L^2(\Omega),~-\beta \Delta p+\gamma\beta\Delta v+g_2(q)-f_2(p) \in L^2(\Omega)\\  
\alpha  \frac{\partial v}{\partial \vec{n}}-\gamma\beta\frac{\partial p}{\partial \vec{n}}=\beta \frac{\partial p}{\partial \vec{n}}-\gamma\beta\frac{\partial v}{\partial \vec{n}}=0\text{ on } \Gamma_1\Big\}.
\end{split}
\end{equation*}
Let $U(t)=(v,v_t,p,p_t)$ and $U_0=(v_0,v_1,p_0,p_1)$, then problem
\eqref{1.1}  can be easily reformulated to the following abstract
evolution equation
\begin{equation}
\label{3.1}
\begin{cases}
\frac{dU(t)}{dt}+\mathscr{A}U(t)=0,\quad t>0,\\
U(0)=U_0.
\end{cases}
\end{equation}

\subsection{Globally Lipschitz sources}
\begin{lemma}
\label{lem3.1} For $i=1,2$, assume that
\begin{enumerate}[$(1)$]
  \item $g_i$ are two continuous and monotone increasing functions satisfying $g_i(0)=0$;
  \item $f_i:H_{\Gamma_0}^1(\Omega)\to L^2(\Omega)$ are globally Lipschitz.
\end{enumerate}
Then, problem \eqref{3.1} admits a unique global strong solution
$U\in W^{1,\infty}(0,T;\mathcal{H})$ for arbitrary $T>0$, provided
the initial datum $U_0\in \mathcal{D}(\mathscr{A})$.
\end{lemma}
\begin{proof}
It suffices to prove that the operator $\mathscr{A}+\omega
\mathcal{I}$ is $m-$accretive for some positive $\omega$, which can
be divided into two steps.

\textbf{Step 1}: $\mathscr{A}+\omega \mathcal{I}$ is accretive for
some positive $\omega$.

For any $U=(v,u,p,q),~\hat{U}=(\hat{v},\hat{u},\hat{p},\hat{q})\in
\mathcal{D}(\mathscr{A})$, integration by parts directly implies
that
\begin{align}
\label{12111939}
&\langle (\mathscr{A}+\omega \mathcal{I})U-(\mathscr{A}+\omega \mathcal{I})\hat{U},U-\hat{U}\rangle_{\mathcal{H}}=\langle \mathscr{A}U-\mathscr{A}\hat{U},U-\hat{U}\rangle_{\mathcal{H}}+\omega\|U-\hat{U}\|_{\mathcal{H}}^2\notag\\
&=-\int_\Omega \alpha_1\nabla (u-\hat{u})\nabla (v-\hat{v})dx\notag\\
&\quad+\int_\Omega \rho \frac{1}{\rho}\left[-\alpha\Delta (v-\hat{v})+\gamma\beta\Delta (p-\hat{p})+g_1(u)-g_1(\hat{u})-f_1(v)+f_1(\hat{v})\right] (u-\hat{u})dx\notag\\
&\quad-\int_\Omega\beta[\gamma\nabla (u-\hat{u})-\nabla (q-\hat{q})][\gamma\nabla (v-\hat{v})-\nabla (p-\hat{p})]dx\notag\\
&\quad+\int_\Omega\mu \frac{1}{\mu}\left[-\beta \Delta (p-\hat{p})+\gamma\beta\Delta (v-\hat{v})+g_2(q)-g_2(\hat{q})-f_2(p)+f_2(\hat{p})\right](q-\hat{q})dx\notag\\
&\quad+\omega\Big[\alpha_1\|\nabla (v-\hat{v})\|_2^2+ \rho \|u-\hat{u}\|_2^2+\beta\|\gamma\nabla (v-\hat{v})-\nabla (p-\hat{p})\|_2^2+\mu \|q-\hat{q}\|_2^2\Big]\notag\\
&=\int_\Omega(g_1(u)-g_1(\hat{u}))(u-\hat{u})dx+\int_\Omega (g_2(q)-g_2(\hat{q}))(q-\hat{q})dx\notag\\
&\quad-\int_\Omega (f_1(v)-f_1(\hat{v}))(u-\hat{u})dx-\int_\Omega(f_2(p)-f_2(\hat{p}))(q-\hat{q})dx\notag\\
&\quad+\omega\Big[\alpha_1\|\nabla (v-\hat{v})\|_2^2+ \rho
\|u-\hat{u}\|_2^2+\beta\|\gamma\nabla (v-\hat{v})-\nabla
(p-\hat{p})\|_2^2+\mu \|q-\hat{q}\|_2^2\Big].
\end{align}
Note that $g_1(u)-g_1(\hat{u}),~g_2(q)-g_2(\hat{q})\in
(H_{\Gamma_0}(\Omega))'\cap L^1(\Omega)$ and
$u-\hat{u},~q-\hat{q}\in H_{\Gamma_0}^{1}(\Omega)$ satisfying
\[(g_1(u)-g_1(\hat{u}))(u-\hat{u})\ge 0,\quad (g_2(q)-g_2(\hat{q}))(q-\hat{q})\ge 0,\]
then by \cite[Lemma 2.6]{B1993}, one has
$(g_1(u)-g_1(\hat{u}))(u-\hat{u})\in L^1(\Omega)$ and
$(g_2(q)-g_2(\hat{q}))(q-\hat{q})\in L^1(\Omega)$. In addition
\begin{equation}
\label{12110} \int_\Omega(g_1(u)-g_1(\hat{u}))(u-\hat{u})dx\ge 0,
\quad \int_\Omega(g_2(q)-g_2(\hat{q}))(q-\hat{q})dx\ge0.
\end{equation}

By utilizing Cauchy–Schwarz inequality and recalling that
$f_1,~f_2:~H_{\Gamma_0}^1(\Omega)\to L^2(\Omega)$ are globally
Lipschitz with Lipschitz constants $L_{f_1},~L_{f_2}$, respectively,
one obtains that
\begin{align}
\label{12111} &\int_\Omega (f_1(v)-f_1(\hat{v}))(u-\hat{u})dx\le
\|f_1(v)-f_1(\hat{v})\|_2\|u-\hat{u}\|_2 \notag\\&\quad\le
L_{f_1}\|\nabla v-\nabla\hat{v}\|_2\|u-\hat{u}\|_2 \le
\frac{L_{f_1}}{2}\|\nabla
(v-\hat{v})\|_2^2+\frac{1}{2}\|u-\hat{u}\|_2^2,
\end{align}
\begin{align}
\label{12112} &\int_\Omega(f_2(p)-f_2(\hat{p}))(q-\hat{q})dx\le
\frac{L_{f_2}
}{2}\|\nabla (p-\hat{p})\|_2^2+\frac{1}{2}\|q-\hat{q}\|_2^2\notag\\
&\quad\le  L_{f_2} \Big[\|\nabla (p-\hat{p})-\gamma\nabla
(v-\hat{v})\|_2^2+\|\gamma\nabla
(v-\hat{v})\|_2^2\Big]+\frac{1}{2}\|q-\hat{q}\|_2^2.
\end{align}
Combining \eqref{12111939} with \eqref{12110}-\eqref{12112} leads to
\begin{align*}
&\langle (\mathscr{A}+\omega \mathcal{I})U-(\mathscr{A}+\omega \mathcal{I})\hat{U},U-\hat{U}\rangle_{\mathcal{H}}\\
&\ge \Big(\omega\alpha_1-\frac{L_{f_1}}{2}-L_{f_2}
\gamma\Big)\|\nabla (v-\hat{v})\|_2^2+\Big(\omega\rho-\frac{1}{2}\Big)\|u-\hat{u}\|_2^2\\
&\quad +\Big(\omega\beta-L_{f_2} \Big)\|\nabla
(p-\hat{p})-\gamma\nabla
(v-\hat{v})\|_2^2+\Big(\omega\mu-\frac{1}{2}\Big)\|q-\hat{q}\|_2^2.
\end{align*}
Thus, for $\omega$ sufficiently large, this estimate allows us to
conclude that $\mathscr{A}+\omega \mathcal{I}$ is accretive.

\textbf{Step 2}: $\mathscr{A}+\lambda \mathcal{I}$ is $m$-accretive
for some positive $\lambda$.

It suffices to show that the range of $\mathscr{A}+\lambda
\mathcal{I}$ is $\mathcal{H}$ for some $\lambda>0$.  We need to show
that for any $(a,b,c,d)\in \mathcal{H}$, there exists  $(v,u,p,q)\in
\mathcal{D}(\mathscr{A})$ such that
\[(\mathscr{A}+\lambda \mathcal{I})(v,u,p,q)=(a,b,c,d),\]
that is
\begin{equation}
\label{3.2}
\begin{cases}
-u+\lambda v= a,\\
\frac{1}{\rho}\left[-\alpha\Delta v+\gamma\beta\Delta p+g_1(u)-f_1(v)\right]+\lambda u=b,\\
-q+\lambda p=c,\\
\frac{1}{\mu}\left[-\beta \Delta p+\gamma\beta\Delta
v+g_2(q)-f_2(p)\right]+\lambda q=d.
\end{cases}
\end{equation}
Note that \eqref{3.2} is equivalent to
\begin{equation}
\label{3.3}
\begin{cases}
-\frac\alpha\lambda \Delta u+\frac{\gamma\beta}{\lambda}\Delta q+g_1(u)-f_1(\frac{a+u}{\lambda})+\lambda u=\rho b+\frac\alpha\lambda\Delta a-\frac{\gamma\beta}{\lambda} \Delta c,\\
-\frac\beta\lambda \Delta q+\frac{\gamma\beta}{\lambda}\Delta
u+g_2(q)-f_2(\frac{q+c}{\lambda})+\lambda q=\mu
d+\frac\beta\lambda\Delta c-\frac{\gamma\beta}{\lambda} \Delta a.
\end{cases}
\end{equation}

Define the inner product in
$\mathcal{V}=H_{\Gamma_0}^{1}(\Omega)\times
H_{\Gamma_0}^{1}(\Omega)$ by
\begin{equation*}
\begin{split}
\langle(v,p),(\bar{v},\bar{p})\rangle_{\mathcal{V}}=\int_\Omega
\nabla v\nabla \bar{v}dx+\int_\Omega\nabla p\nabla \bar{p}dx,
\end{split}
\end{equation*}
and the norm
\begin{equation*}
\begin{split}
\|(v,p)\|^2_{\mathcal{V}}=\|\nabla v\|_2^2+\|\nabla p\|_2^2.
\end{split}
\end{equation*}

Notice that the right hand side of \eqref{3.3} belongs to
$\mathcal{V'}:=(H_{\Gamma_0}(\Omega))'\times
(H_{\Gamma_0}(\Omega))'$. Define the operator $\mathscr{B}:
\mathcal{D}(\mathscr{B})\subset \mathcal{V}\to \mathcal{V'}$ by
\begin{equation*}
\mathscr{B}\left(\begin{array}{c} u \\q
\end{array}\right)^T=\left(\begin{array}{c}-\frac\alpha\lambda
\Delta u+\frac{\gamma\beta}{\lambda}\Delta
q+g_1(u)-f_1(\frac{a+u}{\lambda})+\lambda u\\ -\frac\beta\lambda
\Delta q+\frac{\gamma\beta}{\lambda}\Delta
u+g_2(q)-f_2(\frac{q+c}{\lambda})+\lambda q \end{array}\right)^T
\end{equation*}
with domain
\begin{equation*}
\begin{split}
\mathcal{D}(\mathscr{B})=\Big\{(u,q)\in
\mathcal{V};~g_1(u),~g_2(q)\in (H_{\Gamma_0}(\Omega))'\cap
L^1(\Omega)\Big\}.
\end{split}
\end{equation*}
Thus, we need to prove that the operator $\mathscr{B}:
\mathcal{D}(\mathscr{B})\subset \mathcal{V}\to \mathcal{V'}$ is
surjective. By \cite[Corollary 1.2]{B1993}, it is enough to prove
that $\mathscr{B}$ is maximal monotone and coercive.

Let us decompose $\mathscr{B}$ as two operators:
 \begin{equation*}
\mathscr{B}_1\left(\begin{array}{c} u \\q
\end{array}\right)^T=\left(\begin{array}{c}-\frac\alpha\lambda
\Delta u+\frac{\gamma\beta}{\lambda}\Delta
q-f_1(\frac{a+u}{\lambda})+\lambda u\\ -\frac\beta\lambda \Delta
q+\frac{\gamma\beta}{\lambda}\Delta
u-f_2(\frac{q+c}{\lambda})+\lambda q \end{array}\right)^T
\end{equation*}
and
\begin{equation*}
\mathscr{B}_2\left(\begin{array}{c} u \\q
\end{array}\right)^T=\left(\begin{array}{c}g_1(u)\\
g_2(q)\end{array}\right)^T.
\end{equation*}
\textbf{$\mathscr{B}_1$ is maximal monotone and coercive}.  First,
$\mathcal{D}(\mathscr{B}_1)=\mathcal{V}$. Second,  for any
$V=(u,q)\in \mathcal{V}$ and $\hat{V}=(\hat{u},\hat{q})\in
\mathcal{V}$, one obtains that
\begin{align}
\label{1213}
&\langle\mathscr{B}_1V-\mathscr{B}_1\hat{V}, V-\hat{V}\rangle_{\mathcal{V'}\times\mathcal{V}}\notag\\
&=\frac\alpha\lambda \|\nabla(u-\hat{u})\|_2^2-\frac{2\gamma\beta}{\lambda}\int_\Omega\nabla(q-\hat{q})\nabla(u-\hat{u})dx-\int_\Omega [f_1(\frac{a+u}{\lambda})-f_1(\frac{a+\hat{u}}{\lambda})](u-\hat{u})dx\notag\\
&\quad +\lambda\|u-\hat{u}\|_2^2+\frac\beta\lambda \|\nabla(q-\hat{q})\|_2^2-\int_\Omega [f_2(\frac{c+q}{\lambda})-f_2(\frac{c+\hat{q}}{\lambda})](q-\hat{q})dx+\lambda\|q-\hat{q}\|_2^2\notag\\
&\ge \frac\alpha\lambda \|\nabla(u-\hat{u})\|_2^2-\frac{\gamma^2\beta}{\lambda\varepsilon}\|\nabla(u-\hat{u})\|_2^2-\frac{L}{\lambda}\|\nabla(u-\hat{u})\|_2\|u-\hat{u}\|_2\notag\\
&\quad +\lambda\|u-\hat{u}\|_2^2+\frac\beta\lambda
\|\nabla(q-\hat{q})\|_2^2-\frac{\beta\varepsilon}{\lambda}\|\nabla(q-\hat{q})\|_2^2-\frac{L}{\lambda}\|\nabla(q-\hat{q})\|_2\|q-\hat{q}\|_2+\lambda\|q-\hat{q}\|_2^2
\notag\\
&\ge \Big(\frac\alpha\lambda-\frac{\gamma^2\beta}{\lambda\varepsilon}-\frac{L^2}{2\varepsilon_1\lambda}\Big)\|\nabla(u-\hat{u})\|_2^2+\Big(\lambda-\frac{\varepsilon_1}{2\lambda}\Big)\|u-\hat{u}\|_2^2\notag\\
&\quad +\Big(\frac\beta\lambda-\frac{\beta
\varepsilon}{\lambda}-\frac{L^2}{2\varepsilon_2\lambda}\Big)
\|\nabla(q-\hat{q})\|_2^2+\Big(\lambda-\frac{\varepsilon_2}{2\lambda}\Big)\|q-\hat{q}\|_2^2.
\end{align}
Recall $\alpha_1=\alpha-\gamma^2\beta>0$, let us first choose
$0<\frac{\gamma^2\beta}{\alpha}<\varepsilon<1$, then
$\frac\alpha\lambda-\frac{\gamma^2\beta}{\lambda\varepsilon}>0$ and
$2\frac\beta\lambda-\frac{\beta \varepsilon}{\lambda}>0$. We choose
$\lambda$ sufficiently large such that
$\frac{L^2}{2\lambda(\frac\alpha\lambda-\frac{\gamma^2\beta}{\lambda\varepsilon})}<\varepsilon_1<2\lambda^2$
and $0<\frac{L^2}{2\lambda(\frac\beta\lambda-\frac{\beta
\varepsilon}{\lambda})}<\varepsilon_2<2\lambda^2$, therefore it
follows from \eqref{1213} that
\begin{equation*}
\langle\mathscr{B}_1V-\mathscr{B}_1\hat{V},
V-\hat{V}\rangle_{\mathcal{V'}\times\mathcal{V}}\ge
C\|V-\hat{V}\|_{\mathcal{V}}^2
\end{equation*}
with
\[C:=\max\Big\{\frac\alpha\lambda-\frac{\gamma^2\beta}{\lambda\varepsilon}-\frac{L^2}{2\varepsilon_1\lambda},\lambda-\frac{\varepsilon_1}{2\lambda},\frac\beta\lambda-\frac{\beta \varepsilon}{\lambda}-\frac{L^2}{2\varepsilon_2\lambda},\lambda-\frac{\varepsilon_2}{2\lambda}\Big\}>0,\]
which implies $\mathscr{B}_1$ is strongly monotone. It is direct to
see that the strong monotonicity indicates coercivity of
$\mathscr{B}_1$.

Next, let us verify $\mathscr{B}: \mathcal{V}\to \mathcal{V'}$ is
hemicontinuous. Since any linear operator is hemicontinuous, we
merely need to consider the nonlinear terms
$f_1(\frac{a+u}{\lambda}),~f_2(\frac{q+c}{\lambda})$. Obviously,
$f_1(\frac{a+u}{\lambda})$ and $f_2(\frac{q+c}{\lambda})$ are
continuous from $H_{\Gamma_0}^1(\Omega)$ to
$(H_{\Gamma_0}(\Omega))'$ due to
$f_1,~f_2:~H_{\Gamma_0}^1(\Omega)\to L^2(\Omega)$ are globally
Lipschitz. Thus, $\mathscr{B}_1$ is hemicontinuous. So,
$\mathscr{B}_1$ is maximal monotone.

We are now in a position to prove that
$\mathscr{B}_2:\mathcal{D}(\mathscr{B}_2)\subset \mathcal{V}\to
\mathcal{V'}$ is maximal monotone. First note that
$\mathcal{D}(\mathscr{B}_2)=\mathcal{D}(\mathscr{B})=\Big\{(u,q)\in
\mathcal{V};~g_1(u),~g_2(q)\in (H_{\Gamma_0}(\Omega))'\cap
L^1(\Omega)\Big\}$. Second, we need to study $g_1(u),~g_2(q)$. For
$g_1(u)$, let us define the functional $J_1:
H_{\Gamma_0}^1(\Omega)\to [0,\infty]$ by $J_1(u)=\int_\Omega
j_1(u)dx$, where $j_1:\mathbb{R}\to [0,+\infty)$ is a convex
function defined by $j_1(s)=\int_0^s g_1(t)dt$. It is clear that
$J_1$ is proper, convex, and lower semicontinuous. In addition,
\cite[Corollary 2.3]{B2012} illustrates that $\partial J_1:
H_{\Gamma_0}^1(\Omega)\to (H_{\Gamma_0}(\Omega))'$ satisfies
$\partial J_1(u)=\{z\in (H_{\Gamma_0}(\Omega))'\cap L^1(\Omega):
z=g_1(u)\text{ a.e. in } \Omega\}$. In other words,
$\mathcal{D}(\partial J_1)=\{u\in H_{\Gamma_0}^1(\Omega):g_1(u)\in
(H_{\Gamma_0}(\Omega))'\cap L^1(\Omega)\}$ and for all $u\in
\mathcal{D}(\partial J_1)$, $\partial J_1(u)$ is a singleton such
that $\partial J_1(u)=\{g_1(u)\}$. Since any subdifferential is
maximal monotone, it follows that the maximal monotonicity of the
operator $g_1(u):\mathcal{D}(\partial J_1)\subset
H_{\Gamma_0}^1(\Omega)\to (H_{\Gamma_0}(\Omega))'$. By the same
argument, we also get the maximal monotonicity of the operator
$g_2(q)$. By \cite[Proposition 7.1]{G2014}, we conclude that
$\mathscr{B}_2$ is maximal monotone.

At last, since $g_1$ and $g_2$ are monotone increasing functions
satisfying $g_1(0)=0$, $g_2(0)=0$, it follows that
$\langle\mathscr{B}_2 V,V \rangle\ge 0$ for $V=(u,q)$. Recall
$\mathscr{B}_1$ is coercive, thus
$\mathscr{B}=\mathscr{B}_1+\mathscr{B}_2$ is coercive. By
\cite[Corollary 1.2]{B1993}, we verify the surjectivity of
$\mathscr{B}$. Thus, we have proved that there exist
$(u,q)\in\mathcal{D}(\mathscr{B})\subset  \mathcal{V}$ satisfying
\eqref{3.3}. In addition, it follows from $\eqref{3.2}_1$ and
$\eqref{3.2}_3$ that
$(v,p)=(\frac{a+u}{\lambda},\frac{c+q}{\lambda})\in \mathcal{V}$.

It is not difficult to verify $(v,u,p,q)\in
\mathcal{D}(\mathscr{A})$.  Indeed, consider the weak
form of equation $\eqref{3.2}_2$, for any test function $\varphi \in
H^1_{\Gamma_0}(\Omega)$,
\begin{equation*}
\int_\Omega \bigl[ -\alpha \Delta v + \gamma\beta \Delta p + g_1(u)
- f_1(v) +{\color{blue}\rho \lambda u} \bigr] \varphi \, dx = \int_\Omega \rho b
\varphi \, dx.
\end{equation*}
Let us apply Green's formula, then
\begin{align}
&\alpha \int_\Omega \nabla v \cdot \nabla\varphi \, dx - \gamma\beta
\int_\Omega \nabla p \cdot \nabla\varphi \, dx
- \int_{\Gamma_1} \Big(\alpha \frac{\partial v}{\partial \vec{n}} - \gamma\beta \frac{\partial p}{\partial \vec{n}}\Big) \varphi \, dS \nonumber \\
&+ \int_\Omega \bigl[ g_1(u) - f_1(v) +{\color{blue}\rho \lambda u} \bigr] \varphi \,
dx = \int_\Omega \rho b \varphi \, dx. \label{1443}
\end{align}

On the other hand, because equation $\eqref{3.2}_2$ holds in the
distributional sense, for compactly supported test functions
$\varphi \in C_c^\infty(\Omega)$ the boundary integral vanishes and
we obtain
\begin{equation}\label{1444}
\alpha \int_\Omega \nabla v \cdot \nabla\varphi \, dx - \gamma\beta
\int_\Omega \nabla p \cdot \nabla\varphi \, dx + \int_\Omega \bigl[
g_1(u) - f_1(v) +{\color{blue}\rho \lambda u} \bigr] \varphi \, dx = \int_\Omega \rho
b \varphi \, dx.
\end{equation}
The space $C_c^\infty(\Omega)$ is dense in $H^1_{\Gamma_0}(\Omega)$,
and both sides of \eqref{1444} are continuous linear functionals on
$H^1_{\Gamma_0}(\Omega)$. Consequently, \eqref{1444} holds for all
$\varphi \in H^1_{\Gamma_0}(\Omega)$.

Comparing \eqref{1443} and \eqref{1444} for the same $\varphi \in
H^1_{\Gamma_0}(\Omega)$, we deduce that
\[
\int_{\Gamma_1} \Big(\alpha \frac{\partial v}{\partial \vec{n}} -
\gamma\beta \frac{\partial p}{\partial \vec{n}}\Big) \varphi \, dS =
0 \qquad \forall\, \varphi \in H^1_{\Gamma_0}(\Omega).
\]
Since the trace of $\varphi$ on $\Gamma_1$ can be chosen
arbitrarily, we have
\[
\Big(\alpha \frac{\partial v}{\partial \vec{n}} - \gamma\beta
\frac{\partial p}{\partial \vec{n}}\Big)= 0 \quad \text{a.e. on }
\Gamma_1.
\]
Repeating the same argument with equation $\eqref{3.2}_4$ gives the
second boundary condition
\[
 \beta \frac{\partial p}{\partial \vec{n}} - \gamma\beta \frac{\partial v}{\partial \vec{n}} = 0 \quad \text{a.e. on } \Gamma_1. \]

This completes the proof of Lemma \ref{lem3.1}.
\end{proof}

\subsection{Locally Lipschitz sources}
\begin{lemma}
\label{lem3.2} For $i=1,2$, assume that
\begin{enumerate}[$(1)$]
  \item $g_i$ are two continuous and monotone increasing functions satisfying $g_i(0)=0$ and $g_i(s)s\ge a|s|^{m_i+1}$ for all $|s|\ge 1$, where $a>0$ and $m_i\ge 1$;
  \item $f_i:H_{\Gamma_0}^1(\Omega)\to L^2(\Omega)$ are locally Lipschitz.
\end{enumerate}
Then, problem \eqref{3.1} admits a unique local strong solution
$U\in W^{1,\infty}(0,T;\mathcal{H})$ for arbitrary $T>0$, provided
the initial datum $U_0\in \mathcal{D}(\mathscr{A})$.
\end{lemma}
\begin{proof}
Following \cite{B2010,C2002}, we employ standard truncation of the
sources. Define
\begin{equation*}
f_1^K(v)=
\begin{cases}
f_1(v),&\quad \|\nabla v\|_2\le K,\\
f_1(\frac{Kv}{\|\nabla v\|_2}),&\quad \|\nabla v\|_2> K,
\end{cases}
\text{ and }f_2^K(p)=
\begin{cases}
f_2(p),&\quad \|\nabla p\|_2\le K,\\
f_2(\frac{Kp}{\|\nabla p\|_2}),&\quad \|\nabla p\|_2> K,
\end{cases}
\end{equation*}
where $K$ is a positive constant such that
\[K^2\ge\max\Big\{\frac{2\gamma^2+1}{\alpha_1},\frac2\beta\Big\}(4E(0) + 1).\]

Let us consider the following sequence of approximate problems with
truncated nonlinearity
\begin{equation}
\label{1216}
\begin{cases}
\rho v_{tt}-\alpha\Delta v+\gamma\beta\Delta p+g_1(v_t)=f_1^K(v), &(x,t)\in\Omega\times\mathbb{R}^+,\\
\mu p_{tt}-\beta \Delta p+\gamma\beta\Delta v+g_2(p_t)=f_2^K(p), &(x,t)\in\Omega\times\mathbb{R}^+,\\
v=p=0, &(x,t)\in \Gamma_0\times\mathbb{R}^+,\\
\alpha \frac{\partial v}{\partial \vec{n}}-\gamma\beta\frac{\partial p}{\partial \vec{n}}=\beta \frac{\partial p}{\partial \vec{n}}-\gamma\beta\frac{\partial v}{\partial \vec{n}}=0,  &(x,t)\in \Gamma_1\times\mathbb{R}^+,\\
v(x,0)=v_0(x),~v_t(x,0)=v_1(x),&x\in \Omega,\\
p(x,0)=p_0(x),~p_t(x,0)=n_1(x), &x\in \Omega.
\end{cases}
\end{equation}
Note that for each such $K$, the operators $f_1^K,~f_2^K:
H_{\Gamma_0}^1\to L^2(\Omega)$ are globally Lipschitz continuous.
Therefore, it follows from Lemma \ref{lem3.1} that problem
\eqref{1216} has a unique global strong solution  $U_K\in
W^{1,\infty}(0,T;\mathcal{H})$ for arbitrary $T>0$, provided the
initial datum $U_0\in \mathcal{D}(\mathscr{A})$.

For convenience, in what follows, let us denote $(v_K, p_K)$ by $(v,
p)$. Since $v_t,~p_t\in H_{\Gamma_0}^1$ such that
$g_1(v_t),~g_2(p_t)\in (H_{\Gamma_0}(\Omega))'\cap L^1(\Omega)$, we
may choose $v_t$ and $p_t$ as the multipliers for problem
\eqref{1216} and obtain the following energy identity
\begin{equation}
\label{12161} E(t)+\int_0^t\int_\Omega
(g_1(v_t)v_t+g_1(p_t)p_t)dxds=E(0)+\int_0^t\int_\Omega
(f_1^K(v)v_t+f_2^K(p)p_t)dxds.
\end{equation}

A direct computation implies that $f_i^K: H_{\Gamma_0}^1(\Omega)\to
L^{m_i'}(\Omega)$ are globally Lipschitz with Lipchitz constants
$L_{f_i}^K$, $m_i':=\frac{m_i+1}{m_i},~i=1,2$. By utilizing
H\"older’s and Young’s inequalities,
\begin{equation*}
\begin{split}
\int_0^t\int_\Omega f_1^K(v)v_tdxds&\le \int_0^t\| f_1^K(v)\|_{\frac{m_1+1}{m_1}}\|v_t\|_{m_1+1}ds\\&\le \epsilon \int_0^t\|v_t\|_{m_1+1}^{m_1+1}ds+C(\epsilon)\int_0^t\| f_1^K(v)\|_{\frac{m_1+1}{m_1}}^{\frac{m_1+1}{m_1}}ds\\
&\le \epsilon
\int_0^t\|v_t\|_{m_1+1}^{m_1+1}ds+C(\epsilon)L_{f_1}^K\int_0^t\|
\nabla
v\|_2^{\frac{m_1+1}{m_1}}ds+C(\epsilon)t|f_1(0)|^{\frac{m_1+1}{m_1}}|\Omega|.
\end{split}
\end{equation*}
Recall the assumption on the damping that $g_i(s)s\ge a|s|^{m_i}$
for all $|s|\ge 1$, then
\[\int_0^t\int_\Omega g_1(v_t)v_tdxds\ge a\int_0^t\|v_t\|_{m_1+1}^{m_1+1}ds-at|\Omega|.\]
Let us deal with $\int_0^t\int_\Omega f_2^K(p)p_tdxds$ and
$\int_0^t\int_\Omega g_2(p_t)p_tdxds$ by the same method, then it
follows from  \eqref{12161} that
\begin{align}
\label{12162}
&E(t)+a\int_0^t(\|v_t\|_{m_1+1}^{m_1+1}+\|p_t\|_{m_2+1}^{m_2+1})ds\le E(0)+\epsilon \int_0^t(\|v_t\|_{m_1+1}^{m_1+1}+\|p_t\|_{m_2+1}^{m_2+1})ds\notag\\
&\quad+C(\epsilon)L_{max}^K\int_0^t(\| \nabla
v\|_2^{\frac{m_1+1}{m_1}}+\| \nabla
p\|_2^{\frac{m_2+1}{m_2}})ds+C(\epsilon)t(|f_1(0)|^{\frac{m_1+1}{m_1}}+|f_2(0)|^{\frac{m_2+1}{m_2}})|\Omega|+at|\Omega|.
\end{align}
where $L_{max}^K:=\max\{L_{f_1}^K,L_{f_2}^K\}$. By Young's
inequality, then
\begin{align}
\label{12163}
&\int_0^t(\| \nabla v\|_2^{\frac{m_1+1}{m_1}}+\| \nabla p\|_2^{\frac{m_2+1}{m_2}})ds\notag\\
&\quad\le \frac{m_1+1}{2m_1}\int_0^t\| \nabla v\|_2^2ds+\frac{m_1-1}{2m_1}t+\frac{m_2+1}{2m_2}\int_0^t\| \nabla p\|_2^2ds+\frac{m_2-1}{2m_2}t\notag\\
&\quad\le  \frac{m_1+1}{2m_1}\int_0^t\| \nabla v\|_2^2ds+\frac{m_1-1}{2m_1}t+\frac{m_2+1}{m_2}\int_0^t\| \nabla p-\gamma\nabla v\|_2^2ds\notag\\
&\qquad+\frac{(m_2+1)\gamma^2}{m_2}\int_0^t\| \nabla v\|_2^2ds+\frac{m_2-1}{2m_2}t\notag\\
&\quad\le
2\max\Big\{\frac{m_1+1}{2m_1}+\frac{(m_2+1)\gamma^2}{m_2},
\frac{m_2+1}{m_2}\Big\}\int_0^tE(s)ds+\frac{m_1-1}{2m_1}t+\frac{m_2-1}{2m_2}t.
\end{align}
Combining \eqref{12162} with \eqref{12163}, choosing $a\ge
\epsilon$, one has
\[E(t)\le E(0)+C_1T+C_2\int_0^tE(s)ds,\quad \text{for all }t\in [0,T].\]
Here
\[C_1=C(\epsilon)(|f_1(0)|^{\frac{m_1+1}{m_1}}+|f_2(0)|^{\frac{m_1+1}{m_1}})|\Omega|+a|\Omega|+C(\epsilon)L_{max}^K\Big(\frac{m_1-1}{2m_1}+\frac{m_2-1}{2m_2}\Big),\]
\[C_2=C(\epsilon)L_{max}^K 2\max\Big\{\frac{m_1+1}{2m_1}+\frac{(m_2+1)\gamma^2}{m_2}, \frac{m_2+1}{m_2}\Big\}.\]

By Gronwall's inequality, one has
\begin{equation}
\label{12164} E(t)\le (E(0)+C_1T)e^{C_2t},\quad \text{for all }t\in
[0,T].
\end{equation}
Let us choose
\begin{equation}
\label{12170918} T=\min\Big\{\frac{1}{4C_1},\frac{\ln 2}{C_2}\Big\},
\end{equation} and recall $K^2\ge \max\Big\{\frac{2\gamma^2+1}{\alpha_1},\frac2\beta\Big\}(4E(0)+1)$, it follows from \eqref{12164} that
\begin{equation}
\label{12170928} E(t)\le 2(E(0)+1/4)\le
\frac{K^2}{2\max\Big\{\frac{2\gamma^2+1}{\alpha_1},\frac2\beta\Big\}},\quad\text{
for all }t\in [0,T].
\end{equation}
Note that
\begin{align}
\label{12161346}
&\|\nabla v\|_2^2+\|\nabla p\|_2^2\le (2\gamma^2+1)\|\nabla v\|_2^2+2\|\gamma\nabla v-\nabla p\|_2^2\notag\\
&\quad\le
\max\Big\{\frac{2\gamma^2+1}{\alpha_1},\frac2\beta\Big\}\Big[\alpha_1\|\nabla
v\|_2^2+\beta \|\gamma\nabla v-\nabla p\|_2^2\Big]\le
\max\Big\{\frac{2\gamma^2+1}{\alpha_1},\frac2\beta\Big\} 2E(t),
\end{align}
which implies $\|\nabla v\|_2\le K$ and $\|\nabla p\|_2\le K$.
Therefore, $f_1^K(v) = f_1(v)$ and $f_2^K(q) = f_2(q)$ on the time
interval $[0,T]$. Because of the uniqueness of solutions for problem
\eqref{1216}, the solution to the truncated problem \eqref{1216}
coincides with the solution to problem \eqref{1.1} for $t\in [0,T]$.
This completes the proof of Lemma \ref{lem3.2}.
\end{proof}

\subsection{Completion of the proof of the local existence}
Recall that for the values $3<n_1,n_2<6$, the source $f_i$ is not
locally Lipschitz continuous from $H_{\Gamma_0}^1$  into
$L^2(\Omega)$. Therefore,  in order to apply Lemma \ref{lem3.2} to
prove Theorem \ref{thm2.2}, we shall construct Lipschitz
approximations of the source $f_i$.  Set
\begin{equation}
\label{12161439} f_i^n(s)=f_i(s)\eta_n(s),\quad s\in
\mathbb{R},~i=1,2,~n\in \mathbb{N}
\end{equation}
with smooth cutoff functions $\eta_n\in C_0^\infty(\mathbb{R})$
satisfying $0\le \eta_n\le1$, $\eta_n(s)=1$ if $|s|\le n$,
$\eta_n(s)=0$ if $|s|\ge 2n$, and $|\nabla\eta_n(s)|\le C/n$.

The following result is known in \cite{B2010,G2014}.
\begin{lemma}
\label{lem3.3} For $i=1,2$, suppose $m_i\ge 1$, $0<\epsilon<1$.
Assume $f_i:\mathbb{R}\to \mathbb{R}$ such that $|f_i'(s)|\le
C(|s|^{n_i-1})$, where $n_i\frac{m_i+1}{m_i}\le
\frac{6}{1+2\epsilon}$. For each $n\in \mathbb{N}$, the function
$f_i^n$ defined in \eqref{12161439} satisfy
\begin{enumerate}[$(1)$]
  \item $f_i^n:H_{\Gamma_0}^1(\Omega)\to L^2(\Omega)$ are globally Lipschitz with Lipschitz constant depending on $n$.
  \item $f_i^n:H_{\Gamma_0}^{1-\epsilon}(\Omega)\to L^{\frac{m_i+1}{m_i}}(\Omega)$ is locally Lipschitz continuous where the local Lipschitz constant is independent of $n$.
  \end{enumerate}
\end{lemma}

Based on the truncated sources $f_i^n$ defined in \eqref{12161439},
let us define the nonlinear operator  $\mathscr{A}_n$ in
$\mathcal{H}$ by
\begin{equation*}
\mathscr{A}_n\left(\begin{array}{c}v \\u \\p \\q
\end{array}\right)^T=\left(\begin{array}{c}-u
\\\frac{1}{\rho}\left[-\alpha\Delta v+\gamma\beta\Delta
p+g_1(u)-f_1^n(v)\right] \\-q \\\frac{1}{\mu}\left[-\beta \Delta
p+\gamma\beta\Delta v+g_2(q)-f_2^n(p)\right] \end{array}\right)^T
\end{equation*}
with domain
\begin{equation*}
\begin{split}
\mathcal{D}(\mathscr{A}_n)=\Big\{(v,u,p,q)\in
H_{\Gamma_0}^1(\Omega)\times H_{\Gamma_0}^1(\Omega)\times
H_{\Gamma_0}^1(\Omega)\times
H_{\Gamma_0}^1(\Omega);\\g_1(u),~g_2(q)\in
(H_{\Gamma_0}(\Omega))'\cap L^1(\Omega)
\\ -\alpha\Delta v+\gamma\beta\Delta p+g_1(u)-f_1^n(v)\in L^2(\Omega),~-\beta \Delta p+\gamma\beta\Delta v+g_2(q)-f_2^n(p) \in L^2(\Omega)\\  \alpha  \frac{\partial v}{\partial \vec{n}}-\gamma\beta\frac{\partial p}{\partial \vec{n}}=\beta \frac{\partial p}{\partial \vec{n}}-\gamma\beta\frac{\partial v}{\partial \vec{n}}=0\text{ on } \Gamma_1\Big\}.
\end{split}
\end{equation*}
By Lemma \ref{lem3.3}, $f_1^n(v),~f_2^n(p)\in L^2(\Omega)$ for all
$v,~p\in H_{\Gamma_0}^1(\Omega)$, thus $\mathcal{D}(\mathscr{A}_n)$
is uniform for all $n$. Notice that the space of test functions
$\mathscr{D}(\Omega)^4\subset \mathcal{D}(\mathscr{A}_n)$, and since
the density of $\mathscr{D}(\Omega)^4$ in $\mathcal{H}$, for each
$(v_0, v_1, p_0,p_1)\in  \mathcal{H}$,  there exists a sequence of
functions $U_0^n=(v_0^n, v_1^n, p_0^n,p_1^n)\in
\mathscr{D}(\Omega)^4$  such that $U_0^n\to U_0$ in $\mathcal{H}$.
Therefore,
\begin{equation}
\label{171024} \frac{dU(t)}{dt}+\mathscr{A}_nU(t)=0, \quad
U(0)=(v_0^n, v_1^n, p_0^n,p_1^n)\in \mathscr{D}(\Omega)^4
\end{equation}
has a strong local solution $U^n=(v^n, u^n, p^n,q^n)\in
W^{1,\infty}(0,T;\mathcal{H})$ by Lemmas \ref{lem3.2} and
\ref{lem3.3}. Thanks to Lemma  \ref{lem3.3}, the life span $T$ given
in \eqref{12170918} of each solution $U^n$  is independent of $n$,
since the local Lipschitz constant of the mapping $
f_i^n:H_{\Gamma_0}^1(\Omega)\to L^{\frac{m_1+1}{m_i}}$ is
independent of $n$. Thus, we only emphasize the dependence of $T$ on
$K$. Since $E_n(0)\to E(0)$, we can choose $K$ sufficiently large
such that $K$ is independent of $n$. It follows from
\eqref{12170928} that \[E(t)\le
\Big(2\max\Big\{\frac{2\gamma^2+1}{\alpha_1},\frac2\beta\Big\}\Big)^{-1}K^2,
\quad\text{ for }t\in [0,T],\]  which implies that
\begin{align}
\label{12170932}
\|U^n\|^2_{\mathcal{H}}&=\alpha_1\|\nabla v^n\|_2^2+\beta\|\gamma\nabla v^n-\nabla p^n\|_2^2+\rho\|u^n\|_2^2+\mu\|q^n\|_2^2\notag\\
&\le
\Big(\max\Big\{\frac{2\gamma^2+1}{\alpha_1},\frac2\beta\Big\}\Big)^{-1}K^2.
\end{align}
 By choosing $\epsilon\le a/2$ in \eqref{12162} and combining the uniform boundedness of $E_n(t)$ with \eqref{12161346}, one has
 \begin{equation}
\label{170942}
\int_0^T(\|v^n_t\|_{m_1+1}^{m_1+1}+\|p^n_t\|_{m_2+1}^{m_2+1})ds\le
C(K).
\end{equation}
 It follows from \eqref{12170932} and \eqref{170942} that there exists $U=(v,u,p,q)\in L^\infty(0,T;\mathcal{H})$ such that there exists a subsequence of $\{U^n\}$, labeled again as $\{U^n\}$,
  \begin{equation*}
\label{170955} U^n\to U\text{ weak* in }L^\infty(0,T;\mathcal{H}),
\end{equation*}
  \begin{equation*}
\label{170955} v_t^n\to v_t\text{ weakly in
}L^{m_1+1}(\Omega\times(0,T)),\quad p_t^n\to p_t\text{ weakly in
}L^{m_2+1}(\Omega\times(0,T)),
\end{equation*}
which indicates $u=v_t$ and $q=p_t$.

 By \eqref{12170932} and Aubin-Lions-Simon Compactness Theorem, there exists a subsequence $(v^n,p^n)$, reindexed by $(v^n,p^n)$, such that
\begin{equation*}
\label{171017} (v^n,p^n)\to (v,p)\text{ strongly in }
C([0,T];H_{\Gamma_0}^{1-\epsilon}(\Omega)\times
H_{\Gamma_0}^{1-\epsilon}(\Omega)).
\end{equation*}
Recall that $U^n=(v^n, u^n, p^n,q^n)\in \mathcal{D}(\mathscr{A}_n)$
is a strong solution of \eqref{171024}, then
\begin{equation}
\label{171034}
\begin{split}
&\rho\int_\Omega v_t^n(t)\varphi(t)dx-\rho\int_\Omega v_t^n(0)\varphi(0)dx-\rho\int_0^t\int_\Omega v^n_{t}(s)\varphi_t(s)dxds+\alpha\int_0^t\int_\Omega \nabla v^n(s)\nabla\varphi(s)dxds\\
&\quad-\gamma\beta\int_0^t\int_\Omega \nabla
p^n(s)\nabla\varphi(s)dxds+\int_0^t\int_\Omega
g_1(v^n_t(s))\varphi(s)dxds=\int_0^t\int_\Omega
f_1(v^n(s))\varphi(s)dxds,
\end{split}
\end{equation}
\begin{equation}
\label{171041}
\begin{split}
&\mu\int_\Omega p^n_{t}(t)\psi(t)dx-\mu\int_\Omega p^n_{t}(0)\psi(0)dx-\mu\int_0^t\int_\Omega p^n_{t}(s)\psi_t(s)dxds+\beta\int_0^t\int_\Omega \nabla p^n(s)\nabla\psi(s)dxds\\
&\quad-\gamma\beta\int_0^t\int_\Omega\nabla
v^n(s)\nabla\psi(s)dxds+\int_0^t\int_\Omega
g_2(p^n_t(s))\psi(s)dxds=\int_0^t\int_\Omega f_2(p^n(s))\psi(s)dxds,
\end{split}
\end{equation}
for $\varphi\in C([0,T];H_{\Gamma_0}^1(\Omega))\cap
L^{m_1+1}(\Omega\times(0,T))$, $\psi\in
C([0,T];H_{\Gamma_0}^1(\Omega))\cap L^{m_2+1}(\Omega\times(0,T))$
with $\varphi_t,\psi_t\in C([0,T];L^2(\Omega))$.

For $t\in [0,T]$, we need to show the convergence of nonlinear
sources in \eqref{171034} and \eqref{171041}
\[\lim_{n\to \infty}\int_0^t\int_\Omega f_1(v^n(s))\varphi(s)dxds=\int_0^t\int_\Omega f_1(v(s))\varphi(s)dxds,\]
\[\lim_{n\to \infty}\int_0^t\int_\Omega f_2(p^n(s))\psi(s)dxds=\int_0^t\int_\Omega f_2(p(s))\psi(s)dxds,\]
and the convergence of nonlinear dampings
\[\lim_{n\to \infty}\int_0^t\int_\Omega g_1(v^n_t(s))\varphi(s)dxds=\int_0^t\int_\Omega g_1(v_t(s))\varphi(s)dxds,\]
\[\lim_{n\to \infty}\int_0^t\int_\Omega g_2(p^n_t(s))\psi(s)dxds=\int_0^t\int_\Omega g_2(p_t(s))\psi(s)dxds.\]
The proofs are direct by following \cite{G2014}, the detail proof is
omitted here.

In oder to prove the energy identity \eqref{12161910} stated in
Theorem \ref{thm2.2} for weak solutions.  We shall use the
difference quotients of the solution in time by following the
process in \cite{G2014}. Here we omit the details.

So far, we have completed the proof of Theorem \ref{thm2.2}.

 \section{Global existence of potential well solutions}\label{sec4}

 In this section, we prove the global existence of solutions when the initial data belong to the set $\mathcal W_1$ by the potential well theory.  First, let us prove \eqref{171711}.

 \begin{lemma}\label{1221}
 Let $\mathcal{J}$  be defined by \eqref{17-1} and $d$ by \eqref{17-6}.  Then
 \[d=\inf_{(v,p)\in\mathcal{V}\backslash\{(0,0)\}}\sup_{\lambda\geq0}\mathcal{J}(\lambda (v,p)).\]
 \end{lemma}
 \begin{proof}

 The verification proceeds in three steps.

{\em Step 1.}  For any fixed \((v,p) \neq (0,0)\) in
\(\mathcal{V}\), define the real-valued function
 \[
 \phi(\lambda) := \mathcal{J}(\lambda v, \lambda p), \quad \lambda \geq 0.
 \]
 From \eqref{17-1}, we have
 \begin{align}
\label{A.1}
   \phi(\lambda) & = \frac{\lambda^2}{2}(\alpha_1\|\nabla v\|_2^2 + \beta\|\gamma\nabla v - \nabla p\|_2^2) - \frac{\lambda^{n_1+1}}{n_1+1}\|v\|_{n_1+1}^{n_1+1} - \frac{\lambda^{n_2+1}}{n_2+1}\|p\|_{n_2+1}^{n_2+1}.
\end{align}
 Note that $\alpha_1\|\nabla v\|_2^2 + \beta\|\gamma\nabla v - \nabla p\|_2^2 > 0$ for $(v,p) \neq (0,0)$. If this expression were zero, then $\nabla v = 0$ and $\nabla p = 0$ a.e., which together with the Dirichlet condition on $\Gamma_0$ and Poincaré's inequality forces $v = p = 0$, a contradiction. Differentiating \eqref{A.1} gives
\begin{align}
\label{A.2}
     \phi'(\lambda)& = \lambda (\alpha_1\|\nabla v\|_2^2 + \beta\|\gamma\nabla v - \nabla p\|_2^2) - \lambda^{n_1}\|v\|_{n_1+1}^{n_1+1} - \lambda^{n_2}\|p\|_{n_2+1}^{n_2+1},
\end{align}
\begin{align}
\label{A.3} \phi''(\lambda) = \alpha_1\|\nabla v\|_2^2 +
\beta\|\gamma\nabla v - \nabla p\|_2^2 -
n_1\lambda^{n_1-1}\|v\|_{n_1+1}^{n_1+1}-
n_2\lambda^{n_2-1}\|p\|_{n_2+1}^{n_2+1}.
\end{align}
 Since \(n_1, n_2 > 1\), we observe that (1) \(\phi(0) = 0\), \(\phi'(0) = 0\);
(2) For sufficiently small \(\lambda > 0\), \(\phi'(\lambda) =
\lambda[\alpha_1\|\nabla v\|_2^2 + \beta\|\gamma\nabla v - \nabla
p\|_2^2 - \lambda^{n_1-1}\|v\|_{n_1+1}^{n_1+1} -
\lambda^{n_2-1}\|p\|_{n_2+1}^{n_2+1}] > 0\); (3) As \(\lambda \to
+\infty\), \(\phi'(\lambda) \sim
-\lambda^{\max\{n_1,n_2\}}(\|v\|_{n_1+1}^{n_1+1}+
\|p\|_{n_2+1}^{n_2+1}) \to -\infty\).
 Hence there exists at least one \(\lambda^* > 0\) such that \(\phi'(\lambda^*) = 0\).

 To prove uniqueness, suppose \(\lambda_1, \lambda_2 > 0\) with \(\lambda_1 < \lambda_2\) both satisfy \(\phi'(\lambda_i) = 0\). Then from \eqref{A.2},
 \[
 \frac{\alpha_1\|\nabla v\|_2^2 + \beta\|\gamma\nabla v - \nabla p\|_2^2}{\|v\|_{n_1+1}^{n_1+1}\lambda_1^{n_1-1} + \|p\|_{n_2+1}^{n_2+1}\lambda_1^{n_2-1}} = 1 = \frac{\alpha_1\|\nabla v\|_2^2 + \beta\|\gamma\nabla v - \nabla p\|_2^2}{\|v\|_{n_1+1}^{n_1+1}\lambda_2^{n_1-1} + \|p\|_{n_2+1}^{n_2+1}\lambda_2^{n_2-1}}.
 \]
 Since the function \(h(t) := \|v\|_{n_1+1}^{n_1+1}t^{n_1-1} + \|p\|_{n_2+1}^{n_2+1}t^{n_2-1}\) is strictly increasing for \(t > 0\), we would have \(h(\lambda_1) < h(\lambda_2)\), a contradiction. Therefore, the critical point \(\lambda^*\) is unique.

 Moreover, at this unique critical point,
\begin{equation*}
\begin{split}
\phi''(\lambda^*)& = \alpha_1\|\nabla v\|_2^2 + \beta\|\gamma\nabla v - \nabla p\|_2^2 - n_1(\lambda^*)^{n_1-1}\|v\|_{n_1+1}^{n_1+1} - n_2(\lambda^*)^{n_2-1}\|p\|_{n_2+1}^{n_2+1}\\
&< \alpha_1\|\nabla v\|_2^2 + \beta\|\gamma\nabla v - \nabla p\|_2^2
- (\lambda^*)^{n_1-1}\|v\|_{n_1+1}^{n_1+1} -
(\lambda^*)^{n_2-1}\|p\|_{n_2+1}^{n_2+1} = 0.
\end{split}
\end{equation*}
Thus \(\lambda^*\) is the unique maximum point of \(\phi\) on
\([0,\infty)\).

{\em Step 2.}
 The condition \(\phi'(\lambda^*) = 0\) is equivalent to
 \[
 \lambda^* (\alpha_1\|\nabla v\|_2^2 + \beta\|\gamma\nabla v - \nabla p\|_2^2) = (\lambda^*)^{n_1}\|v\|_{n_1+1}^{n_1+1} + (\lambda^*)^{n_2}\|p\|_{n_2+1}^{n_2+1}.
 \]
 Multiplying both sides by \(\lambda^*\) gives
 \[
 (\lambda^*)^2 (\alpha_1\|\nabla v\|_2^2 + \beta\|\gamma\nabla v - \nabla p\|_2^2) = (\lambda^*)^{n_1+1}\|v\|_{n_1+1}^{n_1+1} + (\lambda^*)^{n_2+1}\|p\|_{n_2+1}^{n_2+1}.
 \]
 Recalling the definition of \(\mathcal{J}'\) from \eqref{17-2}, this is precisely
 \[
 \langle \mathcal{J}'(\lambda^* v, \lambda^* p), (\lambda^* v, \lambda^* p) \rangle = 0,
 \]
 which means \((\lambda^* v, \lambda^* p) \in \mathcal{N}\).

{\em Step 3.}  From {\em Step 1} and {\em Step 2}, for any \((v,p)
\neq (0,0)\),
 \[
 \sup_{\lambda\geq 0} \mathcal{J}(\lambda v, \lambda p) = \mathcal{J}(\lambda^* v, \lambda^* p) \geq d,
 \]
 where the inequality holds because \((\lambda^* v, \lambda^* p) \in \mathcal{N}\) and \(d\) is the infimum over \(\mathcal{N}\). Taking infimum over all \((v,p) \neq (0,0)\) yields
\begin{equation}
\label{A.4} \inf_{(v,p)\in\mathcal{V}\backslash\{(0,0)}
\sup_{\lambda\geq 0} \mathcal{J}(\lambda v, \lambda p) \geq d.
\end{equation}

 Conversely, for any \((v,p) \in \mathcal{N}\), consider again the function \(\phi(\lambda) = \mathcal{J}(\lambda v, \lambda p)\). Since \((v,p) \in \mathcal{N}\), we have \(\langle \mathcal{J}'(v,p), (v,p) \rangle = 0\), which implies \(\phi'(1) = 0\). By the uniqueness proved in Step 1, \(\lambda = 1\) is the unique critical point and hence the maximizer of \(\phi\). Therefore,
 \[
 \mathcal{J}(v,p) = \sup_{\lambda\geq 0} \mathcal{J}(\lambda v, \lambda p) \geq \inf_{(\tilde v,\tilde p)\in\mathcal{V}\backslash\{(0,0)} \sup_{\lambda\geq 0} \mathcal{J}(\lambda \tilde v, \lambda \tilde p).
 \]
 Taking infimum over all \((v,p) \in \mathcal{N}\) gives
 \begin{equation}
\label{A.5}
 d \geq \inf_{(v,p)\in\mathcal{V}\backslash\{(0,0)} \sup_{\lambda\geq 0} \mathcal{J}(\lambda v, \lambda p).
\end{equation}
Combining \eqref{A.4} and \eqref{A.5} establishes the equality
\eqref{171711}.
 \end{proof}

Second, let us illustrate the depth of the potential well $d>0$.
 \begin{lemma}\label{lem4-1}
Let Assumption $\ref{ass}(3)$ hold. For $1<n_1,n_2\leq 5$, then
$d>0$.
\end{lemma}
\begin{proof}
Fix $(v,p) \in \mathcal{N}$. In view of \eqref{17-1}, we get
\begin{align}\label{17-7}
\mathcal{J}(v,p)\geq\left(\frac{1}{2}-\frac{1}{\hat{c}}\right)(\alpha_1\|\nabla
v\|_2^2+\beta\|\gamma\nabla v-\nabla p\|_2^2),
\end{align}
by recalling $\hat{c}:=\min\{n_1+1,n_2+1\}>2$. Thanks to the Sobolev
embedding $H^1_{\Gamma_0}(\Omega)  \hookrightarrow
L^{n_i+1}(\Omega)$ for $1<n_1,n_2\leq 5$,  let us define the best
embedding constants
\begin{align}\label{12171637}
B_i:=\sup_{s\in
H^1_{\Gamma_0}(\Omega)\backslash\{0\}}\frac{\|s\|^{n_i+1}_{n_i+1}}{\|\nabla
s\|^{n_i+1}_2}.
\end{align}
It follows from \eqref{17-3}, \eqref{12161346} and embedding
inequalities \eqref{12171637} that
\begin{align}
\label{171608}
&\alpha_1\|\nabla v\|_2^2+\beta\|\gamma\nabla v-\nabla p\|_2^2=\|v\|_{n_1+1}^{n_1+1}+\|p\|_{n_2+1}^{n_2+1}\le B_1\|\nabla v\|_{2}^{n_1+1}+B_2\|\nabla p\|_{2}^{n_2+1}\notag\\
&\quad\le \hat{C}\Big[(\alpha_1\|\nabla v\|_2^2+\beta\|\gamma\nabla
v-\nabla p\|_2^2)^{\frac{n_1+1}{2}}+(\alpha_1\|\nabla
v\|_2^2+\beta\|\gamma\nabla v-\nabla
p\|_2^2)^{\frac{n_2+1}{2}}\Big],
\end{align}
where\[\hat{C}=\max\Big\{B_1\max\Big\{\frac{2\gamma^2+1}{\alpha_1},\frac2\beta\Big\}^{^{\frac{n_1+1}{2}}},B_2\max\Big\{\frac{2\gamma^2+1}{\alpha_1},\frac2\beta\Big\}^{^{\frac{n_2+1}{2}}}\Big\}.\]
Noting $(v,p)\neq(0,0)$, we infer from \eqref{171608} that
$$
(\alpha_1\|\nabla v\|_2^2+\beta\|\gamma\nabla v-\nabla
p\|_2^2)^{\frac{n_1-1}{2}}+(\alpha_1\|\nabla
v\|_2^2+\beta\|\gamma\nabla v-\nabla p\|_2^2)^{\frac{n_2-1}{2}} \geq
\frac{1}{\hat{C}}.
$$
Then, $\alpha_1\|\nabla v\|_2^2+\beta\|\gamma\nabla v-\nabla
p\|_2^2\ge s_0>0$, where $s_0$ is the unique positive root of the
equation
$s^{\frac{n_1-1}{2}}+s^{\frac{n_2-1}{2}}=\frac{1}{\hat{C}}$. It
follows from \eqref{17-7} that
$$
\mathcal{J}(v,p)\geq \left(\frac{1}{2}-\frac{1}{\hat{c}}\right)s_0,\
\;\;\text{for all}\;\; (v,p)\in\mathcal{N}.
$$
This completes the proof.
\end{proof}

It follows  from \eqref{17-1}  that
\begin{align}\label{171709}
\mathcal{J}(v,p)&\geq \frac{1}{2}\left(\alpha_1\|\nabla v\|_2^2+\beta\|\gamma\nabla v-\nabla p\|_2^2\right) -\frac{B_1}{n_1+1}\|\nabla v\|_{2}^{n_1+1}-\frac{B_2}{n_2+1}\|\nabla p\|_{2}^{n_2+1}\notag\\
&\geq \frac{1}{2}\left(\alpha_1\|\nabla v\|_2^2+\beta\|\gamma\nabla v-\nabla p\|_2^2\right) -\frac{\hat{C}}{n_1+1}\left(\alpha_1\|\nabla v\|_2^2+\beta\|\gamma\nabla v-\nabla p\|_2^2\right)^{\frac{n_1+1}{2}}\notag\\
&\quad-\frac{\hat{C}}{n_2+1}\left(\alpha_1\|\nabla
v\|_2^2+\beta\|\gamma\nabla v-\nabla
p\|_2^2\right)^{\frac{n_2+1}{2}}.
\end{align}
Define the function $\Lambda(s)$  by
\begin{align}\label{271321}
\Lambda(s):=\frac{1}{2}s-\frac{\hat{C}}{n_1+1}s^{\frac{n_1+1}{2}}-\frac{\hat{C}}{n_2+1}s^{\frac{n_2+1}{2}}.
\end{align}
In view of $n_1,n_2>1$, then
\begin{equation}
\label{271318}
\Lambda'(s)=\frac12-\frac{\hat{C}}{2}s^{\frac{n_1-1}{2}}-\frac{\hat{C}}{2}s^{\frac{n_2-1}{2}},
\end{equation}
has only one positive zero at $s^*$. It is easy to see that
$\Lambda(s)$ has a maximum value at $s^*$ on $[0,\infty)$, that is,
$\sup_{s\in[0,\infty)}\Lambda(s)=\Lambda(s^*)>0$.

 \begin{proofth1}
 Let the initial data $(v_0,p_0)\in\mathcal{W}_1$ and $\mathcal{E}(0)<d$.

(1)  By \eqref{231220}, it is direct that $\mathcal{E}(t)$ is
non-increasing. Therefore, $\mathcal{E}(t)\leq\mathcal{E}(0)<d$. And
also for any $t\in [0,T)$,
\begin{align}\label{171741}
\mathcal{J}(v,p)\leq\mathcal{E}(t)\leq\mathcal{E}(0)<d.
\end{align}
Then $(i)$ is proved, and we obtain that $(v,p)\in\mathcal{W}$ for
all $t\in [0,T)$.

(2) Let us prove $(v,p)\in\mathcal{W}_1$ for all $t\in[0,T)$ by
contradiction. Assume that there exists a time $t_1\in(0,T)$ such
that $(v(t_1),p(t_1))\notin\mathcal{W}_1$. Due to $(v(t_1),p(t_1))
\in \mathcal{W}$, by recalling
$\mathcal{W}_1\cup\mathcal{W}_2=\mathcal{W}$ and
$\mathcal{W}_1\cap\mathcal{W}_2=\emptyset$, then it must be the case
$(v(t_1),p(t_1))\in\mathcal{W}_2$.

For any $t$, $s\in [0,T)$, then by the mean value theorem, one
obtains
\begin{align}\label{301159}
\Big|\|v(t)\|_{n_1+1}^{n_1+1}-\|v(s)\|_{n_1+1}^{n_1+1}\Big|&\leq
C\int_\Omega
(|v(t)|^{n_1}+|v(s)|^{n_1})|v(t)-v(s)|dx\notag\\
&\leq C(\|v(t)\|^{n_1}_{\frac65{n_1}}+\|v(s)\|^{n_1}_{\frac65{n_1}})\|v(t)-v(s)\|_6\notag\\
&\leq  C(\|\nabla v(t)\|^{n_1}_2+\|\nabla v(s)\|^{n_1}_2)\|\nabla
v(t)-\nabla v(s)\|_2,
\end{align}
where we have used the assumption that $n_1\leq 5$ and the embedding
$H^1_{\Gamma_0}(\Omega)\hookrightarrow L^6(\Omega)$. It follows from
\eqref{301159} and $v\in C([0,T);H^1_{\Gamma_0}(\Omega))$ that
$\|v(t)\|_{n_1+1}^{n_1+1}$ is continuous on $[0,T)$. Likewise,
$\|p(t)\|_{n_2+1}^{n_2+1}$ is continuous on $[0,T)$. Therefore, the
mapping
\begin{align}
\label{1717521} t \mapsto    \alpha_1\|\nabla
v\|_2^2+\beta\|\gamma\nabla v-\nabla
p\|_2^2-\|v\|_{n_1+1}^{n_1+1}-\|p\|_{n_2+1}^{n_2+1}
\end{align}
is continuous.

In view of $(v(0),p(0))\in\mathcal{W}_1$ and
$(v(t_1),p(t_1))\in\mathcal{W}_2$ as well as the continuity of the
function in (\ref{1717521}),
 the intermediate value theorem asserts that there exists a time $s\in (0,t_1)$ such that
\begin{align}\label{171753}
\alpha_1\|\nabla v(s)\|_2^2+\beta\|\gamma\nabla v(s)-\nabla
p(s)\|_2^2=\|v(s)\|_{n_1+1}^{n_1+1}+\|p(s)\|_{n_2+1}^{n_2+1} .
\end{align}
Define $t^*$ be the supremum of all $s\in(0,t_1)$ satisfying
\eqref{171753}, due to the continuity of the function in
(\ref{1717521}), therefore, $t^*\in(0,t_1)$ satisfies
\eqref{171753}, and $(v,p)\in\mathcal{W}_2$ for any $t\in
(t^*,t_1]$. Let us consider the following two cases.

\textbf{Case 1.}  $(v(t^*),p(t^*))\neq(0,0)$. \\
Since \eqref{171753} holds for $t^*$, one has
$(v(t^*),p(t^*))\in\mathcal{N}$. Then, it follows from \eqref{17-6}
that $\mathcal{J}(v(t^*),p(t^*))\geq d$,  which contradicts
\eqref{171741}.

\textbf{Case 2.}  $(v(t^*),p(t^*))=(0,0)$.\\
 Note that
$(v,p)\in\mathcal{W}_2$ for any $t\in(t^*,t_1]$. By the definition
of the set $\mathcal W_2$ and \eqref{171608}, one has that for any
$t\in(t^*,t_1]$,
\begin{equation*}
\begin{split}
&\alpha_1\|\nabla v\|_2^2+\beta\|\gamma\nabla v-\nabla p\|_2^2<\|v\|_{n_1+1}^{n_1+1}+\|p\|_{n_2+1}^{n_2+1}\\
&\leq \hat{C}\left(\alpha_1\|\nabla v\|_2^2+\beta\|\gamma\nabla
v-\nabla p\|_2^2\right)^{\frac{n_1+1}{2}}
+\hat{C}\left(\alpha_1\|\nabla v\|_2^2+\beta\|\gamma\nabla v-\nabla
p\|_2^2\right)^{\frac{n_2+1}{2}}.
\end{split}
\end{equation*}
Since $(0,0)$ does not belong to $\mathcal W_2$, then $(v,p)
\not=(0,0)$ for any $t\in (t^*,t_1]$. This yields  $\alpha_1\|\nabla
v\|_2^2+\beta\|\gamma\nabla v-\nabla p\|_2^2>s_0$, for any $t\in
(t^*,t_1]$, where $s_0>0$ is the unique positive solution of
$s^{\frac{n_1-1}{2}}+s^{\frac{n_2-1}{2}}=\frac{1}{\hat{C}}$. Since
the weak solution $(v,p)$ is continuous from $[0,T)$ to
$\mathcal{V}$, one has $\alpha_1\|\nabla
v(t^*)\|_2^2+\beta\|\gamma\nabla v(t^*)-\nabla p(t^*)\|_2^2\geq
s_0>0$. This contradicts  the assumption $(v(t^*),p(t^*))=(0,0)$.
Therefore $(v,p)\in \mathcal{W}_1$ for all $t\in [0,T)$. Thus,
$(ii)$ is proved.

$(3)$ In the following, we prove that the weak solution $(v,p)$ on
$[0,T)$ is global in time, in other words, the maximum lifespan
$T=\infty$. We need to prove that the quadratic energy $E(t)$ has a
uniform bound independent of time.

It follows from  \eqref{171741} that
\begin{align}\label{171815}
d>\mathcal{J}(v,p)=\frac{1}{2}\left(\alpha_1\|\nabla
v\|_2^2+\beta\|\gamma\nabla v-\nabla p\|_2^2\right)
-\frac{1}{n_1+1}\|v\|_{n_1+1}^{n_1+1}-\frac{1}{n_2+1}\|p\|_{n_2+1}^{n_2+1}.
\end{align}
Since we have proved that $(v,p)\in\mathcal{W}_1$ for all
$t\in[0,T)$, one has
\begin{align}\label{171815-1}
\frac{1}{n_1+1}\|v\|_{n_1+1}^{n_1+1}+\frac{1}{n_2+1}\|p\|_{n_2+1}^{n_2+1}\leq
\frac{1}{\hat{c}}\left(\alpha_1\|\nabla v\|_2^2+\beta\|\gamma\nabla
v-\nabla p\|_2^2\right).
\end{align}
 It follows from \eqref{171815} and
\eqref{171815-1} that for any $t\in [0,T)$,
\begin{align}\label{171818}
\frac{1}{n_1+1}\|v\|_{n_1+1}^{n_1+1}+\frac{1}{n_2+1}\|p\|_{n_2+1}^{n_2+1}<\frac{2d}{\hat{c}-2}.
\end{align}
Substituting \eqref{171818} into \eqref{231218}, and noting
\eqref{231213}, we get that for any $t\in[0,T),$
\begin{align*}
E(t)+\int_0^t(\|v_t\|_{m_1+1}^{m_1+1}+\|p_t\|_{m_2+1}^{m_2+1})ds&=\mathcal{E}(0)+\frac{1}{n_1+1}\|v\|_{n_1+1}^{n_1+1}+\frac{1}{n_2+1}\|p\|_{n_2+1}^{n_2+1}\notag\\
&<d+\frac{2d}{\hat{c}-2}=\frac{\hat{c}d}{\hat{c}-2},
\end{align*}
which gives $E(t)< \frac{\hat{c}d}{\hat{c}-2}$. $(iii)$ is proved.

By standard iterative procedure, one obtains a global weak solution
defined on $[0,\infty)$. That is to say, the maximum lifespan
$T=\infty$.

(4) To show $(iv)$, it follows from (\ref{171815-1}) that
$$
\mathcal{E}(t)\geq
\frac{1}{2}(\|v_t\|^2_2+\|p_t\|^2_2)+\left(\frac{1}{2}-\frac{1}{\hat{c}}\right)\left(\alpha_1\|\nabla
v\|_2^2+\beta\|\gamma\nabla v-\nabla p\|_2^2\right) \geq
\frac{\hat{c}-2}{\hat{c}}E(t).
$$

The proof of Theorem \ref{thmg} is complete.
 \end{proofth1}

\section{Energy decay rates}\label{sec5}
This section is devoted to proving Theorem \ref{thm2}. Before going
further, let us give some lemmas which are crucial to prove Theorem
\ref{thm2}.

Define
\begin{align}
\label{171725} \tilde{\mathcal{W}}_1:=\{(v,p)\in \mathcal{V}:
\alpha_1\|\nabla v\|_2^2+\beta\|\gamma\nabla v-\nabla
p\|_2^2<s^*,\mathcal{J}(v,p)<\Lambda(s^*)\}.
\end{align}
\begin{lemma}\label{lem3-2}
$\tilde{\mathcal{W}}_1$ is a subset of $\mathcal{W}_1$.
\end{lemma}
\begin{proof}
For $(v,p)\in \mathcal{V}\backslash\{(0,0)\}$, \eqref{171709} gives
that
\[\mathcal{J}(\lambda (v,p))\geq \Lambda(\lambda(\alpha_1\|\nabla v\|_2^2+\beta\|\gamma\nabla v-\nabla p\|_2^2)),\quad\text{ for all } \lambda\geq 0.\]
Then we have
$$
\sup_{\lambda\geq0}\mathcal{J}(\lambda (v,p))
 \geq   \sup_{\lambda\geq0}   \Lambda(\lambda(\alpha_1\|\nabla v\|_2^2+\beta\|\gamma\nabla v-\nabla p\|_2^2))  = \sup_{s\in[0,\infty)}\Lambda(s) = \Lambda(s^*).
$$
 It follows from  \eqref{171711} that
\begin{align}    \label{171726}
\Lambda(s^*)\leq d.
\end{align}
By using \eqref{171709}, we have that for any $(v,p) \in \mathcal{V}
\backslash \{(0,0)\}$ with $\alpha_1\|\nabla
v\|_2^2+\beta\|\gamma\nabla v-\nabla p\|_2^2< s^*$,
\begin{align}   \label{171727}
&\|v\|_{n_1+1}^{n_1+1}+\|p\|_{n_2+1}^{n_2+1}\le B_1\|\nabla v\|_{2}^{n_1+1}+B_2\|\nabla p\|_{2}^{n_2+1}\notag\\
&\quad\leq \hat{C}\left(\alpha_1\|\nabla v\|_2^2+\beta\|\gamma\nabla
v-\nabla p\|_2^2\right)^{\frac{n_1+1}{2}}
+\hat{C}\left(\alpha_1\|\nabla v\|_2^2+\beta\|\gamma\nabla v-\nabla p\|_2^2\right)^{\frac{n_2+1}{2}}\notag\\
&\quad<(\alpha_1\|\nabla v\|_2^2+\beta\|\gamma\nabla v-\nabla
p\|_2^2)\Big[\hat{C}s_*^{\frac{n_1-1}{2}}
+\hat{C}s_*^{\frac{n_2-1}{2}}\Big]\notag\\
&\quad=\alpha_1\|\nabla v\|_2^2+\beta\|\gamma\nabla v-\nabla
p\|_2^2.
\end{align}
Combining \eqref{171725}, \eqref{171726} and \eqref{171727}, we
conclude that $\tilde{\mathcal{W}}_1\subset\mathcal{W}_1$.
\end{proof}

For small enough $\delta>0$, let us define a closed subset of
$\tilde{\mathcal{W}}_1$ by
\begin{align*}
\tilde{\mathcal{W}}^{\delta}_1:=\{(v,p)\in \mathcal{V}:
\alpha_1\|\nabla v\|_2^2+\beta\|\gamma\nabla v-\nabla p\|_2^2\leq
s^*-\delta, \, \mathcal{J}(v,p)\leq\Lambda(s^*-\delta)\}.
\end{align*}
$\tilde{\mathcal{W}}^{\delta}_1$ is a closed set because the space
$\mathcal{V}$ is complete and $\mathcal J$ is continuous from
$\mathcal{V}$ to $\mathbb R$. Clearly,
\begin{align*}
\tilde{\mathcal{W}}^{\delta}_1  \subset \tilde{\mathcal{W}}_1
\subset \mathcal W_1.
\end{align*}

\begin{lemma}\label{lem3-3}
Let Assumption $\ref{ass}(3)$ hold. Assume  that $1<n_i\leq 5$,
$m_i>1$, and  $\delta>0$ is suitably small,
\[\mathcal{E}(0)\leq\Lambda(s^*-\delta),\quad(v_0,p_0)\in\tilde{\mathcal{W}}^{\delta}_1. \]
Then problem \eqref{1.1} admits a global solution
$(v,p)\in\tilde{\mathcal{W}}^{\delta}_1$ for all $t\geq0$.
\end{lemma}
\begin{proof}
Since $\Lambda(t)$ is strictly increasing on $(0,s^*)$, one gets
that $\mathcal E(0) \leq \Lambda(s^*-\delta) < d$ due to
\eqref{171726}. Due to $(v_0,p_0)\in\tilde{\mathcal{W}}^{\delta}_1
\subset \mathcal W_1$, there exists a global solution $(v,p)$ with
\[\mathcal{J}(v,p) \leq   \mathcal E(0) \leq \Lambda(s^*-\delta)\] for all $t\geq0$ by Theorem \ref{thmg}. It
remains to show that
\[\alpha_1\|\nabla v\|_2^2+\beta\|\gamma\nabla v-\nabla p\|_2^2\leq s^*-\delta,\quad\text{ for all }t\geq 0.\]
 Since
$\alpha_1\|\nabla v_0\|_2^2+\beta\|\gamma\nabla v_0-\nabla
p_0\|_2^2\leq s^*-\delta$ and $(v,p)\in
C(\mathbb{R}^+,\mathcal{V})$, we assume to the contrary that there
exists   $t_1>0$ such that $\alpha_1\|\nabla
v(t_1)\|_2^2+\beta\|\gamma\nabla v(t_1)-\nabla
p(t_1)\|_2^2=s^*-\delta+\delta_0$ for some $\delta_0\in(0,\delta)$.
Recall \eqref{171709}, then
\[\mathcal{J}(v(t_1),p(t_1))\geq\Lambda(s^*-\delta+\delta_0)>\Lambda(s^*-\delta)\]
by that $\Lambda(t)$ is strictly increasing on $(0,s^*)$. This
contradicts $\mathcal{J}(v,p)\leq\Lambda(s^*-\delta)$ for any
$t\geq0$.
\end{proof}

We need to borrow the following crucial lemma to prove Theorem
\ref{thm2}.
 \begin{lemma}[\cite{M1999}]\label{lem111}
Let $E: \mathbb{R}^+\to\mathbb{R}^+$ be a non-increasing function
and $\psi : \mathbb{R}^+\to\mathbb{R}^+$ be a strictly increasing
function of class $C^1$ such that
\[\psi(0)=0\text{ and }\psi(t)\to+\infty\text{ as }t\to+\infty.\]
Assume that there exist $\sigma \geq  0$ and $\omega > 0$ such
that$:$
 \[\int_t^{+\infty}(E(s))^{1+\sigma}\psi^{\prime}(s)ds\leq  \frac{1}{\omega}(E(0))^\sigma E(t),\]
 then $E$ has the following decay property$:$
\begin{enumerate}[$(1)$]
  \item if $\sigma=0,$ then $E(t)\leq  E(0)e^{1-\omega\psi(t)}$ for all $t\geq  0;$
  \item if $\sigma>0,$ then $E(t)\leq  E(0)\left(\frac{1+\sigma}{1+\omega\sigma\psi(t)}\right)^{\frac{1}{\sigma}}$ for all $t\geq  0$.
\end{enumerate}
\end{lemma}

\begin{proofth2}
By Definition \ref{def1}, $v_t \in L^{m_1+1}(\Omega\times(0,T))$.
Then we have
\begin{align}\label{6-9-1}
\int^T_0\|v\|^{m_1+1}_{m_1+1}dt&= \int^T_0\int_\Omega\left|\int^t_0v_t(s)ds+v_0\right|^{m_1+1}dxdt\nonumber\\
&\leq
C(T^{m_1+1}\|v_t(t)\|^{m_1+1}_{m_1+1}+T\|v_0\|^{m_1+1}_{m_1+1})<\infty.
\end{align}
This implies $v\in L^{m_1+1}(\Omega\times(0,T))$ for all $T\geq0$.
Likewise, $p \in L^{m_2+1}(\Omega\times(0,T))$.  In this situation,
we can replace $\varphi$ by $\chi(t)v(t)\mathcal{E}^\eta(t)$ in
\eqref{171146} and $\psi$ by $\chi(t)p(t)\mathcal{E}^\eta(t)$ in
\eqref{171147}, where $\eta\ge 0$ will be decided later, then adding
the results yields
\begin{align}
\label{2313481}
&\rho\mathcal{E}^\eta(T)\chi(T)\int_\Omega v_t(T)v(T)dx-\rho\mathcal{E}^\eta(S)\chi(S)\int_\Omega v_t(S)v(S)dx\notag\\
&\quad-\rho\int_S^T\int_\Omega v_{t}(t)(\chi(t)v(t)\mathcal{E}^\eta(t))_tdxdt+\mu\mathcal{E}^\eta(T)\chi(T)\int_\Omega p_{t}(T)p(T)dx\notag\\
&\quad-\mu\mathcal{E}^\eta(S)\chi(S)\int_\Omega p_{t}(S)p(S)dx-\mu\int_S^T\int_\Omega p_{t}(t)(\chi(t)p(t)\mathcal{E}^\eta(t))_tdxdt\notag\\
&\quad+\int_S^T\chi(t)\mathcal{E}^\eta(t)\Big(\alpha_1\|\nabla v\|_2^2+\beta\|\gamma \nabla v-\nabla p|_2^2\Big)dt\notag\\
&\quad+\int_S^T\chi(t)\mathcal{E}^\eta(t)\int_\Omega \Big(|v_t|^{m_1-1}v_tv+|p_t|^{m_2-1}p_tp\Big)dxdt\notag\\
&=\int_S^T\chi(t)\mathcal{E}^\eta(t)\Big(\|v\|_{n_1+1}^{n_1+1}+\|p\|_{n_2+1}^{n_2+1}\Big)dt.
\end{align}
Let us combine \eqref{231213} with \eqref{2313481}, then
\begin{align}
\label{231348}
&2\int_S^T\chi(t)\mathcal{E}^{\eta+1}(t)dt=\underbrace{-\rho\mathcal{E}^\eta(T)\chi(T)\int_\Omega
v_t(T)v(T)dx+\rho\mathcal{E}^\eta(S)\chi(S)\int_\Omega
v_t(S)\varphi(S)dx}_{:=I_1}
\notag\\
&+\underbrace{\rho\int_S^T(\chi(t)\mathcal{E}^\eta(t))_t\int_\Omega v_{t}vdxdt}_{:=I_2}\underbrace{-\mu\mathcal{E}^\eta(T)\chi(T)\int_\Omega p_{t}(T)p(T)dx+\mu\mathcal{E}^\eta(S)\chi(S)\int_\Omega p_{t}(S)p(S)dx}_{:=I_3}\notag\\
& +\underbrace{\mu \int_S^T(\chi(t)\mathcal{E}^\eta(t))_t\int_\Omega
p_{t}pdxdt}_{:=I_4}+\underbrace{(1+\rho)\int_S^T\chi(t)\mathcal{E}^\eta(t)\|v_t\|^2_2dt}_{:=I_5}
+\underbrace{(1+\mu)\int_S^T\chi(t)\mathcal{E}^\eta(t)\|p_t\|^2_2dt}_{:=I_6}\notag\\
&\underbrace{-\int_S^T\chi(t)\mathcal{E}^\eta(t)\int_\Omega
|v_t|^{m_1-1}v_tvdxdt}_{:=I_7}
\underbrace{-\int_S^T\chi(t)\mathcal{E}^\eta(t)\int_\Omega|p_t|^{m_2-1}p_tpdxdt}_{:=I_8}\notag\\
&+\underbrace{\int_S^T\chi(t)\mathcal{E}^\eta(t)\Big(\|v\|_{n_1+1}^{n_1+1}+\|p\|_{n_2+1}^{n_2+1}-\frac{2}{n_1+1}\|v\|_{n_1+1}^{n_1+1}-\frac{2}{n_2+1}\|p\|_{n_2+1}^{n_2+1}\Big)dt}_{:=I_9}.
\end{align}

In the sequel, we will estimate terms $I_i$ one by one.

\textbf{(1) Estimate for $I_1$ and $I_3$}\\
By using \eqref{231220}, Cauchy-Schwarz inequality, Poincar\'e's
inequality \eqref{2.1},  \eqref{12161346},  \eqref{12161914} and
Theorem \ref{thmg}$(iv)$, we get
\begin{align}
\label{231509}
|I_1|&\le \frac{\rho}{2}\chi(0)\mathcal{E}^\eta(0)(\|v(T)\|_2^2+\|v_t(T)\|_2^2+\|v(S)\|_2^2+\|v_t(S)\|_2^2)\notag\\
&\le\frac{\rho}{2}\chi(0)\mathcal{E}^\eta(0)(c^2\|\nabla v(T)\|_2^2+\|v_t(T)\|_2^2+c^2\|\nabla v(S)\|_2^2+\|v_t(S)\|_2^2)\notag\\
&\le \rho\chi(0)\mathcal{E}^\eta(0)\Big[2c^2\max\Big\{\frac{2\gamma^2+1}{\alpha_1},\frac2\beta\Big\}+2\Big]E(S)\notag\\
&\le
\rho\chi(0)\mathcal{E}^\eta(0)\Big[2c^2\max\Big\{\frac{2\gamma^2+1}{\alpha_1},\frac2\beta\Big\}+2\Big]\frac{\hat{c}}{\hat{c}-2}\mathcal{E}(S).
\end{align}
By the same procedure, one has
\begin{equation}
\label{231542}
\begin{split}
|I_3|\le\mu\chi(0)\mathcal{E}^\eta(0)\Big[2c^2\max\Big\{\frac{2\gamma^2+1}{\alpha_1},\frac2\beta\Big\}+2\Big]\frac{\hat{c}}{\hat{c}-2}\mathcal{E}(S).
\end{split}
\end{equation}

\textbf{(2) Estimate for $I_2$ and $I_4$}\\
Integration by parts gives
\begin{equation*}
\label{241403}
\begin{split}
I_{2}=\underbrace{\rho\eta\int_S^T\chi(t)\mathcal{E}^{\eta-1}(t)\mathcal{E}'(t)\int_\Omega
v_{t}vdxdt}_{:=I_{21}}+\underbrace{\rho\int_S^T\chi'(t)\mathcal{E}^{\eta}(t)\int_\Omega
v_{t}vdxdt}_{:=I_{22}}.
\end{split}
\end{equation*}
Using again \eqref{231220}, Cauchy-Schwarz inequality, Poincar\'e's
inequality \eqref{2.1},  \eqref{12161346},  \eqref{12161914} and
Theorem $\ref{thmg}(iv)$, we get
\begin{align}
\label{2315510}
|I_{21}|&\le -\frac{\rho \eta}{2}\int_S^T\chi(t)\mathcal{E}^{\eta-1}(t)\mathcal{E}'(t)(\|v\|_2^2+\|v_t\|_2^2)dt\notag\\
&\le -\frac{\rho \eta}{2}\int_S^T\chi(t)\mathcal{E}^{\eta-1}(t)\mathcal{E}'(t)(c^2\|\nabla v\|_2^2+\|v_t\|_2^2)dt\notag\\
&\le -\rho \eta\Big[c^2\max\Big\{\frac{2\gamma^2+1}{\alpha_1},\frac2\beta\Big\}+1\Big]\int_S^T\chi(t)\mathcal{E}^{\eta-1}(t)\mathcal{E}'(t)E(t)dt\notag\\
&\le -\rho \eta\Big[c^2\max\Big\{\frac{2\gamma^2+1}{\alpha_1},\frac2\beta\Big\}+1\Big]\frac{\hat{c}}{\hat{c}-2}\int_S^T\chi(t)\mathcal{E}^{\eta}(t)\mathcal{E}'(t)dt\notag\\
&\le \frac{\rho \eta}{\eta+1}\Big[c^2\max\Big\{\frac{2\gamma^2+1}{\alpha_1},\frac2\beta\Big\}+1\Big]\frac{\hat{c}}{\hat{c}-2}\chi(0)\mathcal{E}^{\eta+1}(S)\notag\\
&\le \frac{\rho
\eta}{\eta+1}\Big[c^2\max\Big\{\frac{2\gamma^2+1}{\alpha_1},\frac2\beta\Big\}+1\Big]\frac{\hat{c}}{\hat{c}-2}\mathcal{E}^{\eta}(0)\chi(0)\mathcal{E}(S),
\end{align}
\begin{align}
\label{241408}
|I_{22}|&\le -\frac\rho2\int_S^T\chi'(t)\mathcal{E}^{\eta}(t)(\|v\|_2^2+\|v_t\|_2^2)dt\notag\\
&\le -\frac{\rho}{2}\int_S^T\chi'(t)\mathcal{E}^{\eta}(t)(c^2\|\nabla v\|_2^2+\|v_t\|_2^2)dt\notag\\
&\le -\rho \Big[c^2\max\Big\{\frac{2\gamma^2+1}{\alpha_1},\frac2\beta\Big\}+1\Big]\int_S^T\chi'(t)\mathcal{E}^{\eta}(t)E(t)dt\notag\\
&\le -\rho \Big[c^2\max\Big\{\frac{2\gamma^2+1}{\alpha_1},\frac2\beta\Big\}+1\Big]\frac{\hat{c}}{\hat{c}-2}\int_S^T\chi'(t)\mathcal{E}^{\eta+1}(t)dt\notag\\
&\le -\rho \Big[c^2\max\Big\{\frac{2\gamma^2+1}{\alpha_1},\frac2\beta\Big\}+1\Big]\frac{\hat{c}}{\hat{c}-2}\Big[\chi(t)\mathcal{E}^{\eta+1}(t)|_S^T-(\eta+1)\int_S^T\chi(t)\mathcal{E}^{\eta}(t)\mathcal{E}'(t)dt\Big]\notag\\
&\le
\rho\Big[c^2\max\Big\{\frac{2\gamma^2+1}{\alpha_1},\frac2\beta\Big\}+1\Big]\frac{\hat{c}}{\hat{c}-2}\mathcal{E}^{\eta}(0)\chi(0)\mathcal{E}(S).
\end{align}
Combining \eqref{2315510} with \eqref{241408}, one has
\begin{equation}
\label{231551} |I_2|\le \frac{\rho
(2\eta+1)}{\eta+1}\Big[c^2\max\Big\{\frac{2\gamma^2+1}{\alpha_1},\frac2\beta\Big\}+1\Big]\frac{\hat{c}}{\hat{c}-2}\mathcal{E}^{\eta}(0)\chi(0)\mathcal{E}(S).
\end{equation}

By the same procedure, one has
\begin{equation}
\label{2316}
\begin{split}
|I_4|\le \frac{\mu
(2\eta+1)}{\eta+1}\Big[c^2\max\Big\{\frac{2\gamma^2+1}{\alpha_1},\frac2\beta\Big\}+1\Big]\frac{\hat{c}}{\hat{c}-2}\mathcal{E}^{\eta}(0)\chi(0)\mathcal{E}(S).
\end{split}
\end{equation}

\textbf{(3)  Estimate for $I_5$ and $I_6$}\\
We have to consider the following two cases.

\textbf{Case 1: Both $m_1$ and $m_2$ are equal to 1.} It follows
from \eqref{231219} that
\begin{align}
\label{231622}
I_5+I_6&\le \max\{(1+\rho),(1+\mu)\}\chi(0)\int_S^T\mathcal{E}^{\eta}(t)(-\mathcal{E}'(t))dt\notag\\
&\le
\frac{\max\{(1+\rho),(1+\mu)\}}{\eta+1}\mathcal{E}^{\eta}(0)\chi(0)\mathcal{E}(S).
\end{align}

\textbf{Case 2: At least one of $m_1$ and $m_2$  is not equal to 1.}
For this case, we need to  distinguish two subclasses.

$(i)$ The first one is $m:=m_1=m_2>1$, or $m_1>m_2=1$, or
$m_2>m_1=1$.

For $m_1=m_2>1$, using H\"older's and Young's inequalities, and
\eqref{231219},  then
\begin{align}
\label{231654}
I_5&\le (1+\rho)|\Omega|^{\frac{m_1-1}{m_1+1}}\int_S^T\chi(t)\mathcal{E}^{\eta}(t)\|v_t\|_{m_1+1}^2dt\notag\\
&\le \epsilon(1+\rho)|\Omega|^{\frac{m_1-1}{m_1+1}}\int_S^T\chi(t)\mathcal{E}^{\eta\frac{m_1+1}{m_1-1}}(t)dt+C(\epsilon)(1+\rho)|\Omega|^{\frac{m_1-1}{m_1+1}}\int_S^T\chi(t)\|v_t\|_{m_1+1}^{m_1+1}dt\notag\\
&\le \epsilon(1+\rho)|\Omega|^{\frac{m_1-1}{m_1+1}}\int_S^T\chi(t)\mathcal{E}^{\eta\frac{m_1+1}{m_1-1}}(t)dt-C(\epsilon)(1+\rho)|\Omega|^{\frac{m_1-1}{m_1+1}}\int_S^T\chi(t)\mathcal{E}'(t)dt\notag\\
&\le
\epsilon(1+\rho)|\Omega|^{\frac{m_1-1}{m_1+1}}\int_S^T\chi(t)\mathcal{E}^{\eta\frac{m_1+1}{m_1-1}}(t)dt+C(\epsilon)(1+\rho)|\Omega|^{\frac{m_1-1}{m_1+1}}\chi(0)\mathcal{E}(S).
\end{align}
Likewise,
\begin{equation}
\label{231701}
\begin{split}
I_6\le
\epsilon(1+\mu)|\Omega|^{\frac{m_2-1}{m_2+1}}\int_S^T\chi(t)\mathcal{E}^{\eta\frac{m_2+1}{m_2-1}}(t)dt+C(\epsilon)(1+\mu)|\Omega|^{\frac{m_2-1}{m_2+1}}\chi(0)\mathcal{E}(S).
\end{split}
\end{equation}
Let $\eta=\frac{m_1-1}{2}=\frac{m_2-1}{2}$, then by combining
\eqref{231654} with \eqref{231701}, one has
\begin{equation}
\label{231707} I_5+I_6\le
\epsilon(2+\rho+\mu)|\Omega|^{\frac{m-1}{m+1}}\int_S^T\chi(t)\mathcal{E}^{\eta+1}(t)dt+C(\epsilon)(2+\rho+\mu)|\Omega|^{\frac{m-1}{m+1}}\chi(0)\mathcal{E}(S).
\end{equation}
For $m_1>m_2=1$, or $m_2>m_1=1$, we easily obtain the similar
results like \eqref{231707}.

$(ii)$ The second one is $m_1>m_2>1$ or $m_2>m_1>1$.

For $m_1>m_2>1$, let us choose $\eta=\frac{m_1-1}{2}$ in
\eqref{231654}, then
$\eta\frac{m_2+1}{m_2-1}>\eta+1=\frac{m_1+1}{2}$ in \eqref{231701}.
Thus \eqref{231701} can be rewritten as
\begin{align*}
I_6&\le
\epsilon(1+\mu)|\Omega|^{\frac{m_2-1}{m_2+1}}\mathcal{E}^{\eta\frac{m_2+1}{m_2-1}-(\eta+1)}(0)\int_S^T\chi(t)\mathcal{E}^{\eta+1}(t)dt+C(\epsilon)(1+\mu)|\Omega|^{\frac{m_2-1}{m_2+1}}\chi(0)\mathcal{E}(S).
\end{align*}
Therefore, we still have \eqref{231707}.

\textbf{(4)  Estimate for $I_7$ and $I_8$}

\textbf{Case 1. Both $m_1$ and $m_2$ are less than or equal to 5.}
It follows from Young's inequality, \eqref{231218}, embedding
inequality \eqref{12171637},  \eqref{12161346}  and Theorem
\ref{thmg}$(iv)$ that
\begin{align}
\label{231810}
&|I_7|\le C(\varepsilon)\int_S^T\chi(t)\mathcal{E}^\eta(t)\|v_t\|^{m_1+1}_{m_1+1}dt+\varepsilon\int_S^T\chi(t)\mathcal{E}^\eta(t)\|v\|^{m_1+1}_{m_1+1}dt\notag\\
&\le C(\varepsilon)\int_S^T\chi(t)\mathcal{E}^\eta(t)(-\mathcal{E}'(t))dt+\varepsilon B_1\int_S^T\chi(t)\mathcal{E}^\eta(t)\|\nabla v\|^{m_1+1}_2dt\notag\\
&\le \frac{C(\varepsilon)}{\eta+1}\mathcal{E}^\eta(0)\chi(0)\mathcal{E}(S)+\varepsilon B_1\Big(2\max\Big\{\frac{2\gamma^2+1}{\alpha_1},\frac2\beta\Big\} \Big)^{\frac{m_1+1}{2}}\int_S^T\chi(t)\mathcal{E}^\eta(t)E(t)^{\frac{m_1+1}{2}}dt\notag\\
&\le \frac{C(\varepsilon)}{\eta+1}\mathcal{E}^\eta(0)\chi(0)\mathcal{E}(S)\notag\\
&\quad+\varepsilon
B_1\Big(2\max\Big\{\frac{2\gamma^2+1}{\alpha_1},\frac2\beta\Big\}
\Big)^{\frac{m_1+1}{2}}\Big(\frac{\hat{c}}{\hat{c}-2}\mathcal{E}(0)\Big)^{\frac{m_1-1}{2}}\int_S^T\chi(t)\mathcal{E}^{\eta+1}(t)dt.
\end{align}
Likewise,
\begin{align}
\label{231946}
|I_8|&\le \frac{C(\varepsilon)}{\eta+1}\mathcal{E}^\eta(0)\chi(0)\mathcal{E}(S)\notag\\
&\quad+\varepsilon
B_2\Big(2\max\Big\{\frac{2\gamma^2+1}{\alpha_1},\frac2\beta\Big\}
\Big)^{\frac{m_2+1}{2}}\Big(\frac{\hat{c}}{\hat{c}-2}\mathcal{E}(0)\Big)^{\frac{m_2-1}{2}}\int_S^T\chi(t)\mathcal{E}^{\eta+1}(t)dt.
\end{align}

\textbf{Case 2. At least one of $m_1$ and $m_2$ is greater than 5.}

(1) For this case, let us first consider the case that  the
condition $v \in L^{\infty}(\mathbb R^+;
L^{\frac{3}{2}(m_1-1)}(\Omega))$ is imposed. It follows from Young's
and H\"older's inequality, \eqref{231218}, embedding inequality
\eqref{12171637},  \eqref{12161346}  and Theorem $\ref{thmg}(iv)$
that
\begin{align}
\label{241042}
&|I_7|\le C(\varepsilon)\int_S^T\chi(t)\mathcal{E}^\eta(t)\|v_t\|^{m_1+1}_{m_1+1}dt+\varepsilon\int_S^T\chi(t)\mathcal{E}^\eta(t)\|v\|^{m_1+1}_{m_1+1}dt\notag\\
&\le C(\varepsilon)\int_S^T\chi(t)\mathcal{E}^\eta(t)(-\mathcal{E}'(t))dt+\varepsilon\int_S^T\chi(t)\mathcal{E}^\eta(t)\|v\|^2_6\|v\|^{m_1-1}_{\frac{3}{2}(m_1-1)}dt\notag\\
&\le C(\varepsilon)\int_S^T\chi(t)\mathcal{E}^\eta(t)(-\mathcal{E}'(t))dt+\varepsilon\int_S^T\chi(t)\mathcal{E}^\eta(t)\|\nabla v\|^2_2\|v\|^{m_1-1}_{\frac{3}{2}(m_1-1)}dt\notag\\
&\le \frac{C(\varepsilon)}{\eta+1}\mathcal{E}^\eta(0)\chi(0)\mathcal{E}(S)+\varepsilon B_12\max\Big\{\frac{2\gamma^2+1}{\alpha_1},\frac2\beta\Big\} \max_{\substack{t\in \mathbb R^+ }}(\|v\|^{m_1-1}_{\frac{3}{2}(m_1-1)})\int_S^T\chi(t)\mathcal{E}^\eta(t)E(t)dt\notag\\
&\le \frac{C(\varepsilon)}{\eta+1}\mathcal{E}^\eta(0)\chi(0)\mathcal{E}(S)\notag\\
&\quad+\varepsilon
B_1\Big(2\max\Big\{\frac{2\gamma^2+1}{\alpha_1},\frac2\beta\Big\}
\Big)^{\frac{m_1+1}{2}}\frac{\hat{c}}{\hat{c}-2}\max_{\substack{t\in
\mathbb R^+
}}(\|v\|^{m_1-1}_{\frac{3}{2}(m_1-1)})\int_S^T\chi(t)\mathcal{E}^{\eta+1}(t)dt.
\end{align}
Likewise,
\begin{align}
\label{241050}
|I_8|&\le \frac{C(\varepsilon)}{\eta+1}\mathcal{E}^\eta(0)\chi(0)\mathcal{E}(S)\notag\\
&\quad+\varepsilon
B_2\Big(2\max\Big\{\frac{2\gamma^2+1}{\alpha_1},\frac2\beta\Big\}
\Big)^{\frac{m_1+1}{2}}\frac{\hat{c}}{\hat{c}-2}\max_{\substack{t\in
\mathbb R^+
}}(\|p\|^{m_2-1}_{\frac{3}{2}(m_2-1)})\int_S^T\chi(t)\mathcal{E}^{\eta+1}(t)dt.
\end{align}

(2) Second, the condition $v \in L^{\infty}(\mathbb R^+;
L^{\frac{3}{2}(m_1-1)}(\Omega))$ is not imposed. Here we  need to
borrow the idea from \cite[Proposition 2.3.]{HT2022} and
\cite{L2023} to prove
\begin{equation}
\label{931} \|v\|_{m_1+1}\le
2^{\frac{m_1}{m_1+1}}(\|v_0\|_{m_1+1}+t^{\frac{m_1}{m_1+1}}(\mathcal{E}(0))^{\frac{1}{m_1+1}})
\end{equation}
for $v_0\in L^{m_1+1}(\Omega)$. In fact, for $v\in
C^1([0,T];L^2(\Omega))$, one has
\[v=v(0)+\int_0^t v_t(s)ds,\quad \text{for all }t>0\text{ and a.e. }x\in \Omega.\]
Further, it follows that
\[|v|^{m_1+1}=2^{m_1}\Big[|v(0)|^{m_1+1}+\Big|\int_0^t v_t(s)ds\Big|^{m_1+1}\Big],\quad \text{for all }t>0\text{ and a.e. }x\in \Omega.\]
Using H\"{o}lder's inequality, and then integrating with respect to
$x$ over $\Omega$, one obtains
\begin{equation*}
\|v\|_{m_1+1}^{m_1+1}\le
2^{m_1}\Big[\|v(0)\|_{m_1+1}^{m_1+1}+t^{m_1}\int_0^t\|v_t(s)\|_{m_1+1}^{m_1+1}ds\Big].
\end{equation*}
From \eqref{231218}, it is clear to obtain \eqref{931}.

 Using H\"older's inequality,  \eqref{931}, Young's inequality and noticing that
\[\mathcal{E}^\eta(t)=\mathcal{E}^{\frac{\eta+1}{m_1+1}}(t)\times \mathcal{E}^{\frac{m_1\eta-1}{m_1+1}}(t),\]
one has
\begin{align}
\label{3.12}
|I_7|&\le \int_S^T\chi(t)\mathcal{E}^{\eta}(t)(-\mathcal{E}_t(t))^{\frac{m_1}{m_1+1}}\|v\|_{m_1+1}dt\notag\\
&\le  2^{\frac{m_1}{m_1+1}}\|v_0\|_{m+1}\int_S^T\chi(t)\mathcal{E}^{\eta}(t)(-\mathcal{E}_t(t))^{\frac{m_1}{m_1+1}}dt\notag\\
&\quad+2^{\frac{m_1}{m_1+1}}(\mathcal{E}(0))^{\frac{1}{m_1+1}}\int_S^T\chi(t)\mathcal{E}^{\eta}(t)(-\mathcal{E}_t(t))^{\frac{m_1}{m_1+1}}t^{\frac{m_1}{m_1+1}}dt\notag\\
&\le \varepsilon
C\int_S^T\chi(t)\mathcal{E}^{\eta+1}(t)dt+C(\varepsilon)\int_S^T\chi(t)\mathcal{E}^{\frac{m_1\eta-1}{m_1}}(t)(-\mathcal{E}_t(t))(1+t)dt
\notag\\
&\le\varepsilon C\int_S^T\chi(t)\mathcal{E}^{\eta+1}(t)dt
+C(\varepsilon)\mathcal{E}^{\frac{m_1\eta-1}{m_1}}(0)\int_S^T\chi(t)(-\mathcal{E}_t(t))(1+t)dt
\notag\\
&\le\varepsilon C\int_S^T\chi(t)\mathcal{E}^{\eta+1}(t)dt
+C(\varepsilon)\mathcal{E}^{\frac{m_1\eta-1}{m_1}}(0)\mathcal{E}(S).
\end{align}
Here $C=\max\{\|v_0\|_{m_1+1},(\mathcal E(0))^{\frac{1}{m_1+1}}\}.$
And in the last step, we have to use the condition that
$\chi(t)(1+t)$ is uniformly bounded. Likewise,
\begin{equation}
\label{3.1211}
\begin{split}
|I_8|\le\varepsilon C\int_S^T\chi(t)\mathcal{E}^{\eta+1}(t)dt
+C(\varepsilon)\mathcal{E}^{\frac{m_2\eta-1}{m_2}}(0)\mathcal{E}(S).
\end{split}
\end{equation}

\textbf{(4)  Estimate for $I_9$}

Since $(u_0,w_0)\in\tilde{\mathcal{W}}^{\delta}_1$ and
$\mathcal{E}(0)\leq\Lambda(s^*-\delta)$, we have $(v,p)\in
\tilde{\mathcal{W}}^{\delta}_1$ for all $t\geq0$ by Lemma
\ref{lem3-3}. Further,
\begin{align}   \label{241029}
\alpha_1\|\nabla v\|_2^2+\beta\|\gamma\nabla v-\nabla p\|_2^2\leq
s_*-\delta, \;\; \text{for all} \,\, t\geq 0.
\end{align}
It follows from \eqref{171608}, \eqref{241029} and Theorem
\ref{thmg}$(iv)$ that
\begin{align}
\label{2312046}
&|I_9|=2\int_S^T\chi(t)\mathcal{E}^\eta(t)\Big[\Big(\frac12-\frac{1}{n_1+1}\Big)\|v\|_{n_1+1}^{n_1+1}+\Big(\frac12-\frac{1}{n_2+1}\Big)\|p\|_{n_2+1}^{n_2+1}\Big)\Big]dt\notag\\
&\le 2\int_S^T\chi(t)\mathcal{E}^\eta(t)(\alpha_1\|\nabla v\|_2^2+\beta\|\gamma\nabla v-\nabla p\|_2^2)\notag\\
&\quad\times\Big(\hat{C}\Big[\Big(\frac12-\frac{1}{n_1+1}\Big)(s_*-\delta)^{\frac{n_1-1}{2}}+\Big(\frac12-\frac{1}{n_2+1}\Big)(s_*-\delta)^{\frac{n_2-1}{2}}\Big]\Big)dt\notag\\
&\le
2\frac{2\hat{c}}{\hat{c}-2}\Big(\hat{C}\Big[\Big(\frac12-\frac{1}{n_1+1}\Big)(s_*-\delta)^{\frac{n_1-1}{2}}+\Big(\frac12-\frac{1}{n_2+1}\Big)(s_*-\delta)^{\frac{n_2-1}{2}}\Big]\Big)\int_S^T\chi(t)\mathcal{E}^{\eta+1}(t)dt.
\end{align}
Since
$\hat{C}\Big(s_*^{\frac{n_1-1}{2}}+s_*^{\frac{n_2-1}{2}}\Big)=1$, $n_1>1$ and $n_2>1$, it
is easy to get that for any small $\delta>0$,
\[\tilde{C}(\delta):=\hat{C}\Big[(s_*-\delta)^{\frac{n_1-1}{2}}+(s_*-\delta)^{\frac{n_2-1}{2}}\Big]<1.\]
Now we take $\delta>0$ satisfying
\begin{align}\label{delta-1}
\tilde{C}(\delta)\cdot\max\Big\{\frac{n_2+1}{n_2-1}\cdot\frac{n_1-1}{n_1+1},\frac{n_1+1}{n_1-1}\cdot\frac{n_2-1}{n_2+1}\Big\}<1,
\end{align}
which implies
\[C(\delta):=\frac{2\hat{c}}{\hat{c}-2}\Big(\hat{C}\Big[\Big(\frac12-\frac{1}{n_1+1}\Big)(s_*-\delta)^{\frac{n_1-1}{2}}+\Big(\frac12-\frac{1}{n_2+1}\Big)(s_*-\delta)^{\frac{n_2-1}{2}}\Big]\Big)<1.\]

\textbf{Case 1: Both $m_1$ and $m_2$ are equal to 1.}

Inserting  \eqref{231509}, \eqref{231542},
\eqref{231551}-\eqref{231622}, \eqref{231810}, \eqref{231946} and
\eqref{2312046} into \eqref{231348}, and recalling $\eta=0$,
choosing $\chi(t)\equiv 1$,  then
\begin{equation*}
\label{241121}
\begin{split}
2(1-C(\delta))\int_S^T\mathcal{E}(t)dt\le
C(\varepsilon)\mathcal{E}(S)+\varepsilon C\int_S^T\mathcal{E}(t)dt.
\end{split}
\end{equation*}
Choosing $\varepsilon$ small enough such that $\varepsilon
C=\frac{2(1-C(\delta))}{2}$, and letting $T$ go to $\infty$, one
obtains
\begin{equation}
\label{241134}
\begin{split}
\int_S^\infty\mathcal{E}(t)dt\le
\frac{2}{2(1-C(\delta))}C(\varepsilon)\mathcal{E}(S).
\end{split}
\end{equation}
Therefore, using Lemma \ref{lem111} for \eqref{241134}, one has
\[\mathcal{E}(t)\le \mathcal{E}(0)e^{1-\omega t}\]
with $\omega=\frac{1-C(\delta)}{C(\varepsilon)}.$

\textbf{Case 2: At least one of $m_1$ and $m_2$  is not equal to 1.}

(1) Inserting  \eqref{231509}, \eqref{231542}, \eqref{231551},
\eqref{2316}, \eqref{231707}, \eqref{231810}, \eqref{231946}  or
\eqref{241042}, \eqref{241050} and \eqref{2312046} into
\eqref{231348}, and choosing $\chi(t)\equiv 1$, then one has
\begin{equation*}
\label{241147}
\begin{split}
2(1-C(\delta))\int_S^T\mathcal{E}^{\eta+1}(t)dt\le
C(\epsilon,\varepsilon)\mathcal{E}^\eta(0)\mathcal{E}(S)+(\varepsilon
C(\mathcal{E}(0))+\epsilon C)\int_S^T\mathcal{E}^{\eta+1}(t)dt.
\end{split}
\end{equation*}
Choosing $\epsilon,~\varepsilon$ small enough such that $\varepsilon
C(\mathcal{E}(0))+\epsilon C=\frac{2(1-C(\delta))}{2}$,
$\eta=\max\{\frac{m_1-1}{2},\frac{m_2-1}{2}\}$ and letting $T$ go to
$\infty$, one obtains
\begin{equation}
\label{241148}
\begin{split}
\int_S^\infty\mathcal{E}^{\eta+1}(t)dt\le
\frac{2}{2(1-C(\delta))}C(\epsilon,\varepsilon)\mathcal{E}^\eta(0)\mathcal{E}(S).
\end{split}
\end{equation}
Therefore, using Lemma \ref{lem111}  for \eqref{241148}, we have
\[\mathcal{E}(t)\le \mathcal{E}(0)\Big(\frac{1+\eta}{1+\omega \eta t}\Big)^{\frac{1}{\eta}}\]
with $\omega=\frac{1-C(\delta)}{C(\epsilon,\varepsilon)}$.

(2) Inserting   \eqref{231509}, \eqref{231542}, \eqref{231551},
\eqref{2316}, \eqref{231707}, \eqref{3.12}, \eqref{3.1211} and
\eqref{2312046} into \eqref{231348},  then one has
\begin{align*}
&2(1-C(\delta))\int_S^T\chi(t)\mathcal{E}^{\eta+1}(t)dt\notag\\
&\quad\le
C(\epsilon,\varepsilon)\mathcal{E}^\eta(0)\mathcal{E}(S)+(\varepsilon
C(\mathcal{E}(0))+\epsilon
C)\int_S^T\chi(t)\mathcal{E}^{\eta+1}(t)dt.
\end{align*}
Choosing $\epsilon,~\varepsilon$ small enough such that $\varepsilon
C(\mathcal{E}(0))+\epsilon C=\frac{2(1-C(\delta))}{2}$,
$\eta=\max\{\frac{m_1-1}{2},\frac{m_2-1}{2}\}$ and letting $T$ go to
$\infty$, one obtains
\begin{equation}
\label{241535}
\begin{split}
\int_S^\infty\chi(t)\mathcal{E}^{\eta+1}(t)dt\le
\frac{2}{2(1-C(\delta))}C(\epsilon,\varepsilon)\mathcal{E}^\eta(0)\mathcal{E}(S).
\end{split}
\end{equation}
Choosing $\psi(t)=\int_0^t\chi(s)ds$ in Lemma \ref{lem111} for
\eqref{241535}, one directly obtains
\[\mathcal{E}(t)\leq  \mathcal{E}(0)\left(\frac{1+\eta}{1+\omega\eta\psi(t)}\right)^{\frac{1}{\eta}}\]
with $\omega=\frac{1-C(\delta)}{C(\epsilon,\varepsilon)}$.

\end{proofth2}

\section{Blow-up in finite time}\label{sec6}

In this section, we are devoted to proving Theorems \ref{thmb1},
\ref{thm6-2}, \ref{thm2.10} in three different initial total energy
if the source terms are more dominant than damping terms.

\subsection{Blow-up in finite time for negative initial energy}\label{sec6.1}
Let $(v,p)$ be a weak solution of problem \eqref{1.1} in the sense
of Definition \ref{def1}. We define the life span $T_{max}$ of the
solution $(v,p)$ to be the supremum of all $T>0$ such that $(v,p)$
is a solution to problem \eqref{1.1} in the sense of Definition
\ref{def1} on $[0, T]$. In the following, we are devoted to showing
that $T_{max}$ is finite and obtain an upper bound for the life span
of solutions.

The main idea is originated from \cite{G1994}. The key of the proof
is to choose a suitable Lyapunov's functional. Define
$$
G(t)=-\mathcal{E}(t),
$$
and
\begin{equation}
\label{271513}
N(t)=\frac{1}{2}\left(\rho\|v\|^2_2+\mu\|p\|^2_2\right).
\end{equation}
It follows from \eqref{231219} that
\begin{align}\label{6-6}
G'(t)&=\|v_t\|_{m_1+1}^{m_1+1}+\|p_t\|_{m_2+1}^{m_2+1}\ge 0,
\end{align}
which implies $G(t)$ is non-decreasing. In view of
$G(0)=-\mathcal{E}(0)>0$, then
\begin{align}\label{6-7}
&0<G(0)\leq G(t)=-\mathcal{E}(t)=-E(t)+\frac{1}{n_1+1}\|v\|_{n_1+1}^{n_1+1}+\frac{1}{n_2+1}\|p\|_{n_2+1}^{n_2+1}\notag\\
&\quad\le
\frac{1}{n_1+1}\|v\|_{n_1+1}^{n_1+1}+\frac{1}{n_2+1}\|p\|_{n_2+1}^{n_2+1}\le
\frac{1}{\hat{c}}(\|v\|_{n_1+1}^{n_1+1}+\|p\|_{n_2+1}^{n_2+1}).
\end{align}

Define
\begin{align*}
Y(t):=G^{1-\varpi}(t)+\varepsilon N'(t),
\end{align*}
where $0<\varepsilon\leq\min\{1,G(0)\}$ will be determined later,
\begin{align}\label{6-4}
0<\varpi<\min\left\{\frac{1}{m_1+1}-\frac{1}{n_1+1},
\;\;\frac{1}{m_2+1}-\frac{1}{n_2+1}, \;\;\frac{n_1-1}{2(n_1+1)},
\;\;\frac{n_2-1}{2(n_2+1)}\right\},
\end{align}
and
\begin{equation}
\label{261438} N'(t)=\rho\int_\Omega v(t)v_t(t)dx+\mu\int_\Omega
p(t)p_t(t)dx.
\end{equation}
Our aim is to show that $Y(t)$ approaches infinity in finite time.
First,  we verify that
\begin{align}\label{6-8}
Y'(t)=(1-\varpi)G^{-\varpi}(t)G'(t)+\varepsilon N''(t),
\end{align}
where
\begin{align}\label{6-9}
N''(t)&=\rho\|v_t\|_2^2+\mu\|p_{t}\|_2^2-(\alpha_1\|\nabla v\|_2^2+\beta\|\gamma\nabla v-\nabla p\|_2^2)\notag\\
&\quad-\int_\Omega |v_t|^{m_1-1}v_tvdx-\int_\Omega
|p_t|^{m_2-1}p_tpdx+\|v(t)\|_{n_1+1}^{n_1+1}+\|p(t)\|_{n_2+1}^{n_2+1}.
\end{align}
 $N''(t)$ can be established formally by differentiating $N'(t)$ in \eqref{261438} and using equations in (\ref{1.1}), however, this formal procedure is not rigorous due to the lack of the regularity of $v_{tt}(t)$ and $p_{tt}(t)$.

 In what follows, let us give the rigorous proof of \eqref{6-9}. Due to $v_0\in H_{\Gamma_0}^1(\Omega) \hookrightarrow L^6(\Omega)$, one has  $v_0\in L^{m_1+1}(\Omega)$ for $1\leq m_1 <5$. Then combining \eqref{6-9-1}, one has $v(t)\in L^{m_1+1}(\Omega\times(0,T))$ for all $0<T<T_{max}$. Likewise, $p(t)\in
L^{m_2+1}(\Omega\times(0,T))$. This implies that it is feasible to
replace $\varphi$ by $v$ in \eqref{171146}, $\psi$ by $p$ in
\eqref{171147} to obtain
\begin{align}\label{6-10}
N'(t)&=\rho\int_\Omega v_0v_1dx+\mu\int_\Omega p_0p_1dx+\rho\int_0^t\|v_t(s)\|_2^2ds+\mu\int_0^t\|p_{t}(s)\|_2^2ds\notag\\
&\quad-\int_0^t(\alpha_1\|\nabla v(s)\|_2^2+\beta\|\gamma\nabla v-\nabla p\|_2^2)ds-\int_0^t\int_\Omega |v_t(s)|^{m_1-1}v_t(s)v(s)dxds\notag\\
&\quad-\int_0^t\int_\Omega
|p_t(s)|^{m_2-1}p_t(s)p(s)dxds+\int_0^t\|v(s)\|_{n_1+1}^{n_1+1}ds+\int_0^t\|p(s)\|_{n_2+1}^{n_2+1}ds.
\end{align}
Let us show that $N'(t)$ is absolutely continuous, and therefore it
can be differentiated. Recall the fact $v\in
C([0,t];H^1_{\Gamma_0}(\Omega))$ and the embedding
$H^1_{\Gamma_0}(\Omega)\hookrightarrow L^6(\Omega)$, thus for all
$t\in[0,T_{max})$,
\begin{equation}
\label{6-11}
\int_0^t\|v(s)\|_{n_1+1}^{n_1+1}ds+\int_0^t\|p(s)\|_{n_2+1}^{n_2+1}ds<\infty,
\end{equation}
where we have used $1<n_1,n_2<5$ by Remark \ref{rmk-p}. By
H\"older's and Young's inequalities,  we deduce that for all
$t\in[0,T_{max})$,
\begin{align}\label{6-14}
&\int_0^t\Big|\int_\Omega |v_t(s)|^{m_1-1}v_t(s)v(s)dx\Big|ds\le \int_0^t\|v_t(s)\|_{m_1+1}^{m_1}\|v(s)\|_{m_1+1}ds\notag\\
&\quad\le
\frac{m_1}{m_1+1}\int_0^t\|v_t(s)\|_{m_1+1}^{m_1+1}ds+\frac{1}{m_1+1}\int_0^t\|v(s)\|_{m_1+1}^{m_1+1}ds<\infty.
\end{align}
Likewise, $\int_0^t|\int_\Omega
|p_t(s)|^{m_2-1}p_t(s)p(s)dx|ds<\infty$. Therefore, it follows from
\eqref{6-11}, \eqref{6-14} and the regularity of $(v,p)$ that all
terms on the right-hand side of (\ref{6-10}) are absolutely
continuous, and thus we can differentiate (\ref{6-10}) to conclude
\eqref{6-9}.

Second, we aim to find a lower bound for $N''(t)$. By using
H\"{o}lder's and Young's inequalities and $n_1>m_1$, \eqref{6-6} and
\eqref{6-7}, then
\begin{align}\label{261057}
&\int_\Omega |v_t|^{m_1-1}v_tvdx\leq \|v\|_{m_1+1}\|v_t\|^{m_1}_{m_1+1} \leq |\Omega|^{\frac{n_1-m_1}{(n_1+1)(m_1+1)}}\|v\|_{n_1+1}\|v_t\|^{m_1}_{m_1+1}\notag\\
&\quad\le
|\Omega|^{\frac{n_1-m_1}{(n_1+1)(m_1+1)}}(\|v\|_{n_1+1}^{n_1+1}+\|p\|_{n_2+1}^{n_2+1})^{\frac{1}{n_1+1}}
\|v_t\|^{m_1}_{m_1+1}\notag\\
&\quad\le
|\Omega|^{\frac{n_1-m_1}{(n_1+1)(m_1+1)}}(\hat{c}G(t))^{\frac{1}{n_1+1}-\frac{1}{m_1+1}}(\|v\|_{n_1+1}^{n_1+1}+\|p\|_{n_2+1}^{n_2+1})^{\frac{1}{m_1+1}}
\|v_t\|^{m_1}_{m_1+1}\notag\\
&\quad \le \varepsilon_1
(G(t))^{\frac{1}{n_1+1}-\frac{1}{m_1+1}}(\|v\|_{n_1+1}^{n_1+1}+\|p\|_{n_2+1}^{n_2+1})
+C(\varepsilon_1)(G(t))^{\frac{1}{n_1+1}-\frac{1}{m_1+1}}
\|v_t\|^{m_1+1}_{m_1+1}\notag\\
&\quad \le \varepsilon_1 (G(0))^{\frac{1}{n_1+1}-\frac{1}{m_1+1}}(\|v\|_{n_1+1}^{n_1+1}+\|p\|_{n_2+1}^{n_2+1})\notag\\
&\qquad
+C(\varepsilon_1)(G(0))^{\frac{1}{n_1+1}-\frac{1}{m_1+1}+\varpi}(G(t))^{-\varpi}
\|v_t\|^{m_1+1}_{m_1+1}\notag\\
&\quad \le \varepsilon_1 (G(0))^{\frac{1}{n_1+1}-\frac{1}{m_1+1}}(\|v\|_{n_1+1}^{n_1+1}+\|p\|_{n_2+1}^{n_2+1})\notag\\
&\qquad+C(\varepsilon_1)(G(0))^{\frac{1}{n_1+1}-\frac{1}{m_1+1}+\varpi}(G(t))^{-\varpi}
G'(t).
\end{align}
Likewise, we get
\begin{align}\label{261058}
\int_\Omega |p_t|^{m_2-1}p_tpdx& \le \varepsilon_2 (G(0))^{\frac{1}{n_2+1}-\frac{1}{m_2+1}}(\|v\|_{n_1+1}^{n_1+1}+\|p\|_{n_2+1}^{n_2+1})\notag\\
&\quad
+C(\varepsilon_2)(G(0))^{\frac{1}{n_2+1}-\frac{1}{m_2+1}+\varpi}(G(t))^{-\varpi}
G'(t).
\end{align}
Inserting \eqref{261057} and \eqref{261058} into \eqref{6-9}, one
arrives at
\begin{align}\label{e3}
N''(t)&\ge\rho\|v_t\|_2^2+\mu\|p_{t}\|_2^2-(\alpha_1\|\nabla v\|_2^2+\beta\|\gamma\nabla v-\nabla p\|_2^2)\notag\\
&\quad-\Big[C(\varepsilon_1)(G(0))^{\frac{1}{n_1+1}-\frac{1}{m_1+1}+\varpi}+C(\varepsilon_2)(G(0))^{\frac{1}{n_2+1}-\frac{1}{m_2+1}+\varpi}\Big](G(t))^{-\varpi}
G'(t)\notag\\
&\quad+\Big[1-\varepsilon_1
(G(0))^{\frac{1}{n_1+1}-\frac{1}{m_1+1}}-\varepsilon_2
(G(0))^{\frac{1}{n_2+1}-\frac{1}{m_2+1}}\Big](\|v\|_{n_1+1}^{n_1+1}+\|p\|_{n_2+1}^{n_2+1}),
\end{align}
for all $t\in [0,T_{max})$. It follows from \eqref{6-7} that
\begin{equation}\label{261123}
\begin{split}
 2G(t)-\frac{2}{\hat{c}}(\|v\|_{n_1+1}^{n_1+1}+\|p\|_{n_2+1}^{n_2+1})+\rho\|v_t\|_2^2+\mu\|p_{t}\|_2^2\le  -(\alpha_1\|\nabla v\|_2^2+\beta\|\gamma\nabla v-\nabla p\|_2^2).
\end{split}
\end{equation}
Inserting \eqref{261123} into \eqref{e3}, one has
\begin{align}\label{261124}
&N''(t)\ge2\rho\|v_t\|_2^2+2\mu\|p_{t}\|_2^2+2G(t)\notag\\
&\quad-\Big[C(\varepsilon_1)(G(0))^{\frac{1}{n_1+1}-\frac{1}{m_1+1}+\varpi}+C(\varepsilon_2)(G(0))^{\frac{1}{n_2+1}-\frac{1}{m_2+1}+\varpi}\Big](G(t))^{-\varpi}
G'(t)\notag\\
&\quad+\Big[1-\varepsilon_1
(G(0))^{\frac{1}{n_1+1}-\frac{1}{m_1+1}}-\varepsilon_2
(G(0))^{\frac{1}{n_2+1}-\frac{1}{m_2+1}}-\frac{2}{\hat{c}}\Big](\|v\|_{n_1+1}^{n_1+1}+\|p\|_{n_2+1}^{n_2+1}),
\end{align}
for all $t\in [0,T_{max})$.  Let us choose
$\varepsilon_1=\frac{\hat{c}-2}{4\hat{c}}(G(0))^{\frac{1}{m_1+1}-\frac{1}{n_1+1}}$
and $\varepsilon_2=
\frac{\hat{c}-2}{4\hat{c}}(G(0))^{\frac{1}{m_2+1}-\frac{1}{n_2+1}}$,
then
\[1-\varepsilon_1 (G(0))^{\frac{1}{n_1+1}-\frac{1}{m_1+1}}-\varepsilon_2 (G(0))^{\frac{1}{n_2+1}-\frac{1}{m_2+1}}-\frac{2}{\hat{c}}=\frac{\hat{c}-2}{2\hat{c}}> 0.\]

Let us combine \eqref{6-8} with \eqref{261124}, then
\begin{align}\label{6-20}
&Y'(t) \geq 2\varepsilon\rho\|v_t\|_2^2+2\varepsilon\mu\|p_{t}\|_2^2+2\varepsilon G(t)+\frac{\hat{c}-2}{2\hat{c}}\varepsilon(\|v\|_{n_1+1}^{n_1+1}+\|p\|_{n_2+1}^{n_2+1})\notag\\
&\quad+\left\{(1-\varpi)-\varepsilon\Big[C(\varepsilon_1)(G(0))^{\frac{1}{n_1+1}-\frac{1}{m_1+1}+\varpi}+C(\varepsilon_2)(G(0))^{\frac{1}{n_2+1}-\frac{1}{m_2+1}+\varpi}\Big]\right\}G'(t)G^{-\varpi}(t).
\end{align}
Recalling $0<\varpi<\frac{1}{2}$ by \eqref{6-4}, we choose
$0<\varepsilon<1$
 small enough such that
$$
C=(1-\varpi)-\varepsilon\Big[C(\varepsilon_1)(G(0))^{\frac{1}{n_1+1}-\frac{1}{m_1+1}+\varpi}+C(\varepsilon_2)(G(0))^{\frac{1}{n_2+1}-\frac{1}{m_2+1}+\varpi}\Big]\geq0,
$$
 to obtain from \eqref{6-20}  that
\begin{align}\label{6-22}
Y'(t)&\geq 2\varepsilon\rho\|v_t\|_2^2+2\varepsilon\mu\|p_{t}\|_2^2+2\varepsilon G(t)+CG'(t)G^{-\varpi}(t)\notag\\
&\quad+\frac{\hat{c}-2}{2\hat{c}}\varepsilon(\|v\|_{n_1+1}^{n_1+1}+\|p\|_{n_2+1}^{n_2+1})
>0,
\end{align}
which indicates that $Y(t)$ is increasing on $[0,T_{max})$, with
$$
Y(t)=G^{1-\varpi}(t)+\varepsilon N'(t) > Y(0) =
G^{1-\varpi}(0)+\varepsilon N'(0).
$$
If $N'(0)\geq0$, then we do not need any  further condition on
$\varepsilon$. However,  if $N'(0)<0$, we need to take $\varepsilon$
such that $0<\varepsilon\leq-\frac{G^{1-\varpi}(0)}{2N'(0)}$. In
either case, we obtain
\begin{align}\label{6-23}
Y(t)\geq \frac{1}{2}G^{1-\varpi}(0)>0,\ \ \mbox{for}\
t\in[0,T_{max}).
\end{align}

Finally, let us verify that
\begin{align}\label{6-24}
Y'(t)\geq C\varepsilon^{1+\sigma}Y^{\frac{1}{1-\varpi}}(t),\ \
\mbox{for}\ t\in[0,T_{max}),
\end{align}
where $C>0$ is a generic constant independent of $\varepsilon$, and
$\sigma=\max\{\sigma_1,\sigma_2\}>0$ with
\begin{equation}
\label{261447} \sigma_1=1-\frac{2}{(1-2\varpi)(n_1+1)}>0,\quad
\sigma_2=1-\frac{2}{(1-2\varpi)(n_2+1)}>0,
\end{equation}
due to (\ref{6-4}). Indeed, if $N'(t)\leq 0$ for some
$t\in[0,T_{max})$, then
\begin{align}\label{6-25}
Y^{\frac{1}{1-\varpi}}(t)=[G^{1-\varpi}(t)+\varepsilon
N'(t)]^{\frac{1}{1-\varpi}}\le G(t).
\end{align}
Therefore, \eqref{6-22} and \eqref{6-25} give
$$
Y'(t)\geq 2\varepsilon G(t)\geq 2\varepsilon^{1+\sigma}G(t)\geq
2\varepsilon^{1+\sigma}Y^{\frac{1}{1-\varpi}}(t).
$$
If $N'(t)> 0$ for some $t\in[0,T_{max})$, note that
$Y(t)=G^{1-\varpi}(t)+\varepsilon N'(t)\leq G^{1-\varpi}(t)+N'(t)$,
then
\begin{align}\label{6-26}
Y^{\frac{1}{1-\varpi}}(t)\leq C
\Big[G(t)+[N'(t)]^{\frac{1}{1-\varpi}}\Big].
\end{align}
Applying H\"{o}lder's and Young's inequality, and using
$1<\frac{1}{1-\varpi}<2$, we conclude from \eqref{261438} that
\begin{align}\label{6-27}
[N'(t)]^{\frac{1}{1-\varpi}}&\leq \Big(\rho\|v_t\|_2\|v\|_2+\mu\|p_t\|_2\|p\|_2\Big)^{\frac{1}{1-\varpi}}\notag\\
&\leq C\Big(\|v_t\|^{\frac{1}{1-\varpi}}_2\|v\|^{\frac{1}{1-\varpi}}_2+\|p_t\|^{\frac{1}{1-\varpi}}_2\|p\|^{\frac{1}{1-\varpi}}_2\Big)\notag\\
&\leq
C\Big(\|v_t\|^2_2+\|v\|^{\frac{2}{1-2\varpi}}_{n_1+1}+\|p_t\|^2_2+\|p\|^{\frac{2}{1-2\varpi}}_{n_2+1}\Big).
\end{align}
Recalling $\varepsilon\leq G(0)$, then from \eqref{261447},
\eqref{6-7}, \eqref{6-4} and \eqref{261447}, one has
\begin{align}\label{6-29}
\|v\|^{\frac{2}{1-2\varpi}}_{n_1+1}&\le(\|v\|^{n_1+1}_{n_1+1}+\|p\|^{n_2+1}_{n_2+1})^{\frac{2}{(1-2\varpi)(n_1+1)}}\notag\\
&= (\|v\|_{n_1+1}^{n_1+1}+\|p\|^{n_2+1}_{n_2+1})^{\frac{2}{(1-2\varpi)(n_1+1)}-1}(\|v\|_{n_1+1}^{n_1+1}+\|p\|^{n_2+1}_{n_2+1})\notag\\
&\leq
CG(t)^{\frac{2}{(1-2\varpi)(n_1+1)}-1}(\|v\|_{n_1+1}^{n_1+1}+\|p\|^{n_2+1}_{n_2+1})
\notag\\&\leq
CG^{-\sigma_1}(0)(\|v\|_{n_1+1}^{n_1+1}+\|p\|^{n_2+1}_{n_2+1})\notag\\
&\leq
C\varepsilon^{-\sigma_1}(\|v\|_{n_1+1}^{n_1+1}+\|p\|^{n_2+1}_{n_2+1}).
\end{align}
In the same way, we have
\begin{align}\label{6-30}
\|p\|^{\frac{2}{1-2\varpi}}_{n_2+1} \leq
C\varepsilon^{-\sigma_2}(\|v\|_{n_1+1}^{n_1+1}+\|p\|^{n_2+1}_{n_2+1}).
\end{align}
Recall $\sigma=\max\{\sigma_1,\sigma_2\}>0$ and
$\varepsilon^{-\sigma}>1$. By substituting \eqref{6-29} and
\eqref{6-30} into \eqref{6-27}, we get
\begin{align}\label{6-31}
[N'(t)]^{\frac{1}{1-\varpi}}&\leq C\Big(\|v_t\|^2_2+\|p_t\|^2_2+\varepsilon^{-\sigma}(\|v\|_{n_1+1}^{n_1+1}+\|p\|^{n_2+1}_{n_2+1})\Big)\notag\\
&\leq
C\varepsilon^{-\sigma}\Big(\|v_t\|^2_2+\|p_t\|^2_2+\|v\|_{n_1+1}^{n_1+1}+\|p\|^{n_2+1}_{n_2+1}\Big).
\end{align}
Combining  \eqref{6-26}, \eqref{6-31} and \eqref{6-22}, then
\begin{equation*}
\begin{split}
Y^{\frac{1}{1-\varpi}}(t)&\leq C \Big[G(t)+[N'(t)]^{\frac{1}{1-\varpi}}\Big]\\
&\le C\varepsilon^{-\sigma}\Big[G(t)+\|v_t\|^2_2+\|p_t\|^2_2+\|v\|_{n_1+1}^{n_1+1}+\|p\|^{n_2+1}_{n_2+1}\Big]\\
&\le C\varepsilon^{-1-\sigma}Y'(t),
\end{split}
\end{equation*}
for all $t\in[0,T_{max})$ if $N'(t)>0$. Then in either case,
\eqref{6-24} holds true.

It follows  from \eqref{6-23} and \eqref{6-24} that the maximum life
span $T_{max}$ is definitely finite with
\begin{align*}
T_{max}<C\varepsilon^{-(1+\sigma)}Y^{-\frac{\varpi}{1-\varpi}}(0)\leq
C\varepsilon^{-(1+\sigma)}G^{-\varpi}(0),
\end{align*}
which implies the quadratic energy must go to infinity, that is,
\begin{align}\label{6-33}
\limsup_{t\rightarrow T_{max}^-} E(t) = +\infty.
\end{align}
It follows from (\ref{6-7}) that
\begin{align}  \label{6-35}
E(t) &= - G(t) +\frac{1}{n_1+1}\|v\|_{n_1+1}^{n_1+1}+\frac{1}{n_2+1}\|p\|_{n_2+1}^{n_2+1}\notag\\
& <
\frac{1}{n_1+1}\|v\|_{n_1+1}^{n_1+1}+\frac{1}{n_2+1}\|p\|_{n_2+1}^{n_2+1},
\end{align}
due to $G(t)>0$ on $[0,T_{max})$. Therefore,  (\ref{6-33}) and
(\ref{6-35}) illustrate
\begin{align} \label{6-36}
\limsup_{t\rightarrow T_{max}^-}
(\|v(t)\|_{n_1+1}^{n_1+1}+\|p\|_{n_2+1}^{n_2+1})  = + \infty.
\end{align}
Recall $n_1,~n_2<5$, embedding theorem \eqref{12171637} and
\eqref{12161346}, then \begin{equation*}
\begin{split}
&\|v\|_{n_1+1}^{n_1+1}+\|p\|_{n_2+1}^{n_2+1}
\le B_1\|\nabla v\|_2^{n_1+1}+B_2\|\nabla p\|_2^{n_2+1}\\
&\quad\le \hat{C}\Big[(\alpha_1\|\nabla v\|_2^2+\beta\|\gamma\nabla
v-\nabla p\|_2^2)^{\frac{n_1+1}{2}}+(\alpha_1\|\nabla
v\|_2^2+\beta\|\gamma\nabla v-\nabla
p\|_2^2)^{\frac{n_2+1}{2}}\Big].
\end{split}
\end{equation*}
This combines with (\ref{6-36}), one obtains
\begin{align*}
\limsup_{t\rightarrow T_{max}^-} (\alpha_1\|\nabla
v\|_2^2+\beta\|\gamma\nabla v-\nabla p\|_2^2) = +\infty.
\end{align*}
The proof is completed.

\subsection{Blow-up in finite time for bounded positive initial energy}
Let us still define the life span $T_{max}$ of the solution $(v,p)$
to be the supremum of all $T>0$ such that $(v,p)$ is a solution to
problem \eqref{1.1} in the sense of Definition \ref{def1} on $[0,
T]$. In the following, we will show that $T_{max}$ is finite and
obtain an upper bound for the life span of solutions.

By using \eqref{231213} and embedding theorem \eqref{12171637} and
\eqref{12161346}, we have that for $t\in [0,T_{max}),$
\begin{align}\label{b8}
&\mathcal{E}(t)=E(t)-\frac{1}{n_1+1}\|v\|_{n_1+1}^{n_1+1}-\frac{1}{n_2+1}\|p\|_{n_2+1}^{n_2+1}\notag\\
&\geq E(t)-\frac{\hat{C}}{n_1+1}\left(\alpha_1\|\nabla v\|_2^2+\beta\|\gamma\nabla v-\nabla p\|_2^2\right)^{\frac{n_1+1}{2}}\notag\\
&\quad-\frac{\hat{C}}{n_2+1}\left(\alpha_1\|\nabla v\|_2^2+\beta\|\gamma\nabla v-\nabla p\|_2^2\right)^{\frac{n_2+1}{2}}\notag\\
&\geq E(t)-\frac{\hat{C}}{n_1+1}(2E(t))^{\frac{n_1+1}{2}}-\frac{\hat{C}}{n_2+1}(2E(t))^{\frac{n_2+1}{2}},\notag\\
&= \Psi(E(t)).
\end{align}
Here we define the function $\Psi:\mathbb{R}^+\to\mathbb{R}$ by
\begin{align*}
\Psi(y):=y-\frac{\hat{C}}{n_1+1}(2y)^{\frac{n_1+1}{2}}-\frac{\hat{C}}{n_2+1}(2y)^{\frac{n_2+1}{2}}.
\end{align*}
It is clear to see that $\Psi(y)$ is continuously differentiable,
concave and has its maximum at $y=y_0>0$, where $y_0$ satisfies
\begin{align}\label{b2}
\hat{C}(2y_0)^{\frac{n_1-1}{2}}+\hat{C}(2y_0)^{\frac{n_2-1}{2}}=1.
\end{align}
Define
\begin{align}\label{b3}
\hat{d}:=\sup_{[0,\infty)}\Psi(y)=\Psi(y_0)=y_0-\frac{\hat{C}}{n_1+1}(2y_0)^{\frac{n_1+1}{2}}-\frac{\hat{C}}{n_2+1}(2y_0)^{\frac{n_2+1}{2}}.
\end{align}

Let us stop here to give the following claim.

\textbf{Claim.} $ 0<\hat{d}\leq d.$

Let us prove this claim. By using (\ref{b2}) and (\ref{b3}), we have
\begin{align*}
\hat{d}&=y_0-\frac{\hat{C}}{n_1+1}(2y_0)^{\frac{n_1+1}{2}}-\frac{\hat{C}}{n_2+1}(2y_0)^{\frac{n_2+1}{2}}\\
&\geq y_0-\frac{2}{\hat{c}}y_0\left[\hat{C}(2y_0)^{\frac{n_1-1}{2}}+\hat{C}(2y_0)^{\frac{n_2-1}{2}}\right]\\
&=y_0-\frac{2}{\hat{c}}y_0=\frac{\hat{c}-2}{\hat{c}}y_0,
\end{align*}
which implies $\hat{d}>0$.  Recall \eqref{271321}, \eqref{271318}
and the maximum value piont $s^*$ of $\Lambda(s)$, then
\[\Lambda(s^*):=\frac{1}{2}s^*-\frac{\hat{C}}{n_1+1}(s^*)^{\frac{n_1+1}{2}}-\frac{\hat{C}}{n_2+1}(s^*)^{\frac{n_2+1}{2}},\]
\[\Lambda'(s^*)=\frac12-\frac{\hat{C}}{2}(s^*)^{\frac{n_1-1}{2}}-\frac{\hat{C}}{2}(s^*)^{\frac{n_2-1}{2}}=0.\]
Let us choose $s^*=2y_0$, then
\begin{align} \label{b6}
\Lambda(s^*)=\Lambda(2y_0)=
y_0-\frac{\hat{C}}{n_1+1}(2y_0)^{\frac{n_1+1}{2}}-\frac{\hat{C}}{n_2+1}(2y_0)^{\frac{n_2+1}{2}}=\hat{d}.
\end{align}
From \eqref{171709}, we obtain
$$
\mathcal{J}(\lambda (v,p))\geq \Lambda(\lambda(\alpha_1\|\nabla
v\|_2^2+\beta\|\gamma\nabla v-\nabla p\|_2^2)),\ \mbox{for}\
\mbox{all}\ \lambda\geq0.
$$
It follows that $ \sup_{\lambda\geq0}\mathcal{J}(\lambda (u,w))\geq
\Lambda(s^*).$ Then we infer from \eqref{171711} and \eqref{b6} that
$$
d=\inf_{(v,p)\in
\mathcal{V}\backslash(0,0)}\sup_{\lambda\geq0}\mathcal{J}(\lambda
(v,p))\geq \Lambda(s^*)=\hat{d}.
$$
This shows that $\hat{d}$ is not larger than the depth $d$ of the
potential well.

Now, let us go back to our proof.  Since $\Psi(y)$ is decreasing if
$y>y_0$, for  $0\leq\mathcal{E}(0)<\hat{d}=\Psi(y_0)$,  there exists
a unique constant $y_1$ satisfying $y_1>y_0>0$ such that
\begin{align*}
\Psi(y_1)=\mathcal{E}(0).
\end{align*}
Then it follows from \eqref{b8} that
\begin{align}\label{b11}
\hat{d}=\Psi(y_0)>\Psi(y_1)=\mathcal{E}(0)\geq \mathcal{E}(t)\geq
\Psi(E(t)),\quad \text{for }t\in [0,T_{max}).
\end{align}
Note that $\Psi(y)$ is continuous and decreasing if $y>y_0$, and
$\mathcal{E}(t)$ is also continuous. Since we assume $E(0)>y_0$, it
follows from \eqref{b11} that
\begin{align}\label{b12}
E(t)\geq y_1>y_0,\quad \text{for }t\in [0,T_{max}).
\end{align}

Let us define
\begin{align}    \label{defGG}
\mathcal G(t) := \mathcal{M}-\mathcal{E}(t).
\end{align}
It follows from $\mathcal{M}=\frac{\hat{c}-2}{2(2+\hat{c})}y_0<y_0$
and \eqref{b12} that
\begin{align}
\label{271149}
\mathcal G(t) &:= \mathcal{M}-\mathcal{E}(t)\notag\\
&<\frac{\hat{c}-2}{2(2+\hat{c})}y_0+\frac{1}{n_1+1}\|v\|_{n_1+1}^{n_1+1}+\frac{1}{n_2+1}\|p\|_{n_2+1}^{n_2+1}-y_0\notag\\
&\le \frac{1}{\hat{c}}(\|v\|_{n_1+1}^{n_1+1}+\|p\|_{n_2+1}^{n_2+1}).
\end{align}
Due to $\mathcal E'(t)\leq 0$, thus $\mathcal G'(t)\geq 0$, i.e.,
$\mathcal G(t)$ is non-decreasing in time. Recall $\mathcal
E(0)<\mathcal{M}$, then $\mathcal G(0) :=
\mathcal{M}-\mathcal{E}(0)>0$. Therefore, we have
\begin{align} \label{GG0}
\mathcal G(t) \geq  \mathcal G(0) >0,  \quad \text{for }t\in
[0,T_{max}).
\end{align}

Let us consider the function
\begin{align*}
Y(t):=\mathcal G^{1-\varpi}(t)+\varepsilon N'(t),
\end{align*}
for some $\varpi\in (0,\frac{1}{2})$ satisfying (\ref{6-4}) and
$\varepsilon>0$. Our aim is to show that $Y(t)$ approaches infinity
in finite time, by choosing $\varepsilon$ sufficiently small.
Adopting the same arguments as \eqref{6-8}, let us estimate
\begin{align}\label{b15}
Y'(t)=(1-\varpi)\mathcal G^{-\varpi}(t)\mathcal G'(t)+\varepsilon
N''(t).
\end{align}

Due to \eqref{271149}, similar to \eqref{261057} and \eqref{261058},
we
\begin{align}\label{261618}
\int_\Omega |v_t|^{m_1-1}v_tvdx&\le \varepsilon_1
(\mathcal{G}(0))^{\frac{1}{n_1+1}-\frac{1}{m_1+1}}(\|v\|_{n_1+1}^{n_1+1}+\|p\|_{n_2+1}^{n_2+1})
\notag\\
&\quad+C(\varepsilon_1)(\mathcal{G}(0))^{\frac{1}{n_1+1}-\frac{1}{m_1+1}+\varpi}(\mathcal{G}(t))^{-\varpi}
\mathcal{G}'(t),
\end{align}
and
\begin{align}\label{261620}
\int_\Omega |p_t|^{m_2-1}p_tpdx &\le \varepsilon_2
(\mathcal{G}(0))^{\frac{1}{n_2+1}-\frac{1}{m_2+1}}(\|v\|_{n_1+1}^{n_1+1}+\|p\|_{n_2+1}^{n_2+1})
\notag\\
&\quad+C(\varepsilon_2)(\mathcal{G}(0))^{\frac{1}{n_2+1}-\frac{1}{m_2+1}+\varpi}(\mathcal{G}(t))^{-\varpi}
\mathcal{G}'(t).
\end{align}
Therefore, inserting \eqref{261618} and \eqref{261620} into
\eqref{6-9}, we obtain
\begin{align}\label{261630}
N''(t)&\ge\rho\|v_t\|_2^2+\mu\|p_{t}\|_2^2-(\alpha_1\|\nabla v\|_2^2+\beta\|\gamma\nabla v-\nabla p\|_2^2)\notag\\
&\quad-\Big[C(\varepsilon_1)(\mathcal{G}(0))^{\frac{1}{n_1+1}-\frac{1}{m_1+1}+\varpi}+C(\varepsilon_2)(\mathcal{G}(0))^{\frac{1}{n_2+1}-\frac{1}{m_2+1}+\varpi}\Big](\mathcal{G}(t))^{-\varpi}
\mathcal{G}'(t)\notag\\
&\quad+\Big[1-\varepsilon_1
(\mathcal{G}(0))^{\frac{1}{n_1+1}-\frac{1}{m_1+1}}-\varepsilon_2
(\mathcal{G}(0))^{\frac{1}{n_2+1}-\frac{1}{m_2+1}}\Big](\|v\|_{n_1+1}^{n_1+1}+\|p\|_{n_2+1}^{n_2+1}).
\end{align}

It follows from  \eqref{231213} and \eqref{defGG} that
\begin{align}\label{262632}
 &2(\mathcal{G}(t)-\mathcal{M})-\frac{2}{\hat{c}}(\|v\|_{n_1+1}^{n_1+1}+\|p\|_{n_2+1}^{n_2+1})+\rho\|v_t\|_2^2+\mu\|p_{t}\|_2^2\notag\\
 &\quad\le  -(\alpha_1\|\nabla v(t)\|_2^2+\beta\|\gamma\nabla v-\nabla p\|_2^2).
\end{align}
Inserting \eqref{262632} into \eqref{261630}, we get
\begin{align}\label{261635}
N''(t)&\ge2\rho\|v_t\|_2^2+2\mu\|p_{t}\|_2^2+2\mathcal{G}(t)-2\mathcal{M}+\frac{\hat{c}-2}{\hat{c}}(\|v\|_{n_1+1}^{n_1+1}+\|p\|_{n_2+1}^{n_2+1})\notag\\
&\quad-\Big[C(\varepsilon_1)(\mathcal{G}(0))^{\frac{1}{n_1+1}-\frac{1}{m_1+1}+\varpi}+C(\varepsilon_2)(\mathcal{G}(0))^{\frac{1}{n_2+1}-\frac{1}{m_2+1}+\varpi}\Big](\mathcal{G}(t))^{-\varpi}
\mathcal{G}'(t)\notag\\
&\quad-\Big[\varepsilon_1 (\mathcal{G}(0))^{\frac{1}{n_1+1}+\frac{1}{m_1+1}}+\varepsilon_2 (\mathcal{G}(0))^{\frac{1}{n_2+1}-\frac{1}{m_2+1}}\Big](\|v\|_{n_1+1}^{n_1+1}+\|p\|_{n_2+1}^{n_2+1})\notag\\
&\ge2\rho\|v_t\|_2^2+2\mu\|p_{t}\|_2^2+2\mathcal{G}(t)-2\mathcal{M}+\frac{\hat{c}-2}{2}(E(t)+\mathcal{G}(t)-\mathcal{M})\notag\\
&\quad-\Big[C(\varepsilon_1)(\mathcal{G}(0))^{\frac{1}{n_1+1}-\frac{1}{m_1+1}+\varpi}+C(\varepsilon_2)(\mathcal{G}(0))^{\frac{1}{n_2+1}-\frac{1}{m_2+1}+\varpi}\Big](\mathcal{G}(t))^{-\varpi}
\mathcal{G}'(t)\notag\\
&\quad+\Big[\frac{\hat{c}-2}{2\hat{c}}-\varepsilon_1
(\mathcal{G}(0))^{\frac{1}{n_1+1}+\frac{1}{m_1+1}}-\varepsilon_2
(\mathcal{G}(0))^{\frac{1}{n_2+1}-\frac{1}{m_2+1}}\Big](\|v\|_{n_1+1}^{n_1+1}+\|p\|_{n_2+1}^{n_2+1})
\notag\\
&\ge2\rho\|v_t\|_2^2+2\mu\|p_{t}\|_2^2+\Big(2+\frac{\hat{c}-2}{2}\Big)\mathcal{G}(t)+\frac{\hat{c}-2}{4}E(t)\notag\\
&\quad-\Big[C(\varepsilon_1)(\mathcal{G}(0))^{\frac{1}{n_1+1}-\frac{1}{m_1+1}+\varpi}+C(\varepsilon_2)(\mathcal{G}(0))^{\frac{1}{n_2+1}-\frac{1}{m_2+1}+\varpi}\Big](\mathcal{G}(t))^{-\varpi}
\mathcal{G}'(t)\notag\\
&\quad+\Big[\frac{\hat{c}-2}{2\hat{c}}-\varepsilon_1
(\mathcal{G}(0))^{\frac{1}{n_1+1}+\frac{1}{m_1+1}}-\varepsilon_2
(\mathcal{G}(0))^{\frac{1}{n_2+1}-\frac{1}{m_2+1}}\Big](\|v\|_{n_1+1}^{n_1+1}+\|p\|_{n_2+1}^{n_2+1}),
\end{align}
where we have used
\[\frac{\hat{c}-2}{4}E(t)>\frac{\hat{c}-2}{4}y_0=\Big(2+\frac{\hat{c}-2}{2}\Big)\mathcal{M},\]
and
\begin{align*}
\frac{1}{\hat{c}}(\|v\|_{n_1+1}^{n_1+1}+\|p\|_{n_2+1}^{n_2+1})&\ge \frac{1}{n_1+1}\|v\|_{n_1+1}^{n_1+1}+\frac{1}{n_2+1}\|p\|_{n_2+1}^{n_2+1}\notag\\
&=E(t)-\mathcal{E}(t)=E(t)+\mathcal{G}(t)-\mathcal{M}
\end{align*}
from \eqref{231213}. Combining \eqref{b15} with \eqref{261635}, one
has
\begin{align}\label{261652}
&Y'(t)\ge \varepsilon 2\rho\|v_t\|_2^2+\varepsilon2\mu\|p_{t}\|_2^2+\varepsilon\Big(2+\frac{\hat{c}-2}{2}\Big)\mathcal{G}(t)+\varepsilon\frac{\hat{c}-2}{4}E(t)\notag\\
&\quad+\Big\{1-\varpi-\varepsilon\Big[C(\varepsilon_1)(\mathcal{G}(0))^{\frac{1}{n_1+1}-\frac{1}{m_1+1}+\varpi}+C(\varepsilon_2)(\mathcal{G}(0))^{\frac{1}{n_2+1}-\frac{1}{m_2+1}+\varpi}\Big]\Big\}(\mathcal{G}(t))^{-\varpi}
\mathcal{G}'(t)\notag\\
&\quad+\varepsilon\Big[\frac{\hat{c}-2}{2\hat{c}}-\varepsilon_1
(\mathcal{G}(0))^{\frac{1}{n_1+1}+\frac{1}{m_1+1}}-\varepsilon_2
(\mathcal{G}(0))^{\frac{1}{n_2+1}-\frac{1}{m_2+1}}\Big](\|v\|_{n_1+1}^{n_1+1}+\|p\|_{n_2+1}^{n_2+1}).
\end{align}

At this point, we select $\varepsilon_1,\varepsilon_2>0$ such that
$$
\frac{\hat{c}-2}{2\hat{c}}-\varepsilon_1
(\mathcal{G}(0))^{\frac{1}{n_1+1}-\frac{1}{m_1+1}}-\varepsilon_2
(\mathcal{G}(0))^{\frac{1}{n_2+1}-\frac{1}{m_2+1}}\geq
\frac{\hat{c}-2}{4\hat{c}}.
$$
For these fixed values of $\varepsilon_1,\varepsilon_2>0$, we choose
$\varepsilon>0$ sufficiently small that
$$
1-\varpi-\varepsilon\Big[C(\varepsilon_1)(\mathcal{G}(0))^{\frac{1}{n_1+1}-\frac{1}{m_1+1}+\varpi}+C(\varepsilon_2)(\mathcal{G}(0))^{\frac{1}{n_2+1}-\frac{1}{m_2+1}+\varpi}\Big]\geq
\frac{1-\varpi}{2}.
$$
Then, from \eqref{261652} and \eqref{271149} we obtain
\begin{align}\label{b31}
Y'(t)&\geq  \varepsilon 2\rho\|v_t\|_2^2+\varepsilon2\mu\|p_{t}\|_2^2+\varepsilon\Big(2+\frac{\hat{c}-2}{2}\Big)\mathcal{G}(t)+\varepsilon\frac{\hat{c}-2}{4}E(t)\notag\\
&\quad+\frac{1-\varpi}{2}(\mathcal{G}(t))^{-\varpi}
\mathcal{G}'(t)+\frac{\hat{c}-2}{4\hat{c}}(\|v(t)\|_{n_1+1}^{n_1+1}+\|p(t)\|_{n_2+1}^{n_2+1})>0,
\end{align}
for all $t\in [0,T_{max})$. Therefore, $Y(t)$ is increasing on
$[0,T_{max})$, with
$$
Y(t)=\mathcal G^{1-\varpi}(t)+\varepsilon N'(t) > Y(0) = \mathcal
G^{1-\varpi}(0)+\varepsilon N'(0).
$$
Similar to \eqref{6-23}, one can choose $\varepsilon$ sufficiently
small such that
\begin{align}\label{b33}
Y(t) \geq \frac{1}{2} \mathcal G^{1-\varpi}(0)>0,\quad \text{for
}t\in [0,T_{max}).
\end{align}
Now, we claim
\begin{align}\label{nODE}
Y'(t)\geq C\varepsilon^{1+\sigma}Y^{\frac{1}{1-\varpi}}(t),\quad
\text{for }t\in [0,T_{max}),
\end{align}
where $\sigma:=\max\{\sigma_1,\sigma_2\}>0$ is shown in
\eqref{261447}.  By solving differential inequality (\ref{nODE})
with (\ref{b33}), we deduce that the maximum life span $T_{max}$ is
necessarily finite with
\begin{align*}
T_{max}<C\varepsilon^{-(1+\sigma)}Y^{-\frac{\varpi}{1-\varpi}}(0)\leq
C\varepsilon^{-(1+\sigma)}\mathcal G^{-\varpi}(0).
\end{align*}

To prove (\ref{nODE}), we use the following argument. If $N'(t)\leq
0$ for some $t\in[0,T_{max})$, then for such value of $t$, we get
\begin{align}  \label{b34}
Y^{\frac{1}{1-\varpi}}(t)=[\mathcal G^{1-\varpi}(t)+\varepsilon
N'(t)]^{\frac{1}{1-\varpi}}\leq \mathcal G(t).
\end{align}
Thus, we infer from \eqref{b31} and \eqref{b34} that
$$
Y'(t)\geq \Big(2+\frac{\hat{c}-2}{2}\Big)\varepsilon \mathcal
G(t)\geq
\Big(2+\frac{\hat{c}-2}{2}\Big)\varepsilon^{1+\sigma}\mathcal
G(t)\geq
\Big(2+\frac{\hat{c}-2}{2}\Big)\varepsilon^{1+\sigma}Y^{\frac{1}{1-\varpi}}(t),
$$
for any value of $t$ such that $N'(t)\leq 0$.

If $N'(t)> 0$ for some $t\in[0,T_{max})$, then
\begin{align}\label{b35}
Y^{\frac{1}{1-\varpi}}(t)\leq C \Big[\mathcal
G(t)+[N'(t)]^{\frac{1}{1-\varpi}}\Big].
\end{align}

We know that $\|v\|_{n_1+1}^{n_1+1}+\|p\|_{n_2+1}^{n_2+1}> \mathcal
G(t)\geq \mathcal G(0)>0$ by (\ref{271149}) and (\ref{GG0}). Let
$\varepsilon \leq \mathcal G(0)$. Then, following estimates
\eqref{6-27}-\eqref{6-30}, we can derive
\begin{align}  \label{b36}
[N'(t)]^{\frac{1}{1-\varpi}}\leq
C\varepsilon^{-\sigma}\Big(\|v_t\|^2_2+\|p_t\|^2_2+\|v\|_{n_1+1}^{n_1+1}+\|p\|^{n_2+1}_{n_2+1}\Big).
\end{align}
Combining \eqref{b31}, \eqref{b36} and (\ref{b35}), we arrive at
\begin{align*}
Y'(t)&\geq C\varepsilon\Big[\|v_t\|^2_2+\|p_t\|^2_2+\|v\|_{n_1+1}^{n_1+1}+\|p\|^{n_2+1}_{n_2+1}\Big]\\
&\geq C\varepsilon\Big[\mathcal G(t)+\varepsilon^\sigma
[N'(t)]^{\frac{1}{1-\varpi}}\Big] \geq
C\varepsilon^{1+\sigma}\Big[\mathcal G(t)+
[N'(t)]^{\frac{1}{1-\varpi}}\Big]\geq
C\varepsilon^{1+\sigma}Y^{\frac{1}{1-\varpi}}(t),
\end{align*}
for any value of $t$ such that $N'(t)> 0$. As a result, we conclude
that (\ref{nODE}) holds for all values of $t\in [0,T_{max})$.

Finally, we conclude $\limsup_{t \to T_{max}^-} (\alpha_1\|\nabla
v\|_2^2+\beta\|\gamma\nabla v-\nabla p\|_2^2)= +\infty$ by using the
same argument as in Section \ref{sec6.1}. This completes the proof.

In the proof of Theorem $\ref{thmg}(ii)$, we verify the invariance
of $\mathcal{W}_1$ under the dynamics. In fact, using the similar
proof,  we have the same result for $\mathcal{W}_2$ as follows.
\begin{lemma} \label{lem3-2-1}
Let $1<n_1,n_2\leq 5$ and $m_1,m_2>1$ and Assumption $\ref{ass}$
hold. Assume further $(v_0,p_0)\in\mathcal{W}_2$ and
$\mathcal{E}(0)<d$. Then the weak solution $(v,p)$ is in
$\mathcal{W}_2$ for all $t\in [0,T_{max})$, where $[0,T_{max})$ is
the maximal interval of existence.
\end{lemma}

For initial values coming from the unstable part $\mathcal W_2$ of
the potential well,
  a blow-up result is shown in Corollary \ref{cor1}. Now, let us prove Corollary \ref{cor1}.

\begin{proofcor1}
It suffices to show that if $(v_0,p_0)\in\mathcal{W}_2$, then
$E(0)>y_0$.

Since $(v_0,p_0)\in\mathcal{W}_2$, then by the definition of
$\mathcal{W}_2$ and embedding theorem \eqref{12161346}, we get
\begin{align}
\label{271229}
&\alpha_1\|\nabla v_0\|_2^2+\beta\|\gamma\nabla v_0-\nabla p_0\|_2^2<\|v_0\|_{n_1+1}^{n_1+1}+\|p_0\|_{n_2+1}^{n_2+1}\le B_1\|\nabla v_0\|_{2}^{n_1+1}+B_2\|p_0\|_{2}^{n_2+1}\notag\\
&\quad \le \hat{C}\Big[(\alpha_1\|\nabla
v_0\|_2^2+\beta\|\gamma\nabla v_0-\nabla
p_0\|_2^2)^{\frac{n_1+1}{2}}+(\alpha_1\|\nabla
v_0\|_2^2+\beta\|\gamma\nabla v_0-\nabla
p_0\|_2^2)^{\frac{n_2+1}{2}}\Big].
\end{align}
Let us divide both sides of \eqref{271229} by $\alpha_1\|\nabla
v_0\|_2^2+\beta\|\gamma\nabla v_0-\nabla p_0\|_2^2$ to arrive at
$$
 \hat{C}\Big[(\alpha_1\|\nabla v_0\|_2^2+\beta\|\gamma\nabla v_0-\nabla p_0\|_2^2)^{\frac{n_1-1}{2}}+(\alpha_1\|\nabla v_0\|_2^2+\beta\|\gamma\nabla v_0-\nabla p_0\|_2^2)^{\frac{n_2-1}{2}}\Big]>1.
$$
This together with \eqref{b2} gives
\begin{align*}
& \hat{C}\Big[(\alpha_1\|\nabla v_0\|_2^2+\beta\|\gamma\nabla v_0-\nabla p_0\|_2^2)^{\frac{n_1-1}{2}}+(\alpha_1\|\nabla v_0\|_2^2+\beta\|\gamma\nabla v_0-\nabla p_0\|_2^2)^{\frac{n_2-1}{2}}\Big]\\
&\quad>1=\hat{C}(2y_0)^{\frac{n_1-1}{2}}+\hat{C}(2y_0)^{\frac{n_2-1}{2}}.
\end{align*}
Therefore, we have
\begin{align*}
\alpha_1\|\nabla v_0\|_2^2+\beta\|\gamma\nabla v_0-\nabla
p_0\|_2^2>2y_0,
\end{align*}
which implies that $E(0)>y_0$. Then, using Theorem \ref{thm6-2}, we
obtain the blow-up of weak solutions in finite time.
\end{proofcor1}

\subsection{Blow-up in finite time for arbitrarily positive initial energy}
In this section, we wish to prove the blow-up results for
arbitrarily positive initial energy by the Levine's concavity
method. First, let us give the following lemma.
 \begin{lemma}[\cite{L1973,LG2017}]\label{lem2.5}
Suppose a positive, twice-differentiable function $\psi (t)$
satisfies the inequality
\begin{equation*}
\psi''(t)\psi(t)-(1+\theta)({\psi'}(t))^2\ge 0,
\end{equation*}
where $\theta>0.$ If $\psi(0)>0,~\psi'(0)>0$, then
$\psi(t)\to\infty$ as $t\to t_1\le
t_2=\frac{\psi(0)}{\theta\psi'(0)}$.
\end{lemma}

\begin{lemma}\label{lem2.6}
Let Assumption $\ref{ass}(3)$ hold, and $n_1>1,~n_2>1$. If the
initial energy fulfills
\begin{equation}
\label{8165}
0<\mathcal{E}(0)<\frac{\bar{C}}{\hat{c}}\Big(\rho\int_\Omega
v_0v_1dx+\mu\int_\Omega p_0p_1dx\Big),
\end{equation}
then the weak solution $u$ to problem \eqref{1.1} satisfies
\begin{equation}
\label{8167} \rho\int_\Omega vv_tdx+\mu\int_\Omega
pp_tdx-\frac{\hat{c}}{\bar{C}}\mathcal{E}(t)\ge \Big(\rho\int_\Omega
v_0v_1dx+\mu\int_\Omega
p_0p_1dx-\frac{\hat{c}}{\bar{C}}\mathcal{E}(0)\Big)e^{\bar{C}t}>0
\end{equation}
for any $t\in [0,T_{max}),$ where
\begin{align}
\label{8168}
\bar{C}&=\min\Big\{\frac{\hat{c}+2\rho}{\rho},\frac{\hat{c}+2\mu}{\mu},\frac{\hat{c}-2}{2}\Big(\max\Big\{\frac{2\gamma^2+1}{\alpha_1},\frac2\beta\Big\}\Big)^{-1}\frac{1}{c^2}\Big(\frac{\mu}{2}+\frac{1}{4\hat{c}}\Big)^{-1},\notag\\
&\qquad\frac{\hat{c}-2}{2}\Big(\max\Big\{\frac{2\gamma^2+1}{\alpha_1},\frac2\beta\Big\}\Big)^{-1}\frac{1}{c^2}\Big(\frac{\rho}{2}+\frac{1}{4\hat{c}}\Big)^{-1}\Big\}.
\end{align}
\end{lemma}
\begin{proof}
The idea of this proof comes from  \cite{LT2023}. Recalling
\eqref{271513}, \eqref{261438}, \eqref{6-9},  using \eqref{231213}
and Young's inequality, one has {\small\begin{align}\label{8163}
&N''(t)=\rho\|v_t\|_2^2+\mu\|p_{t}\|_2^2-(\alpha_1\|\nabla v\|_2^2+\beta\|\gamma\nabla v-\nabla p\|_2^2)-\int_\Omega v_tvdx-\int_\Omega p_tpdx+\|v\|_{n_1+1}^{n_1+1}+\|p\|_{n_2+1}^{n_2+1}\notag\\
&\ge \frac{\hat{c}+2\rho}{2}\|v_t\|_2^2+\frac{\hat{c}+2\mu}{2}\|p_{t}\|_2^2-\int_\Omega v_tvdx-\int_\Omega p_tpdx+\Big(\frac{\hat{c}}{2}-1\Big)(\alpha_1\|\nabla v\|_2^2+\beta\|\gamma\nabla v-\nabla p\|_2^2)-\hat{c}\mathcal{E}(t)\notag\\
&\ge \frac{\hat{c}+2\rho}{2}\|v_t\|_2^2+\frac{\hat{c}+2\mu}{2}\|p_{t}\|_2^2+\frac{\hat{c}-2}{2}\Big(\max\Big\{\frac{2\gamma^2+1}{\alpha_1},\frac2\beta\Big\}\Big)^{-1}(\|\nabla v\|_2^2+\|\nabla p\|_2^2)\notag\\
&\quad-\frac{\bar{C}}{4\hat{c}}\|v\|_2^2-\frac
{\hat{c}}{\bar{C}}\|v_t\|_2^2-\frac{\bar{C}}{4\hat{c}}\|p\|_2^2-\frac
{\hat{c}}{\bar{C}}\|p_t\|_2^2-\hat{c}\mathcal{E}(t)
\notag\\
&\ge \frac{\hat{c}+2\rho}{2}\|v_t\|_2^2+\frac{\hat{c}+2\mu}{2}\|p_{t}\|_2^2+\Big[\frac{\hat{c}-2}{2}\Big(\max\Big\{\frac{2\gamma^2+1}{\alpha_1},\frac2\beta\Big\}\Big)^{-1}\frac{1}{c^2}-\frac{\bar{C}}{4\hat{c}}\Big](\|v\|_2^2+\|p\|_2^2)\notag\\
&\quad-\frac {\hat{c}}{\bar{C}}\|v_t\|_2^2-\frac
{\hat{c}}{\bar{C}}\|p_t\|_2^2-\hat{c}\mathcal{E}(t).
\end{align}}
Here we have used \eqref{12161346} and \eqref{2.1}.

Define
\[\mathcal{F}(t)=N'(t)-\frac{\hat{c}}{\bar{C}}\mathcal{E}(t)=\rho\int_\Omega v(t)v_t(t)dx+\mu\int_\Omega p(t)p_t(t)dx-\frac{\hat{c}}{\bar{C}}\mathcal{E}(t).\]
Combining \eqref{8163} with \eqref{231219} gives that
\begin{align}
\label{8164}
\frac{d}{dt}\mathcal{F}(t)&\ge\frac{\hat{c}+2\rho}{2}\|v_t(t)\|_2^2+\frac{\hat{c}+2\mu}{2}\|p_{t}(t)\|_2^2\notag\\
&\quad+\Big[\frac{\hat{c}-2}{2}\Big(\max\Big\{\frac{2\gamma^2+1}{\alpha_1},\frac2\beta\Big\}\Big)^{-1}\frac{1}{c^2}-\frac{\bar{C}}{4\hat{c}}\Big](\|v\|_2^2+\|p\|_2^2)-\hat{c}\mathcal{E}(t)\notag\\
&\ge
\bar{C}\Big(\frac\rho2\|v_t\|_2^2+\frac\rho2\|v\|_2^2+\frac\mu2\|p_t\|_2^2+\frac\mu2\|p\|_2^2-\frac
{\hat{c}}{\bar{C}}\mathcal{E}(t)\Big)\ge {\bar{C}} \mathcal{F}(t)
\end{align}
according to \eqref{8168}. Condition \eqref{8165} admits
\[\mathcal{F}(0)=\rho\int_\Omega v_0v_1dx+\mu\int_\Omega p_0p_1dx-\frac{\hat{c}}{\bar{C}}\mathcal{E}(0)>0.\]
Exploiting Gronwall's inequality for \eqref{8164}, we obviously
obtain
\[\mathcal{F}(t)\ge e^{\bar{C}t}\mathcal{F}(0)>0,\]
that is \eqref{8167}.
\end{proof}

\begin{proofth210}
Let us first prove the finite time blow-up for high initial energy
by contradiction. Assume that the solution $(v,p)$ for problem
\eqref{1.1} is global.

One side,  H\"{o}lder's inequality and \eqref{231218}
directly give  that
\begin{align}\label{8169}
    &\|v(t)\|_2=\Big\|v_0+\int_0^tv_t(s) ds\Big\|_2\leq \|v_0\|_2+\int_0^t\|v_t(t)\|_2 ds\notag\\
    &\quad\leq\|v_0\|_2+\sqrt{t}\Big(\int_0^t\|v_t(s)\|_2^2 ds\Big)^{\frac 12}\leq\|v_0\|_2+\sqrt{t}(\mathcal{E}(0)-\mathcal{E}(t))^{\frac 12},
\end{align}
for all $t\in [0,\infty)$.  Since $(v,p)$ is a global solution of
problem \eqref{1.1}, one obtains $\mathcal{E}(t)\geq 0$ for all
$t\in [0,\infty)$. Otherwise, there exists $t_0\in [0,\infty)$ such
that $\mathcal{E}(t_0)<0$. Let us choose
$(v(t_0),p(t_0))$  as the new initial data, then $(v,p)$ blows up
in finite time by Theorem \ref{thmb1}, which is a contradiction.
Further, \eqref{8169} can be rewritten as
\begin{equation*}\label{81610}
    \|v\|_2\leq\|v_0\|_2+\sqrt{t}(\mathcal{E}(0))^{\frac 12}
\end{equation*}
for all $t\in [0,\infty)$. Likewise, $
\|p\|_2\leq\|p_0\|_2+\sqrt{t}(\mathcal{E}(0))^{\frac 12}$. Hence,
one gives
\begin{equation}
\label{271600}
\|v\|_2+\|p\|_2\leq\|v_0(x)\|_2+\|p_0\|_2+2\sqrt{t}(\mathcal{E}(0))^{\frac
12}.
\end{equation}

On the other side, let us recall \eqref{8167}, then
\begin{align}
\label{81611}
\frac{d}{dt}(\rho\|v\|^2_2+\mu\|p\|^2_2)&=2\Big(\rho\int_\Omega vv_tdx+\mu\int_\Omega pp_tdx\Big)\notag\\
&\ge
2\mathcal{F}(0)e^{\bar{C}t}+\frac{2\hat{c}}{\mathcal{\bar{C}}}\mathcal{E}(t)\ge
2\mathcal{F}(0)e^{\bar{C}t}>0.
\end{align}
Integrating \eqref{81611} from $0$ to $t$ arrives at
\begin{align*}
 & \|v\|_2^2+\|p\|_2^2=\|v_0\|_2^2+\|p_0\|_2^2+2\int_0^t\Big(\rho\int_\Omega v(s)v_t(s)dx+\mu\int_\Omega p(s)p_t(s)dx\Big)ds\notag\\
  &\geq\|v_0\|_2^2+\|p_0\|_2^2+2\int_0^te^{\bar{C}\tau}\mathcal{F}(0)d\tau
  =\|v_0\|_2^2+\|p_0\|_2^2+\frac{2}{\bar{C}}(e^{\bar{C}t}-1)\mathcal{F}(0),
\end{align*}
 which contradicts \eqref{271600} for $t$ sufficiently large. Therefore, the solution $(v,p)$ for problem \eqref{1.1} blows up in finite time.

 One remains to discuss the upper bound of the blow-up time.  Without loss of generality, we may assume $\mathcal{E}(t)\ge 0$ for $0<t<T^*$. For any $T^*\in (0,T_{max})$, define the auxiliary function $M(t)$ by
\begin{align*}
M(t)=\rho\|v\|_2^2+\mu\|p\|_2^2+\int_0^t(\|v\|_2^2+\|p\|_2^2)d\tau+(T_{max}-t)(\|v_0\|_2^2+\|p_0\|_2^2)+\kappa(t+\tau)^2
\end{align*}
for $t\in[0,T^*].$

It follows from  a direct computation that
 \begin{equation*}
\begin{split}
M'(t)&=2\Big(\rho\int_\Omega vv_tdx+\mu\int_\Omega pp_tdx\Big)+(\|v\|_2^2+\|p\|_2^2)-(\|v_0\|_2^2+\|p_0\|_2^2)+2\kappa(t+\sigma)\\
&=2N'(t)+2\int_0^t\int_\Omega
(v(s)v_t(s)dx+p(s)p_t(s))dxds+2\kappa(t+\tau)\quad\text{for
}t\in[0,T^*].
\end{split}
\end{equation*}
Recall \eqref{6-9}, then
 \begin{equation*}
\begin{split}
M''(t)&=2(\rho\|v_t\|_2^2+\mu\|p_{t}\|_2^2)-2(\alpha_1\|\nabla
v\|_2^2+\beta\|\gamma\nabla v-\nabla
p\|_2^2)+2(\|v\|_{n_1+1}^{n_1+1}+\|p\|_{n_2+1}^{n_2+1})+2\kappa
\end{split}
\end{equation*}
for $t\in[0,T^*]$.
 Therefore,
 \begin{align}
\label{3.7}
&M(t)M''(t)-\frac{\hat{c}+2}{4}(M'(t))^2\notag\\
&\quad=M(t)M''(t)-\frac{\hat{c}+2}{4}\Big[2N'(t)+2\int_0^t\int_\Omega (v(s)v_t(s)dx+p(s)p_t(s))dxds+2\kappa(t+\tau)\Big]^2\notag\\
&\quad=M(t)M''(t)-(\hat{c}+2)M(t)\Big(\rho\|v_t\|_2^2+\mu\|p_{t}\|_2^2+\int_0^t(\|v_t\|_2^2+\|p_{t}\|_2^2)d\tau+\kappa\Big)+(\hat{c}+2)\xi(t)\notag\\
&\quad\quad+(\hat{c}+2)(T^*-t)(\|v_0\|_2^2+\|p_0\|_2^2)\Big(\rho\|v_t\|_2^2+\mu\|p_{t}\|_2^2+\int_0^t(\|v_t\|_2^2+\|p_{t}\|_2^2)d\tau+\kappa\Big)\notag\\
&\quad \ge M(t)\eta(t)\quad\text{for }t\in[0,T^*],
\end{align}
 where
 \begin{equation*}
\begin{split}
\eta(t)&=M''(t)-(\hat{c}+2)\Big(\rho\|v_t\|_2^2+\mu\|p_{t}\|_2^2+\int_0^t(\|v_t\|_2^2+\|p_{t}\|_2^2)d\tau+\kappa\Big)\\
&=-\hat{c}(\rho\|v_t\|_2^2+\mu\|p_{t}\|_2^2)-2(\alpha_1\|\nabla v\|_2^2+\beta\|\gamma\nabla v-\nabla p\|_2^2)+2(\|v\|_{n_1+1}^{n_1+1}+\|p\|_{n_2+1}^{n_2+1})\\
&\quad-(\hat{c}+2)\int_0^t(\|v_t\|_2^2+\|p_{t}\|_2^2)d\tau-\hat{c}\kappa,
\end{split}
\end{equation*}
and we have applied Cauchy-Schwarz inequality and Young's inequality
to obtain
 {\small\begin{align*}
&\xi(t):=\Big[\rho\|v\|_2^2+\mu\|p\|_2^2+\int_0^t(\|v\|_2^2+\|p\|_2^2)d\tau+\kappa(t+\tau)^2\Big]\\
&\quad\times
\Big[\rho\|v_t\|_2^2+\mu\|p_{t}\|_2^2+\int_0^t(\|v_t\|_2^2+\|p_{t}\|_2^2)d\tau+\kappa\Big]\\
&\quad-\Big[\rho\int_\Omega vv_tdx+\mu\int_\Omega pp_tdx+\int_0^t\int_\Omega (v(s)v_t(s)dx+p(s)p_t(s))dxds+\kappa(t+\tau)\Big]^2\\
&\ge
\Big[\rho\|v\|_2^2+\mu\|p\|_2^2+\int_0^t(\|v\|_2^2+\|p\|_2^2)d\tau+\kappa(t+\tau)^2\Big]
\Big[\rho\|v_t\|_2^2+\mu\|p_{t}\|_2^2+\int_0^t(\|v_t\|_2^2+\|p_{t}\|_2^2)d\tau+\kappa\Big]\\
&\quad-\Big[\sqrt{\rho}\|v\|_2\sqrt{\rho}\|v_t\|_2+\sqrt{\mu}\|p\|_2\sqrt{\mu}\|p_t\|_2\\
&\qquad+\Big(\int_0^t\|v\|_2^2ds\Big)^{\frac12}\Big(\int_0^t\|v_t\|_2^2ds\Big)^{\frac12}+\Big(\int_0^t\|p\|_2^2ds\Big)^{\frac12}\Big(\int_0^t\|p_t\|_2^2ds\Big)^{\frac12}+\kappa(t+\tau)\Big]^2\\
&\ge[\rho\|v\|_2^2\mu\|p_{t}\|_2^2+\mu\|p\|_2^2\rho\|v_t\|_2^2-2\rho\|v\|_2\|v_t\|_2\mu\|p\|_2\|p_t\|_2]\\
&+\Big[\rho\|v\|_2^2\int_0^t\|v_t(s)\|_2^2ds+\rho\|v_t\|_2^2\int_0^t\|v(s)\|_2^2ds-2\rho\|v\|_2^2\Big(\int_0^t\|v_t(s)\|_2^2ds\Big)^{\frac12}\|v_t\|_2^2\Big(\int_0^t\|v(s)\|_2^2ds\Big)^{\frac12}\Big]\\
&+\Big[\rho\|v\|_2^2\int_0^t\|p_t(s)\|_2^2ds+\rho\|v_t\|_2^2\int_0^t\|p(s)\|_2^2ds-2\rho\|v\|_2^2\Big(\int_0^t\|p_t(s)\|_2^2ds\Big)^{\frac12}\|v_t\|_2^2\Big(\int_0^t\|p(s)\|_2^2ds\Big)^{\frac12}\Big]\\
&+[\kappa\rho\|v\|_2^2+\rho\|v_t\|_2^2\kappa(t+\sigma)^2-2\rho\|v\|_2\|v_t\|_2^2\kappa(t+\tau)]\\
&+\Big[\mu\|p\|_2^2\int_0^t\|v_t(s)\|_2^2ds+\mu\|p_t\|_2^2\int_0^t\|v(s)\|_2^2ds-2\mu\|p\|_2^2\Big(\int_0^t\|v_t(s)\|_2^2ds\Big)^{\frac12}\|p_t\|_2^2\Big(\int_0^t\|v(s)\|_2^2ds\Big)^{\frac12}\Big]\\
&+\Big[\mu\|p\|_2^2\int_0^t\|p_t(s)\|_2^2ds+\mu\|p_t\|_2^2\int_0^t\|p(s)\|_2^2ds-2\mu\|p\|_2^2\Big(\int_0^t\|p_t(s)\|_2^2ds\Big)^{\frac12}\|p_t\|_2^2\Big(\int_0^t\|p(s)\|_2^2ds\Big)^{\frac12}\Big]\\
&+[\kappa\mu\|p\|_2^2+\mu\|p_t\|_2^2\kappa(t+\sigma)^2-2\mu\|p\|_2\|p_t\|_2^2\kappa(t+\tau)]\\
&+\Big[\int_0^t\|v(s)\|_2^2ds\int_0^t\|p_t(s)\|_2^2ds+\int_0^t\|v_t(s)\|_2^2ds\int_0^t\|p(s)\|_2^2ds\\
&\quad-2\Big(\int_0^t\|v(s)\|_2^2ds\int_0^t\|p_t(s)\|_2^2ds\int_0^t\|v_t(s)\|_2^2ds\int_0^t\|p(s)\|_2^2ds\Big)^{\frac12}\Big]\\
&+\Big[\kappa\int_0^t\|p(s)\|_2^2ds+\kappa(t+\tau)^2\int_0^t\|p_t(s)\|_2^2ds-2\kappa(t+\tau)\Big(\int_0^t\|p(s)\|_2^2ds\int_0^t\|p_t(s)\|_2^2ds\Big)^{\frac12}\Big]\\
&+\Big[\kappa\int_0^t\|v(s)\|_2^2ds+\kappa(t+\tau)^2\int_0^t\|v_t(s)\|_2^2ds-2\kappa(t+\tau)\Big(\int_0^t\|v(s)\|_2^2ds\int_0^t\|v_t(s)\|_2^2ds\Big)^{\frac12}\Big]\ge
0.
\end{align*}}
It follows from  \eqref{231218}, \eqref{231213},\eqref{2.1} and
\eqref{12161346} and $\|v_0\|_2^2+\|p_0\|_2^2\le
\|v\|_2^2+\|p\|_2^2$ implied by  \eqref{81611}, then
 \begin{align}
\label{81615}
&\eta(t)\ge -\hat{c} \Big[2\mathcal{E}(t)-(\alpha_1\|\nabla v\|_2^2+\beta\|\gamma\nabla v-\nabla p\|_2^2)+\frac{2}{n_1+1}\|v\|_{n_1+1}^{n_1+1}+\frac{2}{n_2+1}\|p\|_{n_2+1}^{n_2+1}\Big]-\hat{c}\kappa\notag\\
&\quad-2(\alpha_1\|\nabla v\|_2^2+\beta\|\gamma\nabla v-\nabla p\|_2^2)+2(\|v\|_{n_1+1}^{n_1+1}+\|p\|_{n_2+1}^{n_2+1})-(\hat{c}+2)\int_0^t(\|v_t\|_2^2+\|p_{t}\|_2^2)d\tau\notag\\
&\ge -2\hat{c}\mathcal{E}(0)+(\hat{c}-2)(\alpha_1\|\nabla v\|_2^2+\beta\|\gamma\nabla v-\nabla p\|_2^2)+(\hat{c}-2)\int_0^t(\|v_t\|_2^2+\|p_{t}\|_2^2)d\tau-\hat{c}\kappa\notag\\
&\ge
-2\hat{c}\mathcal{E}(0)+(\hat{c}-2)\max\Big\{\frac{2\gamma^2+1}{\alpha_1},\frac2\beta\Big\}^{-1}c^2\min\Big\{\frac1\rho,\frac1\mu\Big\}(\|v_0\|_2^2+\|p_0\|_2^2)-\hat{c}\kappa=0.
\end{align}
 Here we have choosen
\begin{equation}
\label{81616} \kappa=
\frac{1}{\hat{c}}\Big\{-2\hat{c}\mathcal{E}(0)+(\hat{c}-2)\max\Big\{\frac{2\gamma^2+1}{\alpha_1},\frac2\beta\Big\}^{-1}c^2\min\Big\{\frac1\rho,\frac1\mu\Big\}(\|v_0\|_2^2+\|p_0\|_2^2)\Big\}>0.
\end{equation}
Therefore, let us combine \eqref{3.7} with \eqref{81615}, then
\begin{align*}
&M(t)M''(t)-\frac{\hat{c}+2}{4}(M'(t))^2\ge 0\quad\text{for
}t\in[0,T^*].
\end{align*}

Notice that
\[M(0)=\rho\|v_0\|_2^2+\mu\|p_0\|_2^2+T_{max}(\|v_0\|_2^2+\|p_0\|_2^2)+\kappa\tau^2>0.\]
Let us  choose
\begin{equation}
\label{8160}
\tau>\max\Big\{0,\frac{2(\|v_0\|_2^2+\|p_0\|_2^2)-(\hat{c}-2)\Big(\rho\int_\Omega
v_0v_1dx+\mu\int_\Omega p_0p_1dx\Big)}{(\hat{c}-2)\kappa}\Big\}
\end{equation}
to assure
\[M'(0)=2\Big(\rho\int_\Omega v_0v_1dx+\mu\int_\Omega p_0p_1dx\Big)+2\kappa\tau>\frac{4(\|v_0\|_2^2+\|p_0\|_2^2)}{\hat{c}-2}>0,\]
then Lemma \ref{lem2.5} yields that $M(t)\to\infty$ as $t\to T^*$
with
\[T^*\le \frac{4M(0)}{(\hat{c}-2)M'(0)}=\frac{2(\rho\|v_0\|_2^2+\mu\|p_0\|_2^2+T_{max}(\|v_0\|_2^2+\|p_0\|_2^2)+\kappa\tau^2)}{(\hat{c}-2)\Big[\Big(\rho\int_\Omega v_0v_1dx+\mu\int_\Omega p_0p_1dx\Big)+\kappa\tau\Big]}.\]
It follows from \eqref{8160} that the upper bound of $T_{max}$ is
derived.
\end{proofth210}

\section*{Acknowledgments}
The authors are  highly grateful to the  referee for their valuable
comments.

\def\cprime{$'$}

\end{document}